%% file: iHOSVD_JCM_final.tex
\documentclass[final]{jcmlatex}
\def\JCMvol{xx}
\def\JCMno{x}
\def\JCMyear{200x}
\def\JCMreceived{xxx}
\def\JCMrevised{xxx}
\def\JCMaccepted{xxx}

\usepackage{amsmath,amssymb,amsfonts}
\usepackage{algorithm,algorithmic}
\usepackage{caption}

\usepackage{subfigure}
\usepackage[usenames]{color}
\usepackage[final,notref,notcite]{showkeys}
\usepackage{url}
\input{macros.tex}

\begin{document}

\markboth{Y. Xu}{Fast algorithms for incomplete HOSVD}

\title{Fast algorithms for Higher-order Singular Value Decomposition \\ from incomplete data}

\author{Yangyang Xu
%\thanks{Institute of Mathematics and its Applications, University of Minnesota, Twin Cities, MN \\ Email: yangyang@ima.umn.edu}
\thanks{Department of Mathematics, University of Alabama, Tuscaloosa, AL \\ Email: yangyang.xu@ua.edu}
}

\maketitle

\begin{abstract}
Higher-order singular value decomposition (HOSVD) is an efficient way for data reduction and also eliciting intrinsic structure of multi-dimensional array data. It has been used in many applications, and some of them involve incomplete data. To obtain HOSVD of the data with missing values, one can first impute the missing entries through a certain tensor completion method and then perform HOSVD to the reconstructed data. However, the two-step procedure can be inefficient and does not make reliable decomposition.

In this paper, we formulate an incomplete HOSVD problem and combine the two steps into solving a single optimization problem, which simultaneously achieves imputation of missing values and also tensor decomposition. We also present one algorithm for solving the problem based on block coordinate update (BCU). Global convergence of the algorithm is shown under mild assumptions and implies that of the popular higher-order orthogonality iteration (HOOI) method, and thus we, for the first time, give global convergence of HOOI.

In addition, we compare the proposed method to state-of-the-art ones for solving incomplete HOSVD and also low-rank tensor completion problems and demonstrate the superior performance of our method over other compared ones. Furthermore, we apply it to face recognition and MRI image reconstruction to show its practical performance.
\end{abstract}



\section{Introduction}
Multi-dimensional arrays (or called \emph{tensor}s) appear in many applications that collect data along multiple dimensions, including space, time, and spectrum, from different subjects (e.g., patients), and under different conditions (e.g., view points, illuminations, expressions). Higher-order singular value decomposition (HOSVD) \cite{de2000multilinear} is an efficient way for dimensionality reduction and eliciting the intrinsic structure of the multi-dimensional data. It generalizes the matrix SVD and decomposes a multi-dimensional array into the product of a core tensor and a few orthogonal matrices, each of which captures the subspace information corresponding to one dimension (also called \emph{mode} or \emph{way}). The decomposition can be used for classification tasks including face recognition \cite{vasilescu2002multilinear}, handwritten digit classification \cite{savas2007handwritten}, human motion analysis and recognition \cite{vasilescu2002human}, and so on. HOSVD can also be applied to predicting unknown values while the acquired data is incomplete such as seismic data reconstruction \cite{kreimer2012tensor} %higher-order web link analysis \cite{kolda2005higher}, 
and personalized web search \cite{sun2005cubesvd}. On imputing missing values, data fitting is the main goal instead of decomposition itself. However, there are applications that involve missing values and also require the decomposition such as face recognition \cite{geng2011face}, facial age estimation \cite{geng2009facial}, and DNA microarray data analysis \cite{omberg2007tensor}.

In this paper, we aim at finding an approximate HOSVD of a given multi-dimensional array with missing values. More precisely, given partial entries of a tensor $\bm{\cM}\in\RR^{m_1\times\ldots\times m_N}$, we estimate its HOSVD as $\bm{\cC}\times_1\vA_1\ldots\times_N\vA_N$ such that the product is close to the underlying tensor $\bm{\cM}$ and $\vA_n$ can capture dominant subspace of the $n$-th mode of $\bm{\cM}$ for all $n$, where $\bm{\cC}\in\RR^{r_1\times\ldots\times r_N}$ is a core tensor, $\vA_n\in\RR^{m_n\times r_n}$ has orthonormal columns for all $n$, and $\times_n$ denotes the mode-$n$ tensor-matrix multiplication (see \eqref{eq:tm} below). To achieve the goal, we propose to solve the following \emph{incomplete HOSVD} problem:
\begin{equation}\label{eq:main1}
\begin{array}{ll}
\underset{\bm{\mcc},\mbfa}{\min}\,\frac{1}{2}\|\mcp_\Omega(\bm{\mcc}\times_1\mbfa_1\ldots\times_N\mbfa_N-\bm{\mcm})\|_F^2,\\[0.1cm]
\st \vA_n^\top\vA_n=\vI,\, \vA_n\in\RR^{m_n\times r_n},\, n=1,\ldots,N,
\end{array}
\end{equation}
where $\mbfa=(\mbfa_1,\ldots,\mbfa_N)$, $(r_1,\ldots,r_N)$ is a given multilinear rank, $\vI$ is the identity matrix of appropriate size, $\Omega$ indexes the observed entries, and $\cP_\Omega$ is a projection that keeps the entries in $\Omega$ and zeros out others. 
Since only partial entries of $\bm{\cM}$ are assumed known in \eqref{eq:main1}, we cannot have $r_n=m_n,\, n=1,\ldots,N$, because otherwise, it will cause overfitting problem. Hence, in general, instead of a full HOSVD, we can only get a truncated HOSVD  of $\bm{\cM}$ from its partial entries. 

To get an approximate HOSVD of $\bm{\cM}$ from its partial entries, one can also first fill in the unobserved entries through a certain tensor completion method and then perform some iterative method to have a (truncated) HOSVD of the estimated tensor. The advantage of our method is that it combines the two steps into solving just one problem and is usually more efficient and accurate. In addition, upon solving \eqref{eq:main1}, we can %immediately get an approximate HOSVD of $\bm{\cM}$ but 
also estimate the unobserved entries of $\bm{\cM}$ from $\bm{\mcc}\times_1\mbfa_1\ldots\times_N\mbfa_N$ and thus achieve the tensor completion as a byproduct. 

We will write \eqref{eq:main1} into one equivalent problem and solve it by the block coordinate descent (BCD) method. Although the problem is non-convex, we will demonstrate that \eqref{eq:main1} solved by the simple BCD can perform better than state-of-the-art tensor completion methods on reconstructing (approximate) low-multilinear-rank tensors. In addition, it can produce more reliable factors and as a result give higher prediction accuracies on certain classification problems such as the face recognition problem.

\subsection{Related work} We first review methods for matrix and tensor factorization with missing values and then existing works on low-rank tensor completion (LRTC). 

\subsubsection*{Matrix and tensor factorization with missing values} The matrix SVD from incomplete data has been studied for decades; see \cite{ruhe1974numerical, kurucz2007methods} for example. It can be regarded as a special case of \eqref{eq:main1} by setting $N=2$ and restricting $\bm{\cC}$ to be a nonnegative diagonal matrix. Further removing the orthogonality constraint on $\vA_n$'s and setting $\bm{\cC}$ as the identity matrix, \eqref{eq:main1} reduces to the matrix factorization from incomplete data (e.g., see \cite{gabriel1979lower}). Existing methods for achieving matrix SVD or factorization with missing values are mainly BCD-type ones such as the expectation maximization (EM) in \cite{srebro2003weighted} that alternates between imputation of the missing values and SVD computation of the most recently estimated matrix, and the successive over-relaxation (SOR) in \cite{wen2012lmafit} that iteratively updates the missing values and the basis and coefficient factors by alternating least squares with extrapolation. There are also approaches %for matrix factorization with missing values 
that updates all variables simultaneously at each iteration. For example, \cite{buchanan2005damped} employs the damped Newton method for matrix factorization with missing values. Usually, the damped Newton method converges faster than BCD-type ones but has much higher per-iteration complexity.

Tensor factorization from incomplete data has also been studied for many years such as the CANDECOMP / PARAFAC (CP) tensor factorization with missing values in \cite{acar2010scalable}, the weighted nonnegative CP tensor factorization in \cite{paatero1997weighted}, the weighted Tucker decomposition in \cite{sorber2013structured, filipovic2013tucker}, the nonnegative Tucker decomposition with missing data in \cite{morup2008algorithms}, and the recently proposed Bayesian CP factorization of incomplete tensors in \cite{zhao2014bayesian}. BCD-type method has also been employed for solving tensor factorization with missing values. For example, \cite{walczak2001dealing} use the EM method for Tucker decomposition from incomplete data, \cite{filipovic2013tucker} uses block coordinate minimization for the weighted Tucker decomposition, and \cite{morup2008algorithms} and \cite{xu2014NTD} solve the nonnegative Tucker decomposition with missing data by the multiplicative updating and alternating prox-linearization method, respectively. There are also non-BCD-type methods such as the Gauss-Newton method in \cite{tomasi2005parafac} and nonlinear conjugate gradient method for CP factorization with missing values, and the damped Newton method in \cite{phan2011damped} for nonnegative Tucker decomposition.  

With all entries observed, \eqref{eq:main1} becomes the best rank-$(r_1,\ldots,r_N)$ approximation problem in \cite{de2000best}, where a higher-order orthogonality iteration (HOOI) method is given. Although HOOI often works well, no convergence result has been shown in the literature except that it makes the objective value nonincreasing at the iterates. As a special case of our algorithm (see Algorithm \ref{alg:ihooi} below), we will give its global convergence in terms of a first-order optimality condition.

Although not explicitly formulated, \eqref{eq:main1} has been used in \cite{geng2011face} for face recognition with incomplete training data and in \cite{geng2009facial} for facial age estimation. Both works achieve the incomplete HOSVD by the EM method which alternates between the imputation of missing values and performing HOOI to the most recently estimated tensor. Our algorithm (see Algorithm \ref{alg:ihooi} below) is similar to the EM method. However, our method performs just one HOOI iteration within each cycle of updating factor matrices instead of running HOOI for many iterations as done in \cite{geng2009facial, geng2011face}, and thus our method has much cheaper per-cycle complexity and faster overall convergence. Another closely related work is \cite{chen2013simultaneous}, which proposes the simultaneous tensor decomposition and completion (STDC) without orthogonality on factor matrices. It makes the core tensor of the same size as the original tensor and square factor matrices. In addition, it models STDC by nuclear norm regularized minimization, which can be much more expensive than \eqref{eq:main1} to solve. 

\subsubsection*{Low-rank tensor completion} Like \eqref{eq:main1}, all other tensor factorization with missing values can also be used to estimate the unobserved entries of the underlying tensor. When the tensor has some low-rankness property, the estimation can be highly accurate. For example, \cite{filipovic2013tucker} applies the Tucker factorization with missing data to LRTC and can have fitting error low to the order of $10^{-5}$ for randomly generated low-multilinear-rank tensors. The work \cite{liufactor} employs CP factorization with missing data and also demonstrates that low-multilinear-rank tensors can be reconstructed into high accuracy. 

Many other methods for LRTC directly impute the missing entries such that the reconstructed tensor has low-rankness property. The pioneering work \cite{liu2013tensor} proposes to minimize the weighted nuclear norm of all mode matricization (see the definition in section \ref{sec:notation}) of the estimated tensor. Various methods are applied in \cite{liu2013tensor} to the weighted nuclear norm minimization including BCD, the proximal gradient, and the alternating direction method of multiplier (ADMM). The same idea is employed in \cite{gandy2011tensor} to general low-multilinear-rank tensor recovery. Recently, \cite{romera2013new} uses, as a regularization term, a tight convex relaxation of the average of the ranks of all mode matricization and applies ADMM to solve the proposed model. The work \cite{mu2013square} proposes to reshape the underlying tensor into a more ``squared'' matrix and then minimize the nuclear norm of the reshaped matrix. It is theoretically shown and also numerically demonstrated in \cite{mu2013square} that if the underlying tensor has at least \emph{four} modes, the ``squared'' method can perform better than the weighted nuclear norm minimization in \cite{liu2013tensor}. 
Besides convex optimization methods, nonconvex heuristics have also been proposed for LRTC. For example, \cite{kressner2013low} explicitly constrains the solution in a low-multilinear-rank manifold and employs the Riemannian conjugate gradient to solve the problem, and \cite{tmac2015} applies low-rank matrix factorization to each mode matricization of the underlying tensor and proposes a parallel matrix factorization model, which is then solved by alternating least squares method.

\subsection{Contributions} We summarize our contributions as follows.
\begin{itemize}
\renewcommand\labelitemi{--}
\item We give an explicit formulation of the incomplete HOSVD problem. Although the problem has appeared in many applications, it has never been explicitly formulated as an optimization problem\footnote{The incomplete Tucker decomposition in \cite{liu2013tensor, filipovic2013tucker} has no orthogonality constraint on the factor matrices and thus differs from our model, and \cite{geng2009facial, geng2011face} apply incomplete HOSVD without explicitly giving a formulation of their problems.}, and thus it is not clear that which objective existing methods are pursuing and whether they have convergence results. An explicit formulation uncovers the objective and enables analyzing the existing methods and also designing more efficient and reliable algorithms.  
\item We also present a novel algorithm for solving the incomplete HOSVD problem based on BCU method. Under some mild assumptions, global convergence of the algorithm is shown in terms of a first-order optimality condition. The convergence result implies, as a special case, that of the popular HOOI heuristic method \cite{de2000best} for finding the best rank-$(r_1,\ldots,r_N)$ approximation of a given tensor, and hence we, for the first time, give the global convergence of the HOOI method.
\item Numerical experiments are performed to test the ability of the proposed algorithm on recovering the factors of underlying tensors. Compared to the method in \cite{filipovic2013tucker} for solving Tucker factorization from incomplete data, our algorithm not only is more efficient but also can give more reliable factors. We also test the proposed method on the LRTC problem and demonstrate that our algorithm can outperform state-of-the-art methods in both running time and solution quality.
\item In addition, we apply our method to face recognition and MRI image reconstruction problems and demonstrate that it can perform well on both applications.
\end{itemize}

\subsection{Notation and preliminaries}\label{sec:notation} We use bold capital letters $\vX,\vY,\ldots$ for matrices and bold caligraphical letter $\bm{\cX},\bm{\cY},\ldots$ for tensors. $\vI$ is reserved for the identity matrix and $\vzero$ for zero matrix, and their dimensions are known from the context. For $n=1,\ldots,N$, we denote $\cO_n=\{\vA_n:\,\vA_n^\top\vA_n=\vI\}$ as the manifold of the $n$-th factor matrix. We use $\sigma_i(\vX)$ to denote the $i$-th largest singular value of $\vX$. By compact SVD of $\vX$, we always mean $\vX=\vU_x\bm{\Sigma}_x\vV_x^\top$ with $\bm{\Sigma}_x$ having all positive singular values of $\vX$ on its diagonal and $\vU_x$ and $\vV_x$ containing the corresponding left and right singular vectors. The $(i_1,\ldots,i_N)$-th component of an $N$-way tensor $\bm{\mcx}$ is denoted as $x_{i_1\ldots i_N}$. For $\bm{\mcx},\bm{\mcy}\in\mbr^{m_1\times\ldots\times m_N}$, their inner product is defined in the same way as that for matrices, i.e.,
$$\langle\bm{\mcx},\bm{\mcy}\rangle=\sum_{i_1=1}^{m_1}\cdots\sum_{i_N=1}^{m_N}x_{i_1\ldots i_N}\cdot y_{i_1\ldots i_N}.$$
The Frobenius norm of $\bm{\mcx}$ is defined as $\|\bm{\mcx}\|_F=\sqrt{\langle\bm{\mcx},\bm{\mcx}\rangle}.$

We review some basic concepts about tensor below. For more details, the readers are referred to \cite{kolda2009tensor}.

A \emph{fiber} of $\bm{\mcx}$ is a vector obtained by fixing all indices of $\bm{\mcx}$ except one, and a \emph{slice} of $\bm{\mcx}$ is a matrix by fixing all indices of $\bm{\mcx}$ except two. 
The \emph{vectorization} of $\bm{\cX}$ gives a vector, which is obtained by stacking all mode-1
fibers of $\bm{\cX}$ and denoted by $\vvec(\bm{\cX})$.
The mode-$n$ \emph{matricization} (also called \emph{unfolding}) of  $\bm{\mcx}$ is denoted as $\mbfx_{(n)}$ or $\unfold_n(\bm{\cX})$, which is a matrix with columns being the mode-$n$ fibers of $\bm{\mcx}$ in the lexicographical order, and we define $\fold_n$ to reverse the process, i.e., $\fold_n(\unfold_n(\bm{\mcx}))=\bm{\mcx}$.
The mode-$n$ product of $\bm{\mcx}\in\mbr^{m_1\times\cdots\times m_N}$ with $\mbfb\in\mbr^{p\times m_n}$ is written as $\bm{\mcx}\times_n\mbfb$ which gives a tensor in $\mbr^{n_1\times \cdots \times m_{n-1}\times p\times m_{n+1}\times \cdots\times m_N}$ and is defined component-wisely by
\begin{equation}\label{eq:tm}
(\bm{\mcx}\times_n \mbfb)_{i_1\cdots i_{n-1}ji_{n+1}\cdots i_N}=\sum_{i_n=1}^{m_n}x_{i_1i_2\cdots i_N}\cdot b_{ji_n}.
\end{equation} 

For any tensor $\bm{\cG}$ and matrices $\vX$ and $\vY$ of appropriate size, it holds
\begin{equation}\label{eq:ttm}
\bm{\cG}\times_n(\vX\vY)=(\bm{\cG}\times_n\vY)\times_n\vX, \,\forall n.
\end{equation}
If $\bm{\mcx}=\bm{\mcc}\times_{i=1}^N\mbfa_i:=\bm{\mcc}\times_1\mbfa_1\cdots\times_N\mbfa_N$, then 
\begin{equation}\label{eq:mat}\mbfx_{(n)}=\mbfa_n\mbfc_{(n)}\left(\otimes_{\substack{i=N\\i\neq n}}^1\mbfa_i\right)^\top,\,\forall n.
\end{equation}
and
\begin{equation}\label{eq:vec}
\vvec(\bm{\mcx})=\left(\otimes_{n=N}^1\mbfa_n\right)\vvec(\bm{\mcc}),
\end{equation}
where 
\begin{equation}\label{eq:kron}
\otimes_{n=N}^1\mbfa_n:=\mbfa_N\otimes\cdots\otimes\mbfa_1,
\end{equation}
 and $\mbfa\otimes\mbfb$ denotes Kronecker product of $\mbfa$ and $\mbfb$. 

%\begin{lemma}
%For any tensor $\bm{\mcc}$ and matrices $\mbfa_1,\ldots,\mbfa_N, \mbfb_1,\ldots,\mbfb_N$ with appropriate sizes, it holds that
%\begin{equation}\label{eq:tmcon}
%\bm{\mcc}\times_1\mbfa_1\ldots\times_N\mbfa_N\times_1\mbfb_1\ldots\times_N\mbfb_N=\bm{\mcc}\times_1(\mbfb_1\mbfa_1)\ldots\times_N(\mbfb_N\mbfa_N).
%\end{equation}
%\end{lemma}

For any matrices $\vA$, $\vB$, $\vC$ and $\vD$ of appropriate size, we have (c.f. \cite[Chapter 4]{horn1991topics})
\begin{subequations}\label{eq:kronpro}
\begin{align}
&\mbfa\otimes\mbfb\otimes\mbfc = (\mbfa\otimes\mbfb)\otimes\mbfc = \mbfa\otimes(\mbfb\otimes\mbfc), \label{eq:kron1}\\
&(\mbfa\otimes\mbfb)(\mbfc\otimes\mbfd) = (\mbfa\mbfc)\otimes(\mbfb\mbfd),\label{eq:kron2}\\
&(\mbfa\otimes\mbfb)^\top=\mbfa^\top\otimes\mbfb^\top, \label{eq:kron3}\\
&(\mbfa\otimes\mbfb)^\dagger=\mbfa^\dagger\otimes\mbfb^\dagger,\label{eq:kron4}
\end{align}
\end{subequations}
where $^\dagger$ denotes the Moore-Penrose pseudo-inverse.

\subsection{Outline} The rest of the paper is organized as follows. In section \ref{sec:alg}, we write \eqref{eq:main1} into one equivalent problem and present an algorithm based on BCU. Convergence analysis of the algorithm is given in section \ref{sec:convg}, and numerical results are presented in section \ref{sec:numerical}. Finally, section \ref{sec:conclusion} concludes the paper. 

\section{Algorithm}\label{sec:alg}
In this section, we write \eqref{eq:main1} into one equivalent problem and apply the BCU method to it. We choose BCU because of the problem's nice structure, which enables BCU to be more efficient than full coordinate update method; see \cite{CF2016}.  Convergence analysis of the algorithm will be given in next section.

%We present two algorithms because the first one is often more efficient while the second one can produce more reliable solutions. In addition, the second algorithm includes as a special case the HOOI method \cite{de2000best} which has been popularly used but still lacks convergence results. The convergence analysis of the first algorithm is relatively easier than that of the second one and will guide us to analyze the convergence of the second algorithm and thus give a convergence result of the HOOI method as a special case.
\subsection{Alternative formulation}

%\subsection{Alternating least squares method}
Introducing auxiliary variable $\bm{\mcx}$, we write \eqref{eq:main1} into the following equivalent problem
\begin{equation}\label{eq:main2}
\begin{array}{ll}
\underset{\bm{\mcc},\mbfa,\bm{\mcx}}{\min}\,f(\bm{\mcc},\mbfa,\bm{\mcx})\equiv\frac{1}{2}\|\bm{\mcc}\times_1\mbfa_1\ldots\times_N\mbfa_N-\bm{\mcx}\|_F^2,\\[0.1cm]
\st \vA_n^\top\vA_n=\vI,\, \vA_n\in\RR^{m_n\times r_n},\, n=1,\ldots,N,\, \mcp_\Omega(\bm{\mcx})=\mcp_\Omega(\bm{\mcm}).
\end{array}
\end{equation}
The equivalence between \eqref{eq:main1} and \eqref{eq:main2} can be easily seen by setting $\cP_{\Omega^c}(\bm{\cX})=\cP_{\Omega^c}(\bm{\mcc}\times_1\mbfa_1\ldots\times_N\mbfa_N)$ in \eqref{eq:main2}. This transformation is similar to those in \cite{wen2012lmafit, ling2012decentralized, admmNMF2012} for low-rank matrix factorization with missing values and also to the EM method in \cite{walczak2001dealing} for CP factorization from incomplete data.
The objective of \eqref{eq:main2} is block multi-convex, and one can apply the alternating least squares (ALS) method to it. The ALS method is also new for finding a solution to \eqref{eq:main1}. However, we will focus on another algorithm and present the ALS method in Appendix \ref{app:als}.
 
Note that with $\vA$ and $\bm{\cX}$ fixed in \eqref{eq:main2}, the optimal $\bm{\cC}$ is given by 
$$\bm{\cC}=\bm{\cX}\times_1\vA_1^\top\ldots\times_N\vA_N^\top,$$
and thus one can eliminate $\bm{\cC}$ by plugging in the above formula to \eqref{eq:main2} and have the following equivalent problem
\begin{equation}\label{eq:main3}
\begin{array}{ll}
\underset{\mbfa,\bm{\mcx}}{\min}\,g(\mbfa,\bm{\mcx})\equiv\frac{1}{2}\|\bm{\cX}\times_1\mbfa_1\vA_1^\top\ldots\times_N\mbfa_N\vA_N^\top-\bm{\mcx}\|_F^2,\\[0.1cm]
\st \vA_n^\top\vA_n=\vI,\, \vA_n\in\RR^{m_n\times r_n},\, n=1,\ldots,N,\, \mcp_\Omega(\bm{\mcx})=\mcp_\Omega(\bm{\mcm}).
\end{array}
\end{equation}
The transformation from \eqref{eq:main2} to \eqref{eq:main3} is similar to that employed by the HOOI method in \cite{de2000best} for finding the best rank-$(r_1,\ldots,r_N)$ of a given tensor.

\subsection{Incomplete HOOI}
As what is done in the HOOI method, we propose to alternatingly update $\vA_1,\ldots,\vA_N$ and $\bm{\cX}$ by minimizing $g$ with respect to one of them while the remaining variables are fixed, one at a time. Specifically, let $(\hat{\vA}_1,\ldots,\hat{\vA}_N, \hat{\bm{\cX}})$ be the current values of the variables and satisfy the feasibility constraints. We renew them to $(\tilde{\vA}_1,\ldots,\tilde{\vA}_N, \tilde{\bm{\cX}})$ through the following updates
\begin{subequations}\label{eq:update3}
\begin{align}
&\tilde{\vA}_n=\argmin_{\vA_n^\top\vA_n=\vI}g(\tilde{\vA}_{<n},\vA_n,\hat{\vA}_{>n},\hat{\bm{\cX}}),\, n=1,\ldots,N,\label{eq:update3-a}\\
&\tilde{\bm{\cX}}=\argmin_{\cP_\Omega(\bm{\cX})=\cP_\Omega(\bm{\cM})}g(\tilde{\vA},\bm{\cX}).\label{eq:update3-x}
\end{align}
\end{subequations}

\subsubsection*{$\vA$-subproblems} Note that if $\vA_n^\top\vA_n=\vI,\forall n$, then
\begin{equation}\label{eq}\|\bm{\cX}\times_1\mbfa_1\vA_1^\top\ldots\times_N\mbfa_N\vA_N^\top-\bm{\mcx}\|_F^2=\|\bm{\cX}\|_F^2-\|\bm{\cX}\times_1\vA_1^\top\ldots\times_N\vA_N^\top\|_F^2.
\end{equation}
Hence, from \eqref{eq:mat}, the update of $\vA_n$ in \eqref{eq:update3-a} can be written as
\begin{equation}\label{equpdate3-a}\tilde{\vA}_n=\argmax_{\vA_n^\top\vA_n=\vI}\|\vA_n^\top\hat{\vG}_n\|_F^2,\end{equation}
where $\hat{\vG}_n=\unfold_n(\hat{\bm{\cX}}\times_{i=1}^{n-1}\tilde{\vA}_i^\top\times_{i=n+1}^N\hat{\vA}_i^\top)$ .
Let $\vU_n$ be the matrix containing the left $r_n$ leading singular vectors of $\hat{\vG}_n$. Then $\tilde{\vA}_n=\vU_n$ solves \eqref{equpdate3-a}. Note that for any orthogonal matrix $\vQ_n$, $\vU_n\vQ_n$ is also a solution of \eqref{equpdate3-a}, and the observation will play an important role in the convergence analysis of the proposed algorithm in section \ref{sec:convg}.

\subsubsection*{$\bm{\cX}$-subproblem} The problem in \eqref{eq:update3-x} can be reduced to solving a normal equation. However, the equation can be extremely large and expensive to solve by a direct method or even an iterative solver for linear system. We propose to approximately solve \eqref{eq:update3-x} by the gradient descent method. Splitting $\bm{\cX}=\cP_{\Omega^c}(\bm{\cX})+\cP_\Omega(\bm{\cM})$ and using \eqref{eq},  we write \eqref{eq:update3-x} equivalently into
\begin{equation}\label{eq:rest}\min_{\bm{\cX}}h(\bm{\cX};\tilde{\vA})\equiv\frac{1}{2}\|\cP_{\Omega^c}(\bm{\cX})\|_F^2-\frac{1}{2}\|\cP_{\Omega^c}(\bm{\cX})\times_{i=1}^N\tilde{\vA}_i^\top+\cP_\Omega(\bm{\cM})\times_{i=1}^N\tilde{\vA}_i^\top\|_F^2.
\end{equation}
Since $\cP_\Omega(\hat{\bm{\cX}})=\cP_\Omega(\bm{\cM})$, it is not difficult to show (see Appendix \ref{app:pf-partx})
\begin{equation}\label{eq:grad-partx}\nabla_{\bm{\cX}}h(\hat{\bm{\cX}};\tilde{\vA})=\cP_{\Omega^c}(\hat{\bm{\cX}})-\cP_{\Omega^c}(\hat{\bm{\cX}}\times_{i=1}^N\tilde{\vA}_i\tilde{\vA}_i^\top).
\end{equation}
According to the next lemma and \cite[Theorem 3.1]{BeckTeboulle2009}, one can solve \eqref{eq:rest} or equivalently \eqref{eq:update3-x} by iteratively updating $\bm{\cX}$ through (starting from $\tilde{\bm{\cX}}=\hat{\bm{\cX}}$)
\begin{equation}\label{eq:grad}\tilde{\bm{\cX}}\leftarrow \cP_{\Omega^c}(\tilde{\bm{\cX}}\times_{i=1}^N\tilde{\vA}_i\tilde{\vA}_i^\top)+\cP_\Omega(\bm{\cM}).
\end{equation}

\begin{lemma}\label{lem:lip}
If $\vA_n^\top\vA_n=\vI,\,\forall n$, then $h(\bm{\cX};\vA)$ defined in \eqref{eq:rest} is convex with respect to $\bm{\cX}$, and $\nabla_{\bm{\cX}}h(\bm{\cX};\vA)$ is Lipschitz continuous with constant \emph{one}, i.e.,
$$\|\nabla_{\bm{\cX}}h(\hat{\bm{\cX}};\vA)-\nabla_{\bm{\cX}}h(\hat{\bm{\cY}};\vA)\|_F \le \|\hat{\bm{\cX}}-\hat{\bm{\cY}}\|_F,\,\forall \hat{\bm{\cX}},\hat{\bm{\cY}}.$$
\end{lemma}

Numerically, we observe that performing just \emph{one} update in \eqref{eq:grad} is enough to make sufficient decrease of the objective and the algorithm can converge surprisingly fast. Therefore, we perform only one update in \eqref{eq:grad} to $\bm{\cX}$, and that is exactly letting
\begin{equation}\label{eq:sol-x2}
\tilde{\bm{\cX}}=\argmin_{\cP_\Omega(\bm{\cX})=\cP_\Omega(\bm{\cM})}\|\bm{\cX}-\hat{\bm{\cX}}\times_{i=1}^N\tilde{\vA}_i\tilde{\vA}_i^\top\|_F^2.
\end{equation} The pseudocode of the resulting method for solving \eqref{eq:main3} is shown in Algorithm \ref{alg:ihooi}.

\begin{algorithm}\caption{Incomplete higher-order orthogonality iteration (iHOOI)}\label{alg:ihooi}
\begin{algorithmic}[1]
{\small
\STATE \textbf{Input:} index set $\Omega$, observed entries $\mcp_\Omega(\bm{\mcm})$, and initial point $(\mbfa^0,\bm{\mcx}^0)$ with $(\vA_i^0)^\top \vA_i^0=\vI,\forall i$ and $\mcp_\Omega(\bm{\mcx}^0)=\mcp_\Omega(\bm{\mcm})$.
\FOR{$k=0,1,\ldots$}
\FOR{$n = 1,\ldots,N$}
\STATE Let $\vA_n^{k+1}$ be the matrix containing the left $r_n$ leading singular vectors of $\vG_n^k$ where 
\begin{equation}\label{eq:gnk}
\vG_n^k=\unfold_n(\bm{\cX}^k\times_{i=1}^{n-1}(\vA_i^{k+1})^\top\times_{i=n+1}^N(\vA_i^k)^\top).
\end{equation}
\ENDFOR
\STATE Update the entries of $\bm{\cX}$ not in $\Omega$ by 
\begin{equation}\label{eq:ihooi-x}\cP_{\Omega^c}(\bm{\mcx}^{k+1})=\cP_{\Omega^c}\left(\bm{\mcx}^k\times_{i=1}^N\big(\vA_i^{k+1}(\vA_i^{k+1})^\top\big)\right).
\end{equation}
\IF{stopping criterion is satisfied}
\STATE Let $\bm{\cC}=\bm{\cX}^{k+1}\times_{i=1}^N (\vA_i^{k+1})^\top$ and return $(\bm{\mcc},\mbfa^{k+1},\bm{\mcx}^{k+1})$.
\ENDIF
\ENDFOR
}
\end{algorithmic}
\end{algorithm}

\begin{remark}[Comparison between Algorithms \ref{alg:ihooi} and \ref{alg:lrtfit}]
Since $\bm{\cC}$ is absorbed into the update of $\vA$ and $\bm{\cX}$ in Algorithm \ref{alg:ihooi}, we expect that it would perform no worse than Algorithm \ref{alg:lrtfit} for solving \eqref{eq:main1} in terms of convergence speed and solution quality and will demonstrate it in section \ref{sec:numerical} (e.g., see Tables \ref{table:facereg} and \ref{table:mri_rec}). However, notice that the update of $\vA_n$ in Algorithm \ref{alg:ihooi} typically requires SVD of $\vG_n^k$ and can be expensive if $m_n$ and $\Pi_{i\neq n}r_i$ are both large (e.g., see the test in secion \ref{sec:facerecog}).
\end{remark}

\begin{remark}[Differences between Algorithm \ref{alg:ihooi} and EM methods in the literature] Algorithm \ref{alg:ihooi} is very similar to some existing EM methods such as those in \cite{walczak2001dealing, geng2009facial, geng2011face}. The EM method in \cite{walczak2001dealing} is somehow a mixture of  Algorithm \ref{alg:ihooi} and Algorithm \ref{alg:lrtfit} in Appendix \ref{app:als}. Its $\vA$-updates are the same as those in Algorithm \ref{alg:ihooi} while its $\bm{\cX}$-update is the same as that in Algorithm \ref{alg:lrtfit}. The difference between Algorithm \ref{alg:ihooi} and the methods in \cite{geng2009facial, geng2011face} is that within each ``\textbf{for}'' loop, the latter methods perform $\vA$-updates iteratively to get a rank-$(r_1,\ldots,r_N)$ approximation of the estimated $\bm{\cX}$. 
\end{remark}

%\subsection{Complexity comparison between Algorithms \ref{alg:lrtfit} and \ref{alg:ihooi}} Since $\bm{\cC}$ is absorbed into the update of $\vA$ and $\bm{\cX}$ in Algorithm \ref{alg:ihooi}, we expect that it would perform no worse than Algorithm \ref{alg:lrtfit} for solving \eqref{eq:main1} in terms of convergence speed and solution quality and will demonstrate it in section \ref{sec:numerical}. In this subsection, we analyze the per-iteration complexity of both algorithms and understand when Algorithm \ref{alg:ihooi} can be overall more efficient than Algorithm \ref{alg:lrtfit}. We first consider balance case by assuming $m_n=m, r_n=r,\forall n$ and show that Algorithms \ref{alg:lrtfit} and \ref{alg:ihooi} have roughly the same per-iteration complexity. Then we consider unbalanced case, in which Algorithm \ref{alg:ihooi} can have more expensive update compared to that of Algorithm \ref{alg:lrtfit}.
%
%\subsubsection*{Balance case} Assume $m_n=m, r_n=r,\forall n$.

\subsection{Extension} One generalization to the incomplete HOSVD is to find the HOSVD of an underlying tensor $\bm{\cM}$ from its underdetermined measurements $\cL(\bm{\cM})$, where $\cL$ is a linear operator with adjoint $\cL^*$. For this scenario, one can consider to solve the problem 
\begin{equation}\label{eq:gen-prob}
\begin{array}{ll}
\underset{\bm{\mcx},\mbfa}{\min}\,\|\bm{\mcx}\times_1\mbfa_1\vA_1^\top\ldots\times_N\mbfa_N\vA_N^\top-\bm{\mcx}\|_F^2,\\[0.1cm]
\st \vA_n^\top\vA_n=\vI,\, \vA_n\in\RR^{m_n\times r_n},\, n=1,\ldots,N,\ \cL(\bm{\cX})=\cL(\bm{\cM}).
\end{array}
\end{equation}
A simple modification of Algorithm \ref{alg:ihooi} suffices to handle \eqref{eq:gen-prob} by using the same $\vA$-updates in \eqref{equpdate3-a} and changing the $\bm{\cX}$-update in \eqref{eq:sol-x2} to 
$$\tilde{\bm{\cX}}=\argmin_{\cL(\bm{\cX})=\cL(\bm{\cM})}\|\bm{\cX}-\hat{\bm{\cX}}\times_{i=1}^N\tilde{\vA}_i\tilde{\vA}_i^\top\|_F^2,$$
which is equivalent to letting
$$\tilde{\bm{\cX}}=\hat{\bm{\cX}}\times_{i=1}^N\tilde{\vA}_i\tilde{\vA}_i^\top+\cL^*(\cL\cL^*)^{-1}(\cL(\bm{\cM})-\cL(\hat{\bm{\cX}}\times_{i=1}^N\tilde{\vA}_i\tilde{\vA}_i^\top)).$$
From the above update, we see that to make the modified algorithm efficient, the evaluation of $\cL,\cL^*$ and $(\cL\cL^*)^{-1}$ needs be cheap such as $\cL=\cP_\Omega$ in the incomplete HOSVD and $\cL$ being a partial Fourier transformation considered in \cite{gandy2011tensor} for low-rank tensor recovery.

\section{Convergence Analysis}\label{sec:convg} In this section, we analyze the convergence of Algorithm \ref{alg:ihooi}. %Assuming boundedness on the imputation of missing entries of the underlying tensor, we give global convergence of Algorithm \ref{alg:lrtfit} in terms of a first-order optimality condition of \eqref{eq:main2}. 
We show its global convergence in terms of the first-order optimality condition of \eqref{eq:main2}. The main assumption we make is a condition (see \eqref{eq:gap-svd}) similar to that made by the orthogonal iteration method (c.f. \cite[section 7.3.2]{GolubVanLoan1996}) for computing $r$-dimensional dominant invariant subspace of a matrix.

\subsection{First-order optimality conditions}
The Lagrangian function of \eqref{eq:main3} is
$$\cL_g(\vA,\bm{\cX},\bm{\Lambda},\bm{\cY})=g(\vA,\bm{\cX})+\frac{1}{2}\sum_{n=1}^N\langle\bm{\Lambda}_n,\vA_n^\top\vA_n-\vI\rangle+\langle\bm{\cY},\cP_\Omega(\bm{\cX})-\cP_\Omega(\bm{\cM})\rangle,$$
where $\bm{\Lambda}=(\bm{\Lambda}_1,\ldots,\bm{\Lambda}_N)$ and $\bm{\cY}$ are Lagrangian multipliers, and $\bm{\Lambda}_n$'s are symmetric matrices.
Letting $\nabla \cL_g=\vzero$, we derive the KKT conditions of \eqref{eq:main3} to be
\begin{subequations}\label{eq:kkt3}
\begin{align}
\vG_n\vG_n^\top\vA_n-\vA_n\bm{\Lambda}_n=&\,\vzero,\,n=1,\ldots,N,\label{eq:kkt3-a}\\
\bm{\cX}-\bm{\cX}\times_{i=1}^N\vA_i\vA_i^\top+\cP_\Omega(\bm{\cY})=&\,\vzero,\label{eq:kkt3-x}\\
\vA_n^\top\vA_n-\vI=&\,\vzero,\,n=1,\ldots,N,\label{eq:kkt3-fea-a}\\
\cP_\Omega(\bm{\cX})-\cP_\Omega(\bm{\cM})=&\,\vzero,\label{eq:kkt3-fea-x}
\end{align}
\end{subequations}
where $$\vG_n=\unfold_n\big(\bm{\cX}\times_{i=1}^{n-1}\vA_i^\top
\times_{i=n+1}^N\vA_i^\top\big).$$ 
From \cite[Proposition 3.1.1]{Bertsekas-NLP}, we have that any local minimizer of \eqref{eq:main3} satisfies the conditions in \eqref{eq:kkt3}. Due to nonconvexity of \eqref{eq:main3}, we cannot in general guarantee global optimality, so instead we aim at showing the first-order optimality conditions in \eqref{eq:kkt3} holds in the limit. %a result similar to that in Theorem \ref{thm:convg1}.

\subsection{Convergence result} Assuming complete observations, Algorithm \ref{alg:ihooi} includes the HOOI algorithm in \cite{de2000best} as a special case. To the best of our knowledge, in general, no convergence result has been established for HOOI, except that it makes the objective value nonincreasing at the iterates and thus converging to some real number \cite[pp. 478]{kolda2009tensor}. The special case of HOOI with $r_n=1,\forall n$ has been analyzed in the literature (e.g., \cite{zhang2001rank, uschmajew2014ALS}). In this subsection, we make an assumption similar to that assumed by the orthogonal iteration (e.g., \cite[Theorem 7.31]{GolubVanLoan1996}) and show global convergence of Algorithm \ref{alg:ihooi} in terms of a first-order optimality condition. Our result implies the convergence of HOOI as a special case.

%The convergence of Algorithm \ref{alg:ihooi} is more difficult to establish than that of Algorithm \ref{alg:lrtfit}. It is essentially because the $\vA_n$-subproblem \eqref{equpdate3-a} is solved over a non-convex set, and we cannot show $\lim_{k\to\infty}(\vA_n^k)^\top\vG_n^k-(\vA_n^{k+1})^\top\vG_n^k=\vzero$ similar to \eqref{eq:lim-a} which is used to show the convergence of Algorithm \ref{alg:lrtfit}. However, after orthogonal transformation on $\vA_n^{k+1}$ (see Lemma \ref{lem:dec-spec} below), the difference will approach \emph{zero} under a certain condition. We proceed our analysis with the following lemma.

%In the subsequent analysis, we let $$\vG_n^k=\unfold_n\left(\bm{\cX}^k\times_{i=1}^{n-1}(\vA_i^{k+1})^\top\times_{i=n+1}^N(\vA_i^k)^\top\right).$$
We first establish a few lemmas.
\begin{lemma}\label{lem:cvg2-dec}
Let $\{(\vA^k,\bm{\cX}^k)\}_{k=1}^\infty$ be the sequence generated by Algorithm \ref{alg:ihooi}. Then
\begin{subequations}\label{eq:cvg2-lim}
\begin{align}
&\lim_{k\to\infty} \|(\vA_n^{k+1})^\top\vG_n^k\|_F^2-\|(\vA_n^k)^\top\vG_n^k\|_F^2=0,\, n=1,\ldots,N,\label{eq:cvg2-lim-a}\\
&\lim_{k\to\infty}\bm{\cX}^k-\bm{\cX}^{k+1}=\vzero\label{eq:cvg2-lim-x},
\end{align}
\end{subequations}
where $\vG_n^k$ is defined in \eqref{eq:gnk}.
\end{lemma}

\begin{proof}
Note that for all $n$,
$$g(\vA_{< n}^{k+1},\vA_{\ge n}^k,\bm{\cX}^k)-g(\vA_{\le n}^{k+1},\vA_{> n}^k,\bm{\cX}^k)=\frac{1}{2}\big(\|(\vA_n^{k+1})^\top\vG_n^k\|_F^2-\|(\vA_n^k)^\top\vG_n^k\|_F^2\big).$$
Summing up the above inequality over $n$ gives
\begin{equation}\label{eq:cvg2-temp1}
g(\vA^k,\bm{\cX}^k)-g(\vA^{k+1},\bm{\cX}^k)= \frac{1}{2}\sum_{n=1}^N\big(\|(\vA_n^{k+1})^\top\vG_n^k\|_F^2-\|(\vA_n^k)^\top\vG_n^k\|_F^2\big).
\end{equation}
In addition, from Lemma \ref{lem:lip} and \cite[Lemma 2.1]{xu2013block}, it follows that
$$g(\vA^{k+1},\bm{\cX}^k)-g(\vA^{k+1},\bm{\cX}^{k+1})\ge\frac{1}{2}\|\bm{\cX}^k-\bm{\cX}^{k+1}\|_F^2,$$
which together with \eqref{eq:cvg2-temp1} and the nonnegativity of $g$ indicates
$$\sum_{k=0}^\infty\left(\sum_{n=1}^N\big(\|(\vA_n^{k+1})^\top\vG_n^k\|_F^2-\|(\vA_n^k)^\top\vG_n^k\|_F^2\big)+\|\bm{\cX}^k-\bm{\cX}^{k+1}\|_F^2\right)\le 2g(\vA^0,\bm{\cX}^0)<\infty,$$
and thus \eqref{eq:cvg2-lim} immediately follows.
\hfill$\blacksquare$\end{proof}

%Since the set $\cO_n=\{\vA_n: \vA_n^\top\vA_n=\vI\}$ is not convex, 
In general \eqref{eq:cvg2-lim-a} does not imply $(\vA_n^{k+1})^\top\vG_n^k-(\vA_n^k)^\top\vG_n^k\to\vzero$ as $k\to\infty$. However, note that $\vA_n^{k+1}\tilde{\vD}_n$ maximizes $\|\vA_n^\top\vG_n^k\|_F^2$ over $\cO_n$ for any orthogonal matrix $\tilde{\vD}_n$. We can choose a certain orthogonal $\tilde{\vD}_n^{k+1}$ such that $(\vA_n^{k+1}\tilde{\vD}_n^{k+1})^\top\vG_n^k-(\vA_n^k)^\top\vG_n^k\to\vzero$ as $k\to\infty$ under some conditions.

\begin{lemma}[von Neumann's Trace Inequality \cite{mirsky1975trace}]\label{lem:von-ineq}
For any matrices $\vX$ and $\vY$ in $\RR^{s\times t}$, it holds that
\begin{equation}\label{eq:von-ineq}\langle\vX,\vY\rangle\le \sum_{i=1}^{\min(s,t)}\sigma_i(\vX)\sigma_i(\vY).
\end{equation}
The inequality \eqref{eq:von-ineq} holds with equality if $\vX$ and $\vY$ have the same left and right singular vectors.
\end{lemma}

We use the Trace Inequality \eqref{eq:von-ineq} to show the following result.

\begin{lemma}\label{lem:spec-a}
Let $\vX\in\RR^{s\times r}$ be a matrix with orthonormal columns, i.e., $\vX^\top\vX=\vI$. For any matrix $\vY\in\RR^{s\times t}$, let $\vY=\vU\bm{\Sigma}\vV^\top+\vU_\perp\bm{\Sigma}_\perp\vV_\perp^\top$ be its full SVD, where $\vU\in\RR^{s\times r}$ corresponds to the leading $r$ singular values. Then
\begin{equation}\label{eq:spec-a}
\|\vU^\top\vY\|_F^2-\|\vX^\top\vY\|_F^2\ge \frac{1-\sigma_{r+1}^2(\vY)/\sigma_r^2(\vY)}{1+\sigma_{r+1}^2(\vY)/\sigma_r^2(\vY)}\|\tilde{\vD}^\top\vU^\top\vY-\vX^\top\vY\|_F^2,
\end{equation}
where we use the convention $0/0=0$, and
\begin{equation}\label{eq:probd}\tilde{\vD}=\argmax_{\vD:\,\vD^\top\vD=\vI}\langle\tilde{\vD}^\top, \vX^\top\vU\bm{\Sigma}^2\rangle.
\end{equation}
If $\vX^\top\vU\bm{\Sigma}^2$ has full SVD $\tilde{\vU}\tilde{\bm{\Sigma}}\tilde{\vV}^\top$, we can take $\tilde{\vD}=\tilde{\vV}\tilde{\vU}^\top$.
\end{lemma}

\begin{proof}
Let $$\alpha=\frac{1-\sigma_{r+1}^2(\vY)/\sigma_r^2(\vY)}{1+\sigma_{r+1}^2(\vY)/\sigma_r^2(\vY)}.$$
Then $(1+\alpha)\sigma_{r+1}^2(\vY)=(1-\alpha)\sigma_r^2(\vY)$. From Lemma \ref{lem:von-ineq}, it follows that $\tilde{\vD}=\tilde{\vV}\tilde{\vU}^\top$ solves the maximization problem in \eqref{eq:probd} if the full SVD of $\vX^\top\vU\bm{\Sigma}^2$ is $\tilde{\vU}\tilde{\bm{\Sigma}}\tilde{\vV}^\top$. In addition, since $\sigma_{\max}(\vX^\top\vU)\le 1$,
\begin{equation}\label{eq:d}
\|\vX^\top\vU\bm{\Sigma}\|_F^2=\langle\vX^\top\vU,\vX^\top\vU\bm{\Sigma}^2\rangle \le \sum_{i=1}^r\sigma_i(\vX^\top\vU\bm{\Sigma}^2)
=\langle\tilde{\vD},\vX^\top\vU\bm{\Sigma}^2\rangle.
\end{equation} 
We have
\begin{align*}
&\,\|\vU^\top\vY\|_F^2-\|\vX^\top\vY\|_F^2 - \alpha\|\tilde{\vD}^\top\vU^\top\vY-\vX^\top\vY\|_F^2\\
=\,&(1-\alpha)\sum_{i=1}^r\sigma_i^2(\vY)+2\alpha\langle\tilde{\vD}^\top
\bm{\Sigma},\vX^\top\vU\bm{\Sigma}\rangle-(1+\alpha)\|\vX^\top\vY\|_F^2\\
=\,&(1-\alpha)\sum_{i=1}^r\sigma_i^2(\vY)+2\alpha\langle\tilde{\vD}^\top
\bm{\Sigma},\vX^\top\vU\bm{\Sigma}\rangle-(1+\alpha)\|\vX^\top\vU\bm{\Sigma}\|_F^2-(1+\alpha)\|\vX^\top\vU_\perp\bm{\Sigma}_\perp\|_F^2\\
\ge\,& (1-\alpha)\sum_{i=1}^r\sigma_i^2(\vY)-(1-\alpha)\|\vX^\top\vU\bm{\Sigma}\|_F^2-(1+\alpha)\|\vX^\top\vU_\perp\bm{\Sigma}_\perp\|_F^2\\
=\,& (1-\alpha)\sum_{i=1}^r\sigma_i^2(\vY)-\left\langle (\vU^\top;\vU_\perp^\top)\vX\vX^\top(\vU,\vU_\perp), \diag\big((1-\alpha)\bm{\Sigma}\bm{\Sigma}^\top,(1+\alpha)\bm{\Sigma}_\perp\bm{\Sigma}_\perp^\top\big)\right\rangle
\ge &\,0,
\end{align*}
where the second equality uses $\vU^\top\vU_\perp=\vzero$, the first inequality follows from \eqref{eq:d}, and the last inequality uses Lemma \ref{lem:von-ineq} and $\sigma_i(\vX\vX^\top)=0,\,\forall i>r$. This completes the proof.
\hfill$\blacksquare$\end{proof}

Directly from Lemma \ref{lem:spec-a}, we have the following result.
\begin{lemma}\label{lem:dec-spec}
Let $\{(\vA^k,\bm{\cX}^k)\}_{k=1}^\infty$ be the sequence generated by Algorithm \ref{alg:ihooi}. There exist orthogonal matrices $\{\tilde{\vD}_n^k:\,n=1,\ldots,N\}_{k=1}^\infty$ such that
$$\big\|(\vA_n^{k+1})^\top\vG_n^k\big\|_F^2-\big\|(\vA_n^k)^\top\vG_n^k\big\|_F^2\ge \frac{1-\sigma_{r_n+1}^2(\vG_n^k)/\sigma_{r_n}^2(\vG_n^k)}{1+\sigma_{r_n+1}^2(\vG_n^k)/\sigma_{r_n}^2(\vG_n^k)}\big\|(\vA_n^{k+1}\tilde{\vD}_n^{k+1})^\top\vG_n^k-(\vA_n^k)^\top\vG_n^k\big\|_F^2,\,\forall n, k.$$
\end{lemma}

Now we are ready to state and show the convergence result of Algorithm \ref{alg:ihooi}.
\begin{theorem}[Global convergence of Algorithm \ref{alg:ihooi}]\label{thm:convg2}
Let $\{(\vA^k,\bm{\cX}^k)\}_{k=1}^\infty$ be the sequence generated from Algorithm \ref{alg:ihooi}. If
\begin{equation}\label{eq:gap-svd}\limsup_{k\to\infty}\frac{\sigma_{r_n+1}(\vG_n^k)}{\sigma_{r_n}(\vG_n^k)}<1,\,\forall n, 
\end{equation}
then for any finite limit point $(\bar{\vA},\bar{\bm{\cX}})$ of $\{(\vA^k,\bm{\cX}^k)\}_{k=1}^\infty$, there exist multipliers $\bar{\bm{\Lambda}}$ and $\bar{\bm{\cY}}$ such that $(\bar{\vA},\bar{\bm{\cX}},\bar{\bm{\Lambda}},\bar{\bm{\cY}})$ satisfies the KKT conditions in \eqref{eq:kkt3}. Furthermore, if $\{\cP_{\Omega^c}(\bm{\cX}^k)\}_{k=1}^\infty$ is bounded, then there exist multiplier sequences $\{\bm{\Lambda}^k\}_{k=1}^\infty$ and $\{\bm{\cY}^k\}_{k=1}^\infty$ such that
\begin{equation}\label{eq:kkt3-lim}
\lim_{k\to\infty}\nabla\cL_g(\vA^k,\bm{\cX}^k,\bm{\Lambda}^k,\bm{\cY}^k)=\vzero.
\end{equation}
\end{theorem}

\begin{remark}
The condition in \eqref{eq:gap-svd} is similar to the one assumed by the orthogonal iteration method \cite[section 7.3.2]{GolubVanLoan1996} for computing $r$-dimensional dominant invariant subspace of a matrix $\vX$. Typically, the convergence of the orthogonal iteration method requires that there is a gap between the $r$-th and $(r+1)$-th largest eigenvalues of $\vX$ in magnitude, because otherwise, the $r$-dimensional dominant invariant subspace of $\vX$ is not unique. Similarly, if $\sigma_{r_n}(\vG_n^k)=\sigma_{r_{n}+1}(\vG_n^k)$, then the left $r_n$-dimensional dominant singular vector space of $\vG_n^k$ is not uniquely determined, and the decomposition can oscillate (i.e., \eqref{eq:cvg2-lim-d} may not hold) in the case that $\sigma_{r_n}(\vG_n^k)=\sigma_{r_{n}+1}(\vG_n^k)$ holds for infinite number of iteration $k$'s and some mode $n$. 

The drawback of Theorem \ref{thm:convg2} is that the assumption \eqref{eq:gap-svd} depends on the iterates. For the purpose of reconstructing a low-multilinear-rank tensor, if $\bm{\cX}^k$ converges to a rank-$(r_1,\ldots,r_N)$ tensor, then \eqref{eq:gap-svd} automatically holds. However, in general, it is unclear how to remove or weaken the assumption.
From the proof below, we see that all results in the theorem can be obtained if \eqref{eq:cvg2-lim-d} holds, which is indicated by \eqref{eq:gap-svd}. In addition, the theorem implies that the sequence produced by Algorithm \ref{alg:ihooi} cannot converge to a non-critical point because if the sequence converges, then \eqref{eq:cvg2-lim-d} holds and thus the convergence results follow.
\end{remark}

\begin{proof}
%Since the feasibility conditions \eqref{eq:kkt3-fea-a} and \eqref{eq:kkt3-fea-x} are kept during the algorithm, we only need to show \eqref{eq:kkt3-a} and \eqref{eq:kkt3-x}. 
Let $$\bm{\cY}^k=\cP_\Omega\left(\bm{\cX}^k\times_{i=1}^N\big(\vA_i^k(\vA_i^k)^\top\big)\right)
-\cP_\Omega(\bm{\cM})$$
and $\bm{\Lambda}^k=(\bm{\Lambda}^k_1,\ldots,\bm{\Lambda}^k_N)$ with $\bm{\Lambda}_n^k$'s being symmetric matrices such that
\begin{equation}\label{eq:lam}\vG_n^{k-1}(\vG_n^{k-1})^\top\vA_n^k-\vA_n^k\bm{\Lambda}_n^k
=\vzero,\,n=1,\ldots,N.
\end{equation}
From \cite[Proposition 3.1.1]{Bertsekas-NLP}, the existence of $\bm{\Lambda}^k$ is guaranteed since $\vA_n^k$ maximizes $\|\vA_n^\top\vG_n^{k-1}\|_F^2$ over $\cO_n$. It follows from the $\bm{\cX}$-update \eqref{eq:ihooi-x} and \eqref{eq:cvg2-lim-x} that
\begin{equation}\label{eq:cvg2-lim-y}
\lim_{k\to\infty}\bm{\cX}^k-\bm{\cX}^k\times_{i=1}^N\big(\vA_i^k(\vA_i^k)^\top\big)+\bm{\cY}^k=\vzero.
\end{equation}
 
Let $\{\tilde{\vD}_n^k:\, n=1,\ldots,N\}_{k=1}^\infty$ be the matrices specified in Lemma \ref{lem:dec-spec}. Under the condition in \eqref{eq:gap-svd}, we have from \eqref{eq:cvg2-lim-a} and Lemma \ref{lem:dec-spec} that
\begin{equation}\label{eq:cvg2-lim-d}
\lim_{k\to\infty}(\vA_n^k\tilde{\vD}_n^k)^\top\vG_n^{k-1}-(\vA_n^{k-1})^\top\vG_n^{k-1}=\vzero.
\end{equation}
For $n=1,\ldots,N$, let $\tilde{\vA}_n^k=\vA_n^k\tilde{\vD}_n^k$ and
$$\tilde{\vG}_n^k=\unfold_n(\bm{\cX}^k\times_{i=1}^{n-1}(\tilde{\vA}_i^{k+1})^\top\times_{i=n+1}^N\vA_i^k).$$
Then \eqref{eq:cvg2-lim-d} implies
\begin{equation}\label{eq:main3-dec}\lim_{k\to\infty}\bm{\cX}^{k-1}\times_{i=1}^{n-1}(\tilde{\vA}_i^k)^\top\times_{i=n}^N(\vA_i^{k-1})^\top-
\bm{\cX}^{k-1}\times_{i=1}^n(\tilde{\vA}_i^k)^\top\times_{i=n+1}^N(\vA_i^{k-1})^\top
=\vzero,\, \forall n. \end{equation}
Let 
\begin{align*}
&\vH_n^k=\unfold_n(\bm{\cX}^k\times_{i=1}^{n-1}(\vA_i^k)^\top\times_{i=n+1}^N(\vA_i^k)^\top),\\[0.1cm] &\tilde{\vH}_n^k=\unfold_n(\bm{\cX}^k\times_{i=1}^{n-1}(\tilde{\vA}_i^k)^\top\times_{i=n+1}^N(\tilde{\vA}_i^k)^\top).
\end{align*}
Then it follows from \eqref{eq:cvg2-lim-x} and \eqref{eq:main3-dec} that
\begin{equation}\label{eq:cvg2-temp2}\lim_{k\to\infty}(\tilde{\vA}_n^k)^\top\tilde{\vH}_n^k-
(\tilde{\vA}_n^k)^\top\tilde{\vG}_n^{k-1}=\vzero,\,\forall n.
\end{equation}
From the basic theorem of linear algebra, one can write
$$\tilde{\vH}_n^k=\tilde{\vA}_n^k\vY_{\tilde{h}_n^k}+\vZ_{\tilde{h}_n^k},\qquad \tilde{\vG}_n^{k-1}=\tilde{\vA}_n^k\vY_{\tilde{g}_n^{k-1}}+\vZ_{\tilde{g}_n^{k-1}},$$
where the columns of $\vZ_{\tilde{h}_n^k}$ and $\vZ_{\tilde{g}_n^{k-1}}$ belong to the null space of $(\tilde{\vA}_n^k)^\top$. Then \eqref{eq:cvg2-temp2} becomes
\begin{equation}\label{eq:cvg2-temp3}\lim_{k\to\infty}\vY_{\tilde{h}_n^k}-\vY_{\tilde{g}_n^{k-1}}=\vzero,\,\forall n.
\end{equation}

For any finte limit point $(\bar{\vA},\bar{\bm{\cX}})$ of $\{(\vA^k,\bm{\cX}^k)\}_{k=1}^\infty$, there exists a subsequence $\{(\vA^k,\bm{\cX}^k)\}_{k\in\cK}$ converging to $(\bar{\vA},\bar{\bm{\cX}})$, and $\{\tilde{\vG}_n^{k-1}\}_{k\in\cK}$ and $\{\tilde{\vH}_n^k\}_{k\in\cK}$ are bounded. We have from \eqref{eq:cvg2-temp2} that
$$\lim_{\cK\ni k\to\infty}\vY_{\tilde{h}_n^k}\vY_{\tilde{h}_n^k}^\top-
\vY_{\tilde{g}_n^{k-1}}\vY_{\tilde{g}_n^{k-1}}^\top=\vzero,\,\forall n$$
which implies
\begin{equation}\label{eq:cvg2-temp4}
\lim_{\cK\ni k\to\infty}\tilde{\vG}_n^{k-1}(\tilde{\vG}_n^{k-1})^\top\vA_n^k-\tilde{\vH}_n^k(\tilde{\vH}_n^k)^\top
\vA_n^k=\vzero,\,\forall n.
\end{equation}
Note $\tilde{\vG}_n^{k-1}(\tilde{\vG}_n^{k-1})^\top=\vG_n^{k-1}(\vG_n^{k-1})^\top$ and $\tilde{\vH}_n^k(\tilde{\vH}_n^k)^\top=\vH_n^k(\vH_n^k)^\top$. We have from \eqref{eq:cvg2-temp4} that
\begin{equation}\label{eq:cvg2-temp5}
\lim_{\cK\ni k\to\infty}\vG_n^{k-1}(\vG_n^{k-1})^\top\vA_n^k-\vH_n^k(\vH_n^k)^\top\vA_n^k=\vzero,\,\forall n,
\end{equation}
which together with \eqref{eq:lam} indicates 
\begin{equation}\label{eq:cvg2-temp6}
\lim_{\cK\ni k\to\infty}\vH_n^k(\vH_n^k)^\top\vA_n^k-
\vA_n^k\bm{\Lambda}_n^k=\vzero, \,\forall n.
\end{equation}
From \eqref{eq:lam}, it holds that 
$$\bm{\Lambda}_n^k=(\vA_n^k)^\top\vG_n^{k-1}(\vG_n^{k-1})^\top\vA_n^k,$$
and thus \eqref{eq:cvg2-temp5} implies
$$\lim_{\cK\ni k\to\infty}\bm{\Lambda}_n^k=\lim_{\cK\ni k\to\infty}(\vA_n^k)^\top\vH_n^k(\vH_n^k)^\top\vA_n^k=
\bar{\vA}_n^\top\bar{\vH}_n\bar{\vH}_n^\top\bar{\vA}_n=:\bar{\bm{\Lambda}}_n,\,\forall n,$$
where $\bar{\vH}_n=\bar{\bm{\cX}}\times_{i=1}^{n-1}\bar{\vA}_i^\top
\times_{i=n+1}^N\bar{\vA}_i^\top$. In addition,
$$\lim_{\cK\ni k\to\infty}\bm{\cY}^k=\cP_\Omega\big(\bar{\bm{\cX}}
\times_{i=1}^N(\bar{\vA}_i\bar{\vA}_i^\top)\big)-\cP_\Omega(\bm{\cM})=:\bar{\bm{\cY}},$$
and $(\bar{\vA},\bar{\bm{\cX}},\bar{\bm{\Lambda}},\bar{\bm{\cY}})$ satisfies the KKT conditions in \eqref{eq:kkt3} from \eqref{eq:cvg2-lim-y}, \eqref{eq:cvg2-temp6} and the feasibility of $(\vA^k,\bm{\cX}^k)$ for all $k$. 

If $\{\cP_{\Omega^c}(\bm{\cX}^k)\}_{k=1}^\infty$ is bounded, then $\{\bm{\cX}^k\}_{k=1}^\infty$ is bounded, and so are $\{\tilde{\vG}_n^{k-1}\}_{k=1}^\infty$ and $\{\tilde{\vH}_n^k\}_{k=1}^\infty$. Hence, \eqref{eq:cvg2-temp6} holds with $\cK$ being the whole sequence, and \eqref{eq:kkt3-lim} immediately follows from \eqref{eq:cvg2-lim-y}, \eqref{eq:cvg2-temp6} and the feasibility of $(\vA^k,\bm{\cX}^k)$ for all $k$. This completes the proof.
\hfill$\blacksquare$\end{proof}

\section{Numerical results}\label{sec:numerical}
In this section, we test Algorithm \ref{alg:ihooi}, dubbed as iHOOI, on both synthetic and real-world datasets and compare it to the alternating least squares method in Appendix \ref{app:als}, named as ALSaS, and state-of-the-art methods for tensor factorization with missing values and also low-rank tensor completion. 

%Since $\bm{\cC}$ is absorbed into the update of $\vA$ and $\bm{\cX}$ in Algorithm \ref{alg:ihooi}, we expect that it would perform no worse than Algorithm \ref{alg:lrtfit} for solving \eqref{eq:main1} in terms of convergence speed and solution quality and will demonstrate it in section \ref{sec:numerical} (e.g., see Tables \ref{table:facereg} and \ref{table:mri_rec}). However, notice that the update of $\vA_n$ in Algorithm \ref{alg:ihooi} typically requires SVD of $\vG_n^k$ and can be very expensive if $m_n$ and $\Pi_{i\neq n}r_i$ are both large (e.g., see the test in secion \ref{sec:facerecog}).

\subsection{Implementation details} The problem \eqref{eq:main1} requires estimation on rank $r_n$'s. Depending on applications, they can be either fixed or adaptively updated. Usually, smaller $r_n$'s make more data compression but may result in larger fitting error while too large ones can cause overfitting problem. 
Following \cite{wen2012lmafit}, we apply a rank-increasing strategy to both Algorithms \ref{alg:lrtfit} and \ref{alg:ihooi}. We start from small $r_n$'s and then gradually increase them based on the data fitting. Specifically, we set $r_n$'s to some small positive integers (e.g., $r_n=1,\forall n$) at the beginning of the algorithms. After each iteration $k$, we increase one $r_n$ to $\min(r_n+\Delta r_n, r_n^{\max})$ if 
\begin{equation}\label{eq:rechg}
\left|1-\frac{\fit_{k+1}}{\fit_k}\right|\le 10^{-2}
\end{equation}
where $\Delta r_n$ is a small positive integer, $r_n^{\max}$ is the user-specified maximal rank estimate, $\fit_k=\|\cP_\Omega(\bm{\cC}^k\times_{i=1}^N\vA_i^k-\bm{\cM})\|_F$
for Algorithm \ref{alg:lrtfit}, and
$\fit_k=\|\cP_\Omega(\bm{\cX}^k\times_{i=1}^N(\vA_i^k(\vA_i^k)^\top)-\bm{\cM})\|_F$
for Algorithm \ref{alg:ihooi}. The condition in \eqref{eq:rechg} implies ``slow'' progress in the current $(r_1,\ldots,r_N)$-dimensional space, and thus we slightly enlarge the search space. Throughout the tests, we set $\Delta r_n=1,\forall n$, and as the condition in \eqref{eq:rechg} is satisfied, we choose mode $n_0=\argmax_n (r_n^{\max}-r_n)$ and increase $r_{n^0}$ while keeping other $r_n$'s unchanged. In addition, we augment $\vA_{n_0}$ by first adding a randomly generated vector as the last column and then orthonormalizing it. We terminate Algorithms \ref{alg:lrtfit} and \ref{alg:ihooi} if they run to a maximum number of iterations or a maximum running time or one of the following conditions is satisfied at a certain iteration $k$:
\begin{subequations}\label{eq:crit}
\begin{align}
\frac{\fit_{k+1}}{\|\cP_\Omega(\bm{\cM})\|_F}\le &\, tol,\\
\frac{|\obj_{k+1}-\obj_k|}{1+\obj_k}\le &\, tol, 
\end{align}
\end{subequations}
where $\obj_k=f(\bm{\cC}^k,\vA^k,\bm{\cX}^k)$ for Algorithm \ref{alg:lrtfit} and $\obj_k=g(\vA^k,\bm{\cX}^k)$ for Algorithm \ref{alg:ihooi}, and $tol$ is a small positive number that will be specified below.

\subsection{Convergence behavior}\label{sec:speed} We first test the convergence of iHOOI and compare it to ALSaS and the weighted Tucker (WTucker) factorization with missing values in \cite{filipovic2013tucker}. The model solved by WTucker is similar to \eqref{eq:main1} but has no orthogonality constraint on the factor matrices. WTucker is also a BCD-type method and employs the nonlinear conjugate gradient (NCG) method to solve each subproblem. We test the three algorithms on two random tensors. Both of them have the form of $\bm{\cM}=\bm{\cC}\times_1\vA_1\times_2\vA_2\times_3\vA_3$, where the entries of $\bm{\cC}$ and $\vA_n$'s follow identically independent standard normal distribution. The first tensor has balanced size $100\times 100\times 100$ with core size $10\times10\times10$, and the second one is unbalanced with size $10\times10\times1000$ and core size $4\times4\times10$. For both tensors, we uniformly randomly choose 20\% entries of $\bm{\cM}$ as known.  %We set the maximum number of inner iterations to 100 for NCG as suggested in its code.
We use the same random starting points for the three algorithms and run them to 100 iterations or 20 seconds on the first tensor and 1000 iterations or 20 seconds on the second one. Figure \ref{fig:speed} plots their produced relative error $\|\cP_\Omega(\bm{\cC}^k\times_{i=1}^3\vA_i^k-\bm{\cM})\|_F/\|\bm{\cM}\|_F$ with respect to the iteration $k$ or running time. 

In the first row of Figure \ref{fig:speed}, for ALSaS and iHOOI, we initialize $r_n=\round(\frac{3}{4}\rankk(\vM_{(n)})),\,\forall n$ and apply the rank-increasing strategy  with $r_n^{\max}=\round(\frac{5}{4}\rankk(\vM_{(n)})),\,\forall n$, and for WTucker, we fix $r_n=\round(\frac{5}{4}\rankk(\vM_{(n)})),\,\forall n$ as suggested in \cite{filipovic2013tucker}. In the second row, we fix $r_n=\rankk(\vM_{(n)}),\,\forall n$ for all three algorithms. From the figure, we see that ALSaS and iHOOI perform almost the same in both rank-increasing and rank-fixing cases. WTucker decreases the fitting error faster than ALSaS and iHOOI in the beginning. However, it does not decrease the error any more or decreases extremely slowly after a few iterations while ALSaS and iHOOI can further decrease the errors into lower values. This may be because WTucker stagnates at some local solution. %We note that even if we fix $r_n =\rankk(\unfold_n(\bm{\cM}))$, WTucker still performs like what is shown.
Although the code of WTucker may be modified to also include a rank-adjusting strategy to help avoid local solutions, we note that NCG for subproblems converges slowly in the first several outer iterations and then becomes fast due to warm-start. Hence, we doubt that WTucker can be even slower with an adaptive rank-adjusting strategy.

\begin{figure}\caption{Convergence behavior of ALSaS, iHOOI, and WTucker without using true ranks (first row) and with $r_n$'s fixed to true ranks (second row) on $100\times100\times100$ (first two columns) and $10\times10\times1000$ (last two colums) random tensors.}\label{fig:speed}
\centering
\includegraphics[width=0.22\textwidth]{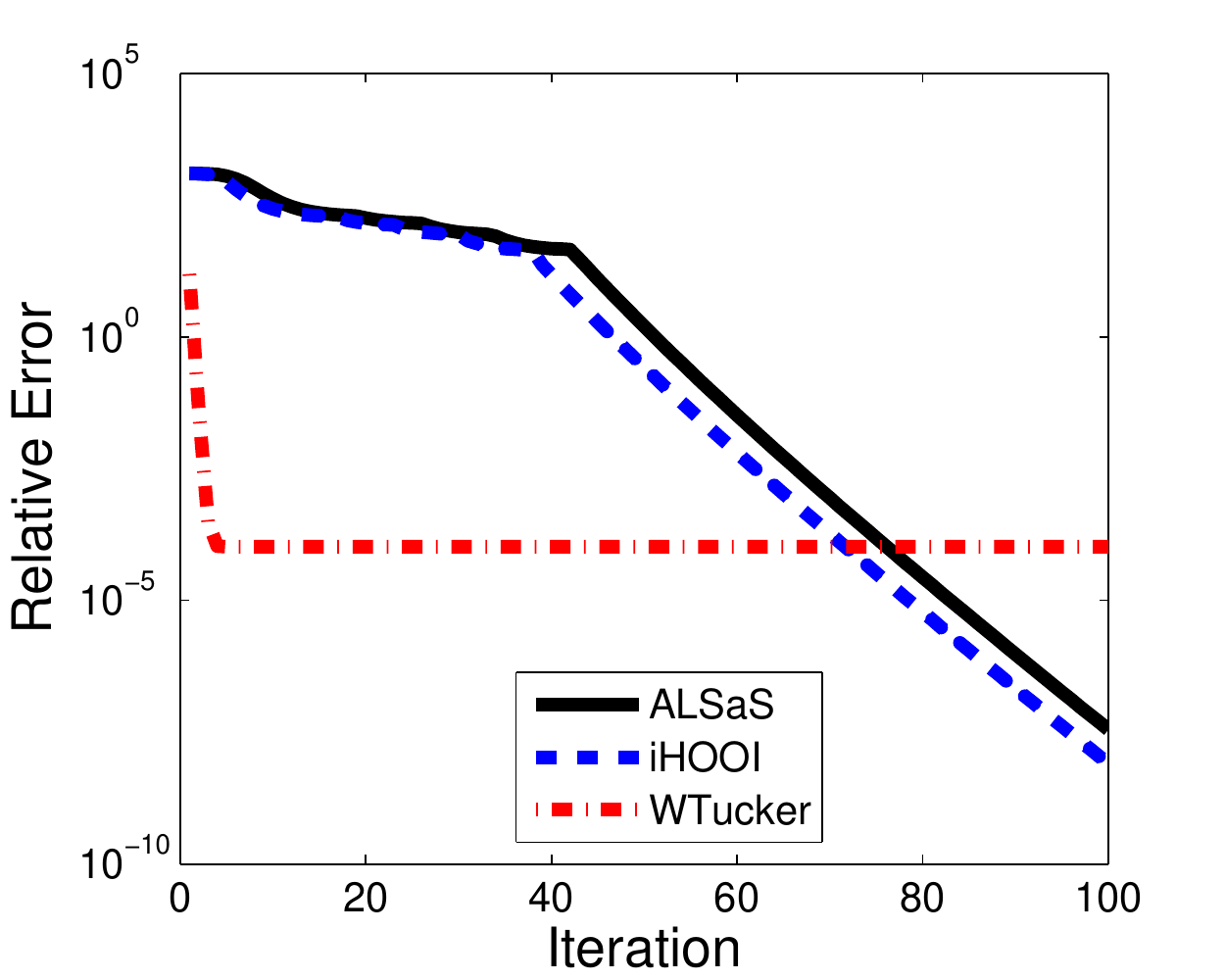}
\includegraphics[width=0.22\textwidth]{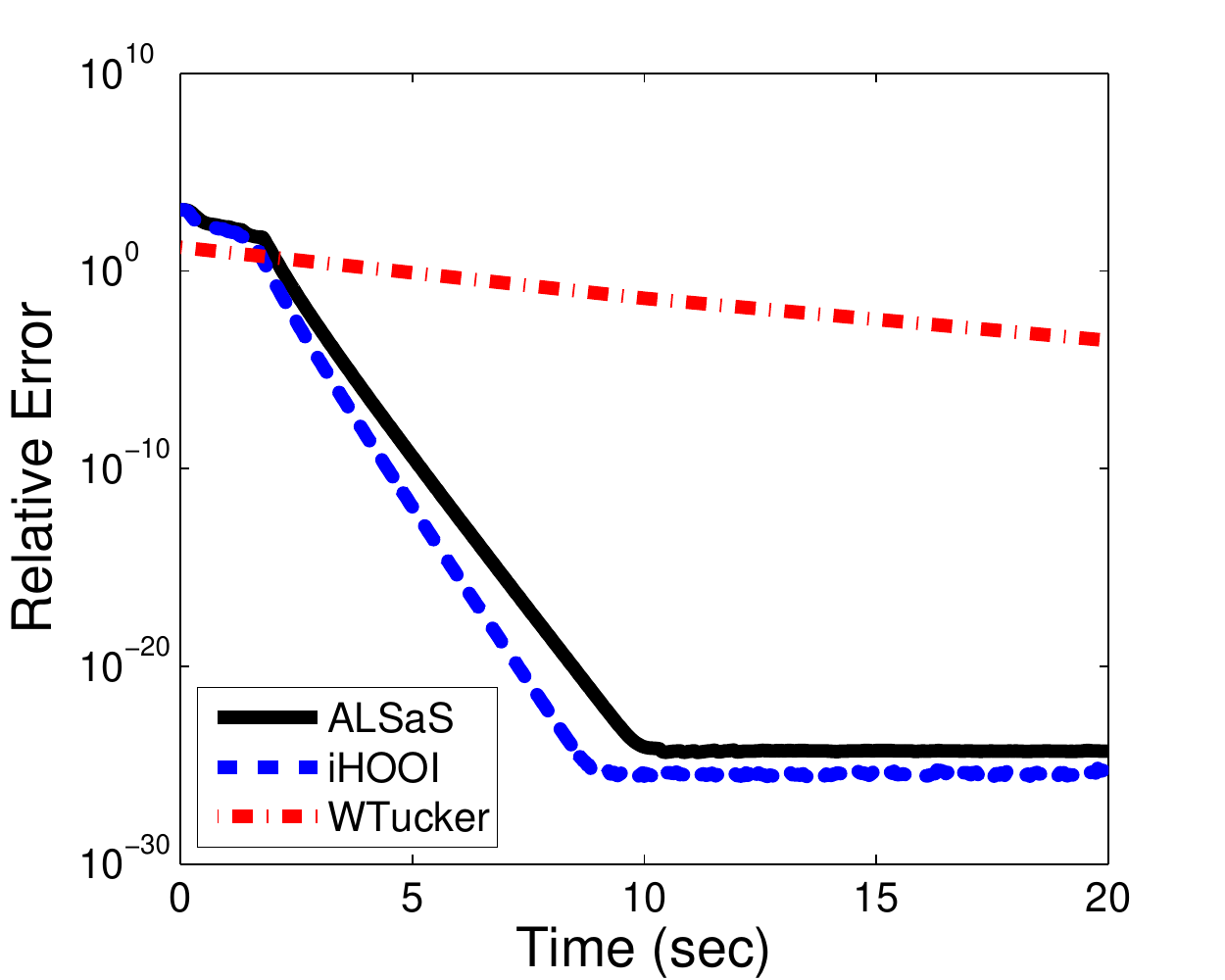}
\includegraphics[width=0.22\textwidth]{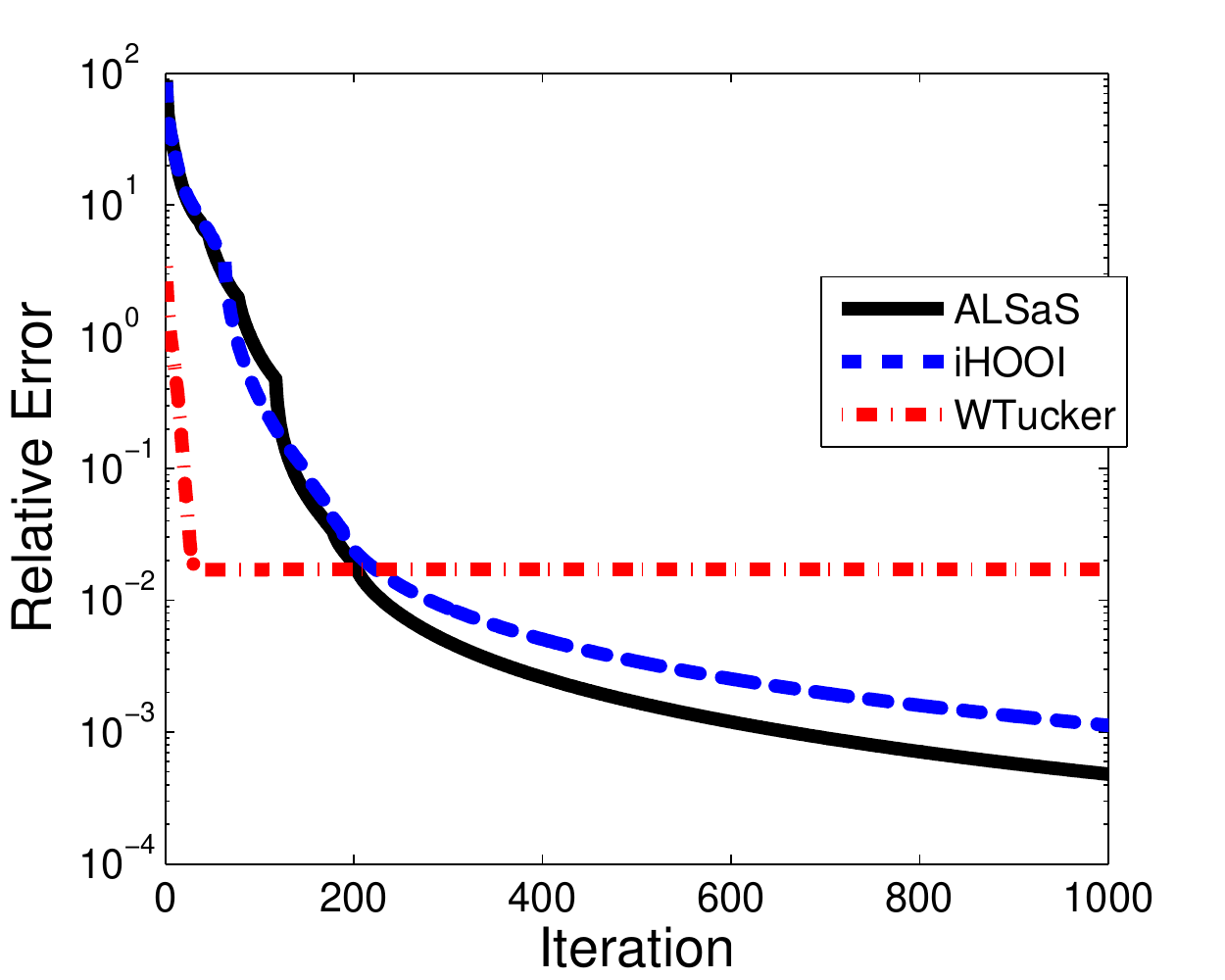}
\includegraphics[width=0.22\textwidth]{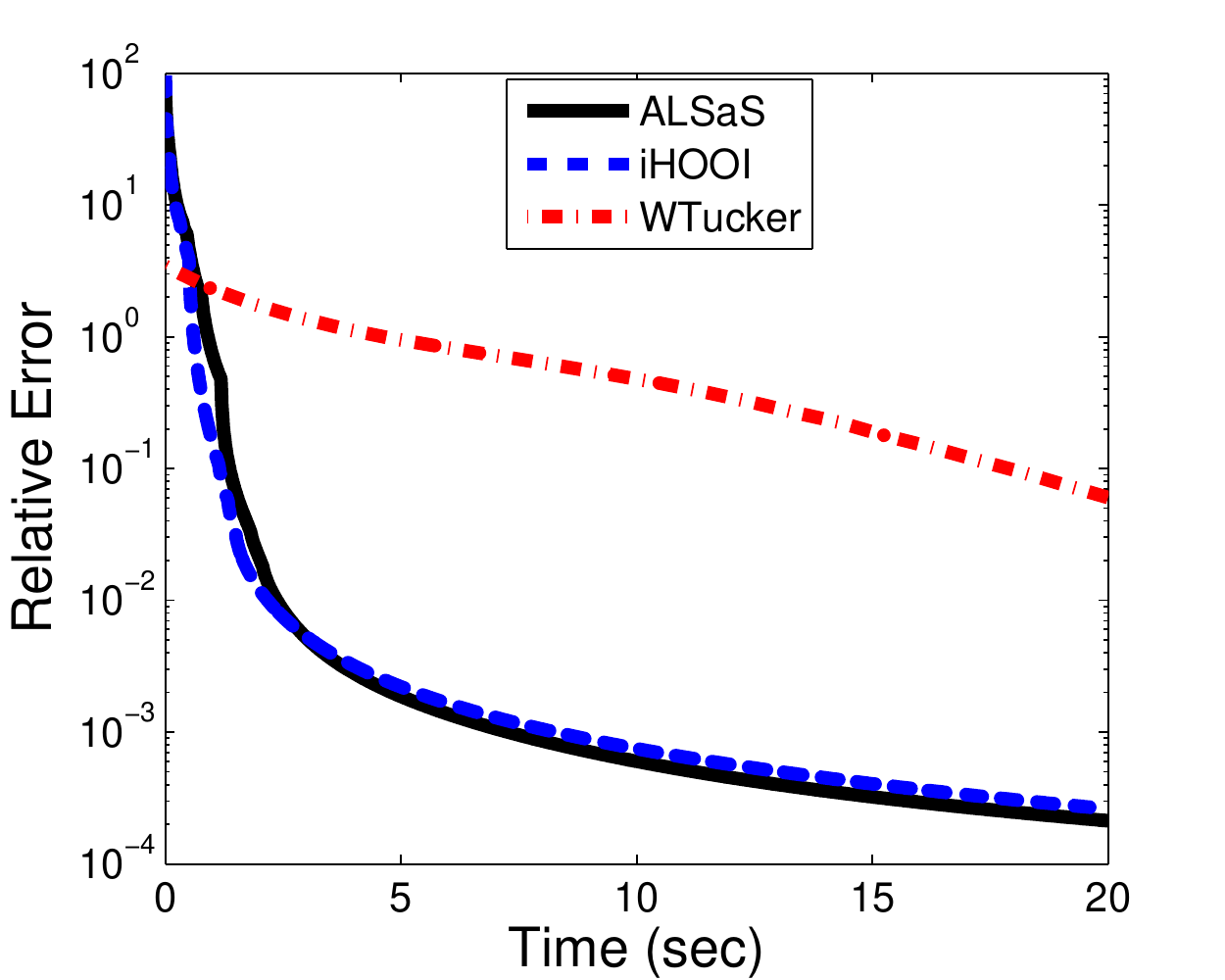}\\
\includegraphics[width=0.22\textwidth]{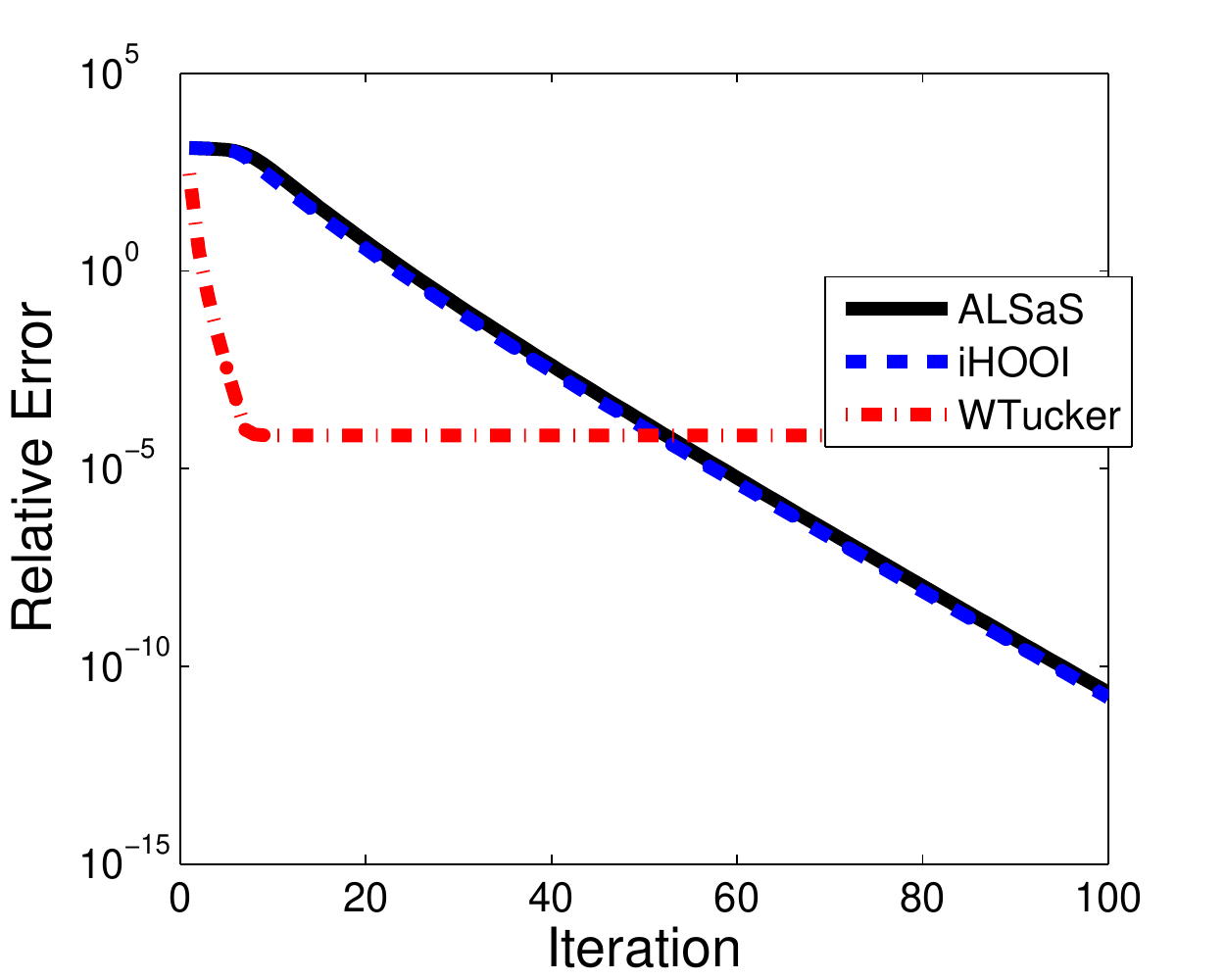}
\includegraphics[width=0.22\textwidth]{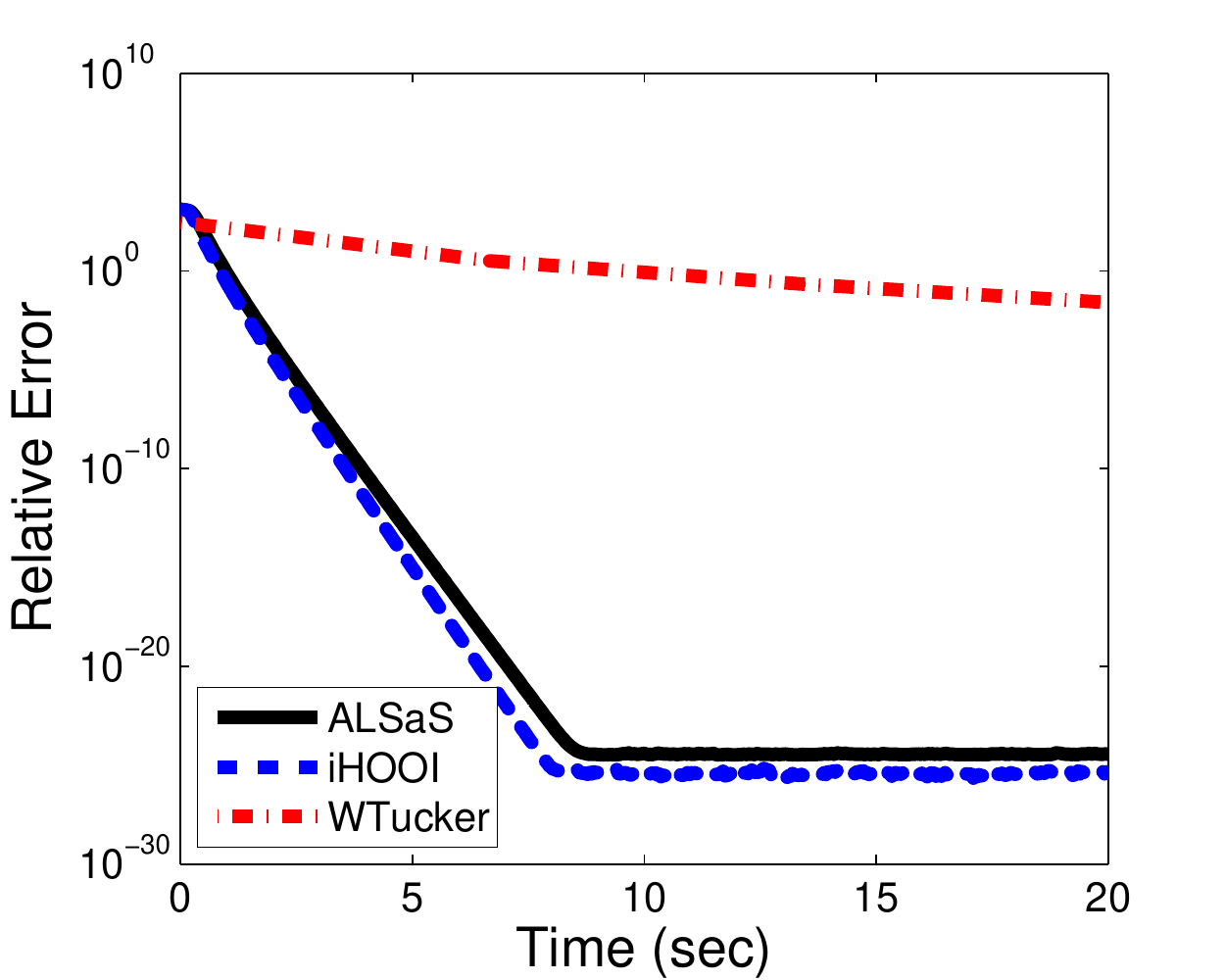}
\includegraphics[width=0.22\textwidth]{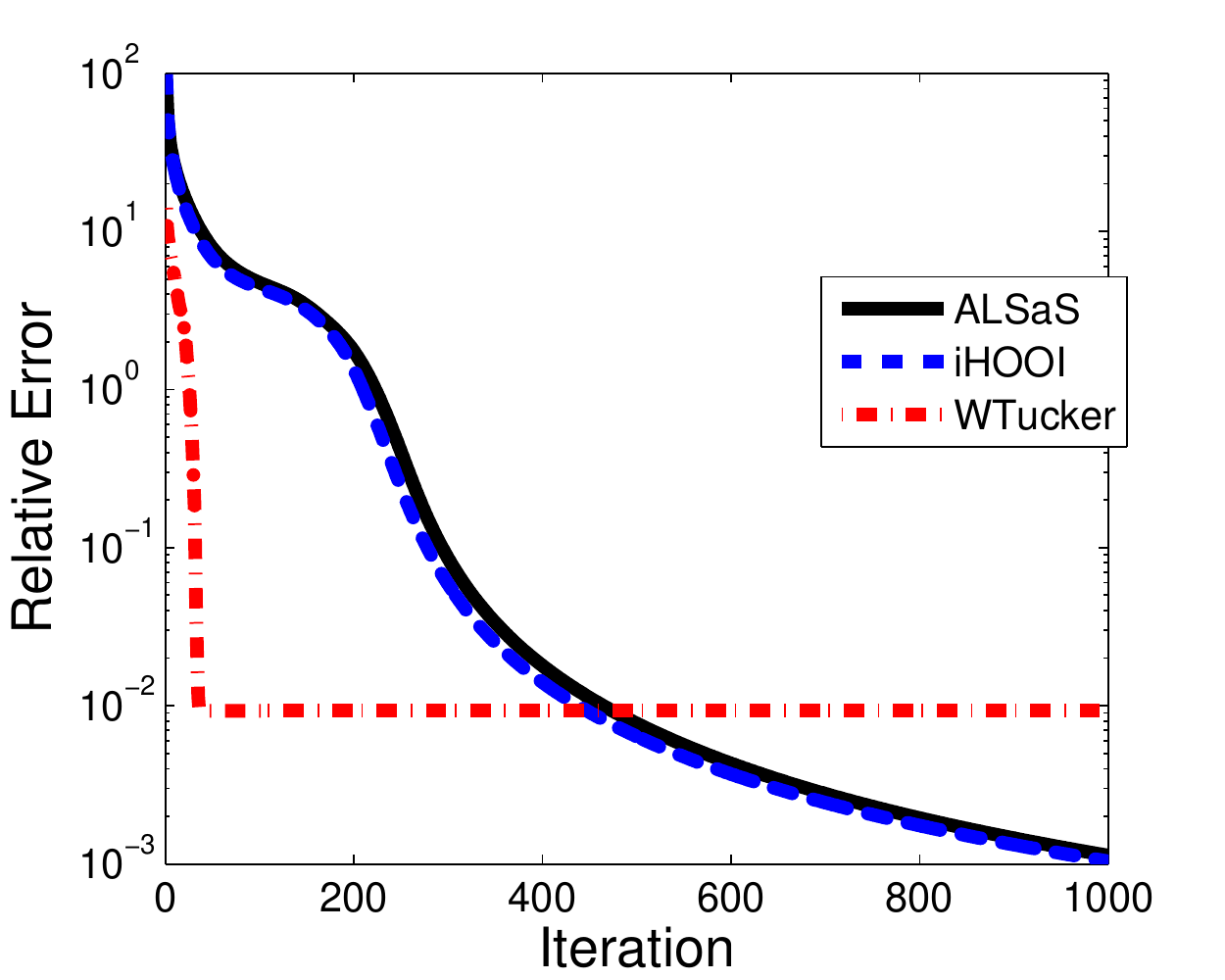}
\includegraphics[width=0.22\textwidth]{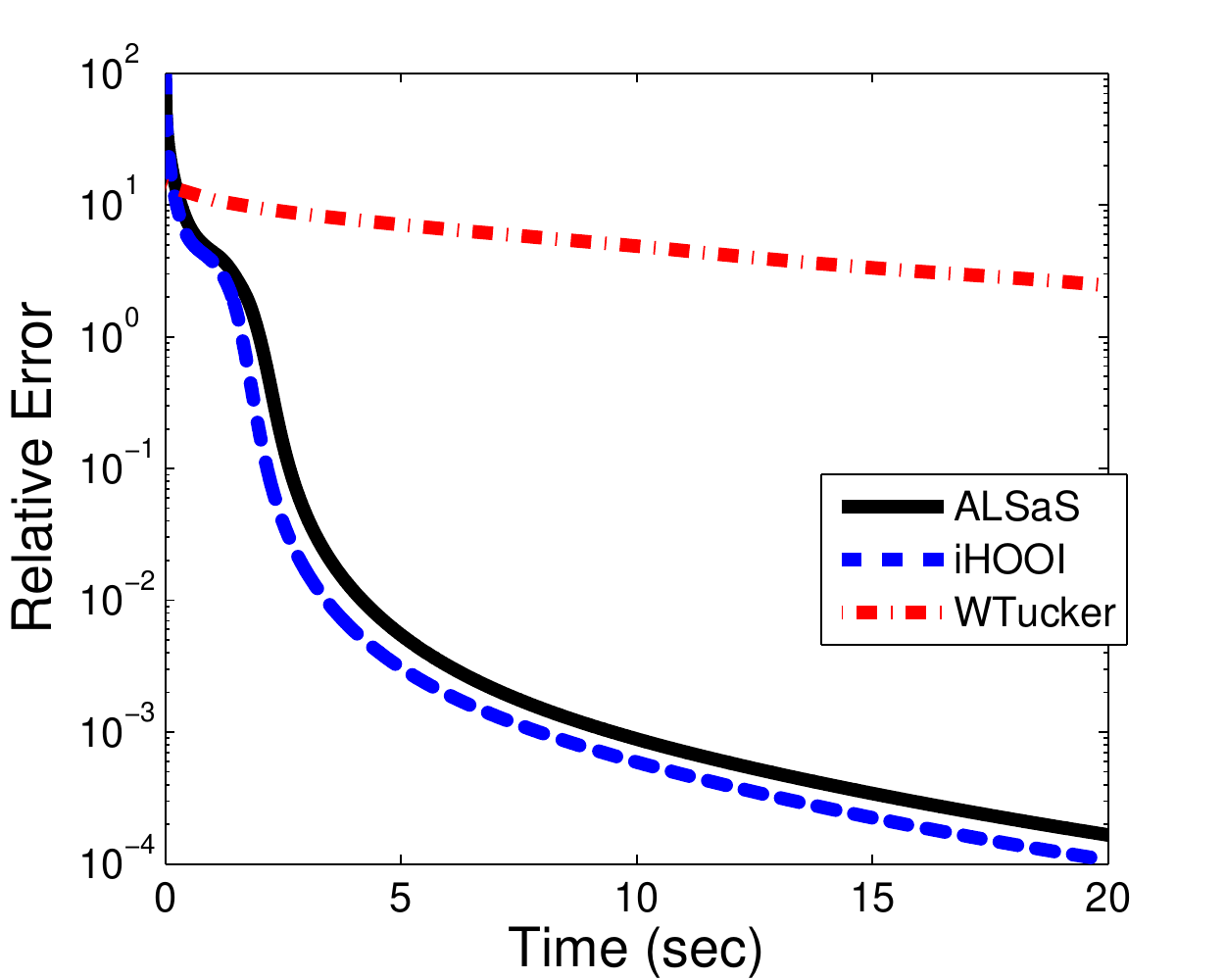}
\end{figure}

\subsection{Recoverability on factors} In this section, we compare the performance of ALSaS, iHOOI, and WTucker in a different measure. The goal of solving \eqref{eq:main1} is to get an approximate HOSVD of the underlying tensor $\bm{\cM}$, and similarly WTucker is to obtain an approximate Tucker factorization of $\bm{\cM}$. For consistent comparison, we normalize the factors given by WTucker in the same way as in \eqref{eq:normal}. Note that the normalization does not change the data fitting. To evaluate the decomposition, we measure the distance of the output factors (after certain rotation) to the original ones. Specifically, suppose %that the original tensor is $\bm{\cM}=\bm{\cC}\times_{n=1}^N\vA_n$ and the estimated one is $\hat{\bm{\cM}}=\hat{\bm{\cC}}\times_{n=1}^N\hat{\vA}_n$, where 
that $\bm{\cC},\hat{\bm{\cC}}\in\RR^{r_1\times\ldots\times r_N}$, and $\vA_n, \hat{\vA}_n\in\RR^{m_n\times r_n}$  both have orthonormal columns for all $n$. We let
\begin{equation}\label{eq:dist}
\err(\bm{\cC},\vA; \hat{\bm{\cC}},\hat{\vA}) = \frac{\|\bm{\cC}\times_{n=1}^N\vA_n-\hat{\bm{\cC}}\times_{n=1}^N\hat{\vA}_n\|_F}{\|\bm{\cC}\times_{n=1}^N\vA_n\|_F}+\sum_{n=1}^N\frac{\sqrt{r_n}-\|\vA_n^\top\hat{\vA}_n\|_F}{\sqrt{r_n}}.
\end{equation}
From the following theorem, we have that if $\err(\bm{\cC},\vA; \hat{\bm{\cC}},\hat{\vA})$ is small, the subspaces spanned by $\vA_n$ and $\hat{\vA}_n$ are close to each other for all $n$, and also after some orthogonal transformation, $\hat{\bm{\cC}}$ is close to $\bm{\cC}$. It is not difficult to show the theorem, and thus we omit its proof.

\begin{theorem}
For $\vA_n, \hat{\vA}_n\in\RR^{m_n\times r_n}$ with orthonormal columns,
it holds that $\|\vA_n^\top\hat{\vA}_n\|_F\le\sqrt{r_n}$. If the inequality holds with equality, then $\vA_n^\top\hat{\vA}_n$ is orthogonal and $\hat{\vA}_n=\vA_n\vA_n^\top\hat{\vA}_n$. In addition, if $\bm{\cC}\times_{n=1}^N\vA_n=\hat{\bm{\cC}}\times_{n=1}^N\hat{\vA}_n$ and $\|\vA_n^\top\hat{\vA}_n\|_F=\sqrt{r_n},\,\forall n$, then $\bm{\cC}=\hat{\bm{\cC}}\times_{n=1}^N(\vA_n^\top\hat{\vA}_n).$
\end{theorem}  

We test the three algorithms on two random datasets. The first set consists of $50\times50\times50$ random tensors and the second one $30\times30\times30\times30$. They are first generated in the same way as that in section \ref{sec:speed} with $\rankk(\vM_{(n)})=r, \forall n$, and then factor matrix $\vA_n$'s are orthonormalized. We compare the performance of ALSaS, iHOOI, and WTucker on different $r$'s and sample ratios defined as
$$\mathrm{SR}:=\frac{|\Omega|}{\Pi_{i=1}^N m_i},$$
where $|\Omega|$ denotes the cardinality of $\Omega$. The indices in $\Omega$ are selected uniformly at random. For each pair of $r$ and SR, we generate 30 tensors independently and run the three algorithms with fixed $r_n=r,\forall n$ to 2000 iterations or stopping tolerance $10^{-6}$. All three algorithms start from the same random points. Let $(\bm{\cC},\vA)$ be the generated factors and $(\hat{\bm{\cC}},\hat{\vA})$ the output of an algorithm. If $\err(\bm{\cC},\vA; \hat{\bm{\cC}},\hat{\vA})\le 10^{-2}$, we regard the recovery to be successful. Figure \ref{fig:factor} shows the rates of successful recovery by each algorithm. From the figure, we see that iHOOI performs the same as ALSaS in all cases except at $r=21$ and $\mathrm{SR}=30\%$ in the second dataset where the former is slightly better. Both ALSaS and iHOOI perform much better than WTucker in particular for the second dataset.

\begin{figure}\caption{Rates of successfully recovering factors by ALSaS, iHOOI, and WTucker on $50\times50\times50$ (first row) and $30\times30\times30\times30$ (second row) randomly generated tensors for different multilinear ranks and sample ratios.}\label{fig:factor}
\vspace{0.2cm}
\centering
\begin{tabular}{ccc}
SR=10\% & SR=30\% & SR=50\% \\
\includegraphics[width=0.25\textwidth]{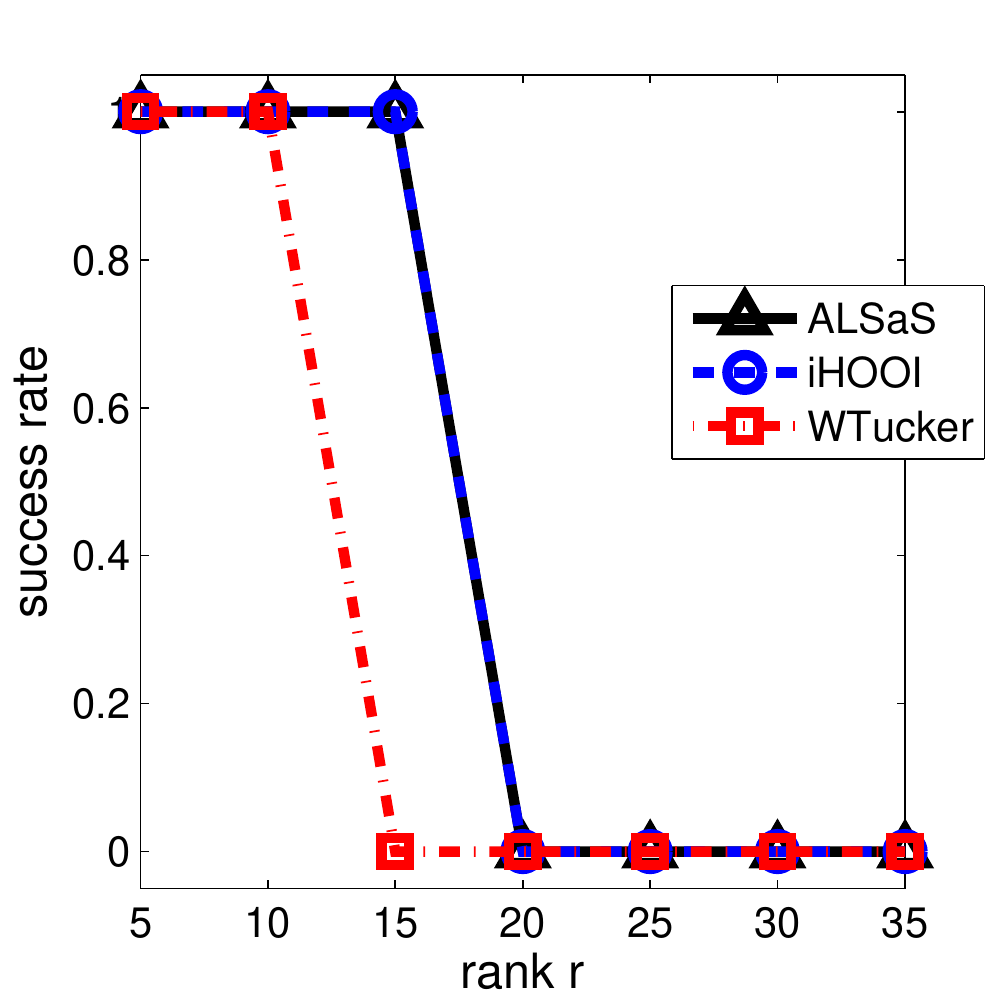}&
\includegraphics[width=0.25\textwidth]{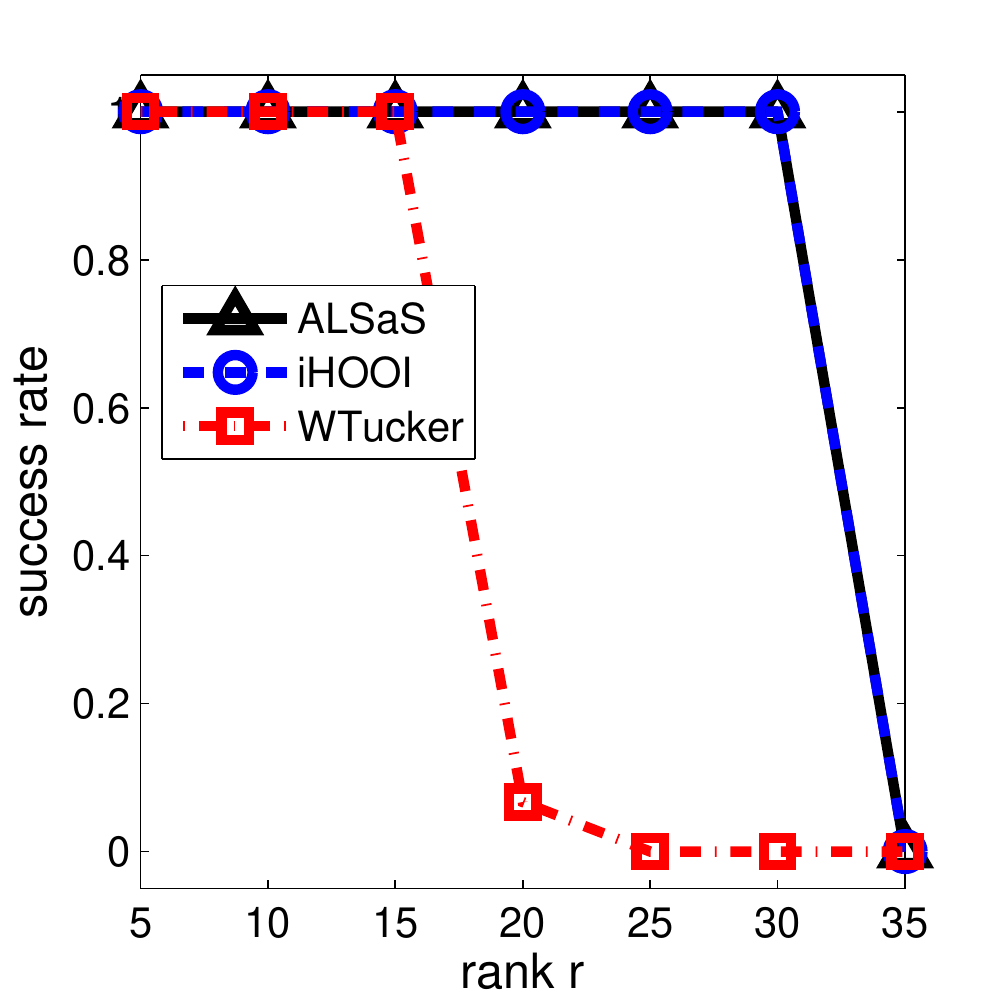}&
\includegraphics[width=0.25\textwidth]{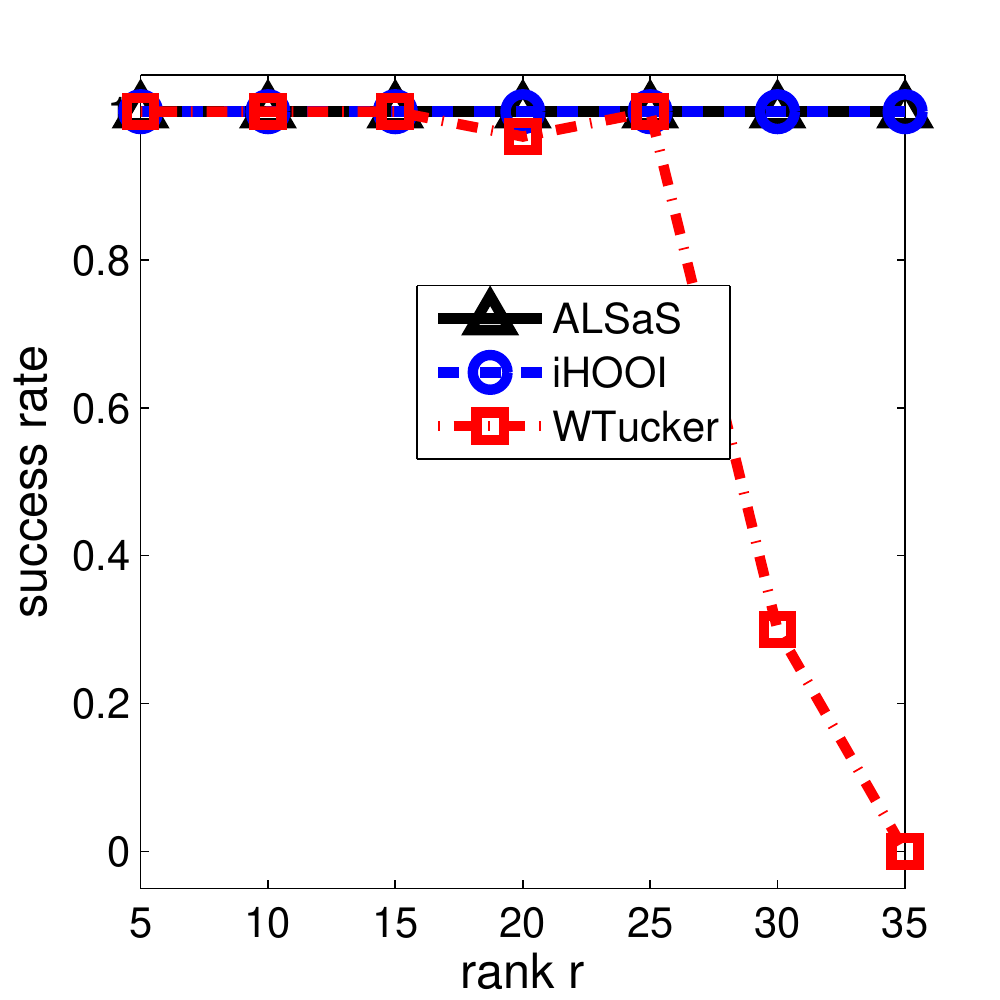}\\
\includegraphics[width=0.25\textwidth]{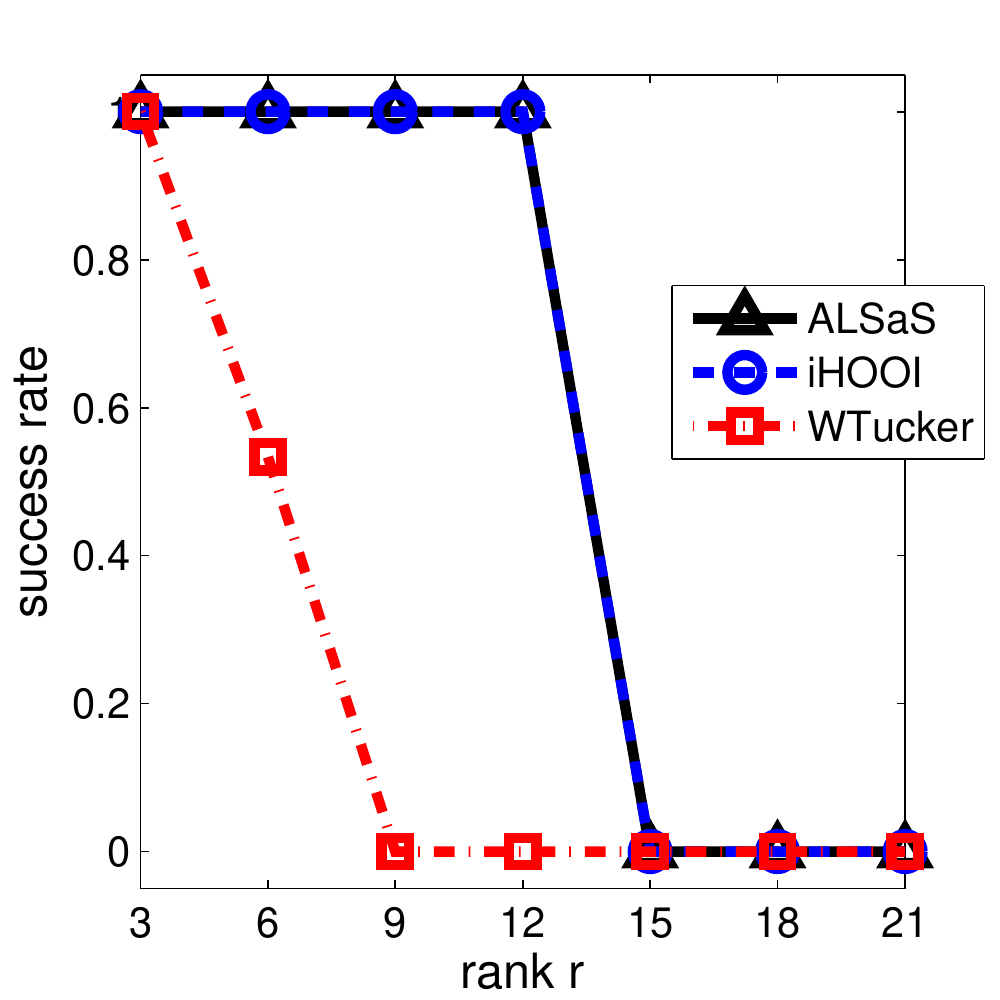}&
\includegraphics[width=0.25\textwidth]{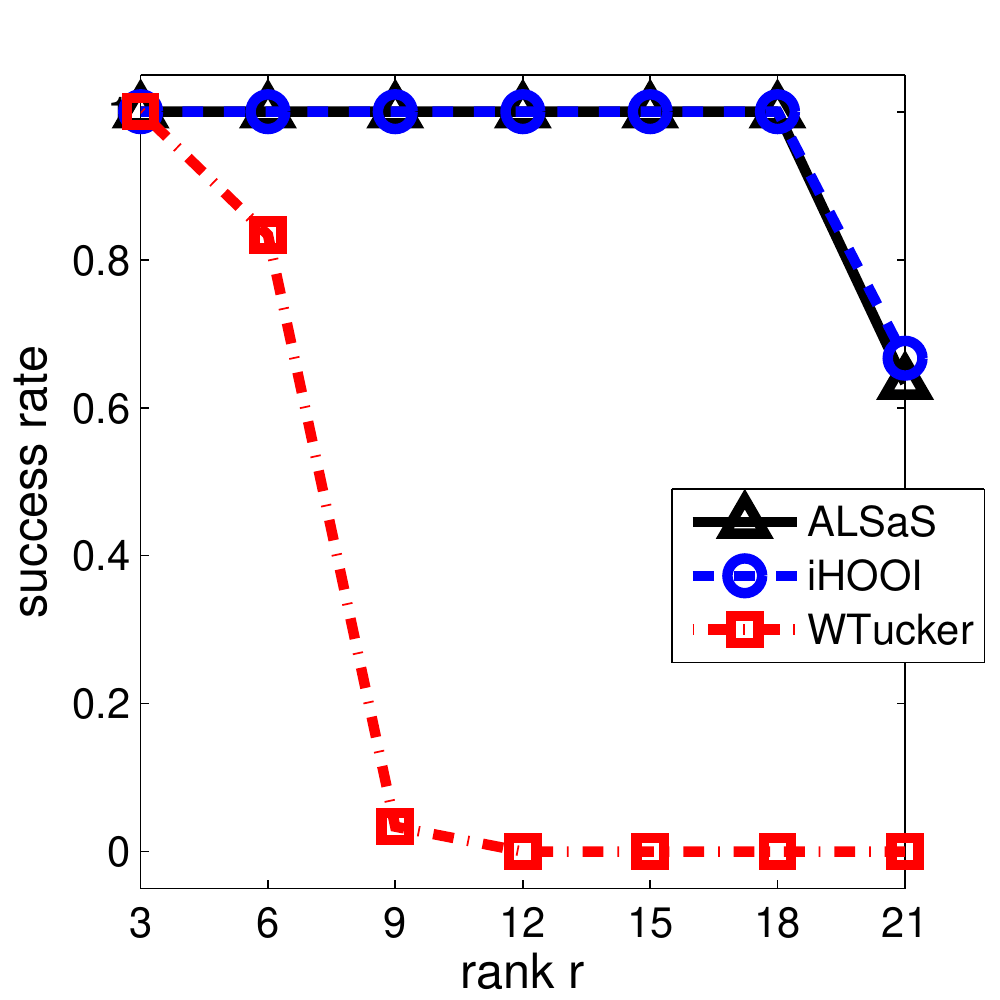}&
\includegraphics[width=0.25\textwidth]{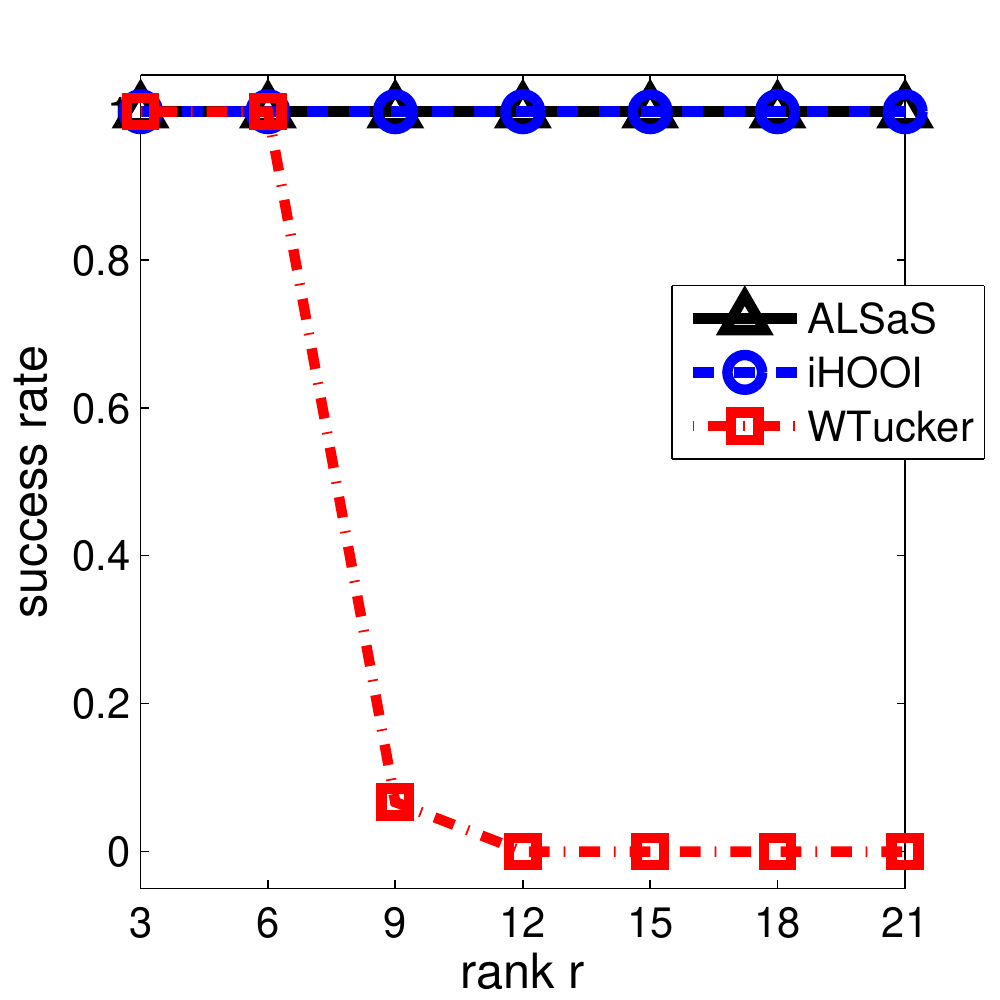}
\end{tabular}
\end{figure}

\subsection{Application to face recognition}\label{sec:facerecog} In this subsection, we use the factors obtained from ALSaS, iHOOI, and WTucker for face recognition and compare their prediction accuracies. As in the previous test, we normalize the factors given by WTucker. We use the cropped images in the extended Yale Face Database B\footnote{\url{http://vision.ucsd.edu/~leekc/ExtYaleDatabase/ExtYaleB.html}} \cite{georghiades2001few, lee2005acquiring},
which has 38 subjects with each one consisting of 64 face images taken under different illuminations. Each image originally has pixels of $168\times192$ and is downsampled into $51\times 58$ in our test. We vectorize each downsampled image and form the dataset into a $38\times64\times2958$ tensor. Some face images of the first subject are shown in Figure \ref{fig:facereg}.

\begin{figure}\caption{Selected face images of the first subject in the extended Yale Face Database B corresponding to 16 different illuminations.}\label{fig:facereg}
\vspace{0.2cm}
\centering
\includegraphics[width=0.055\textwidth]{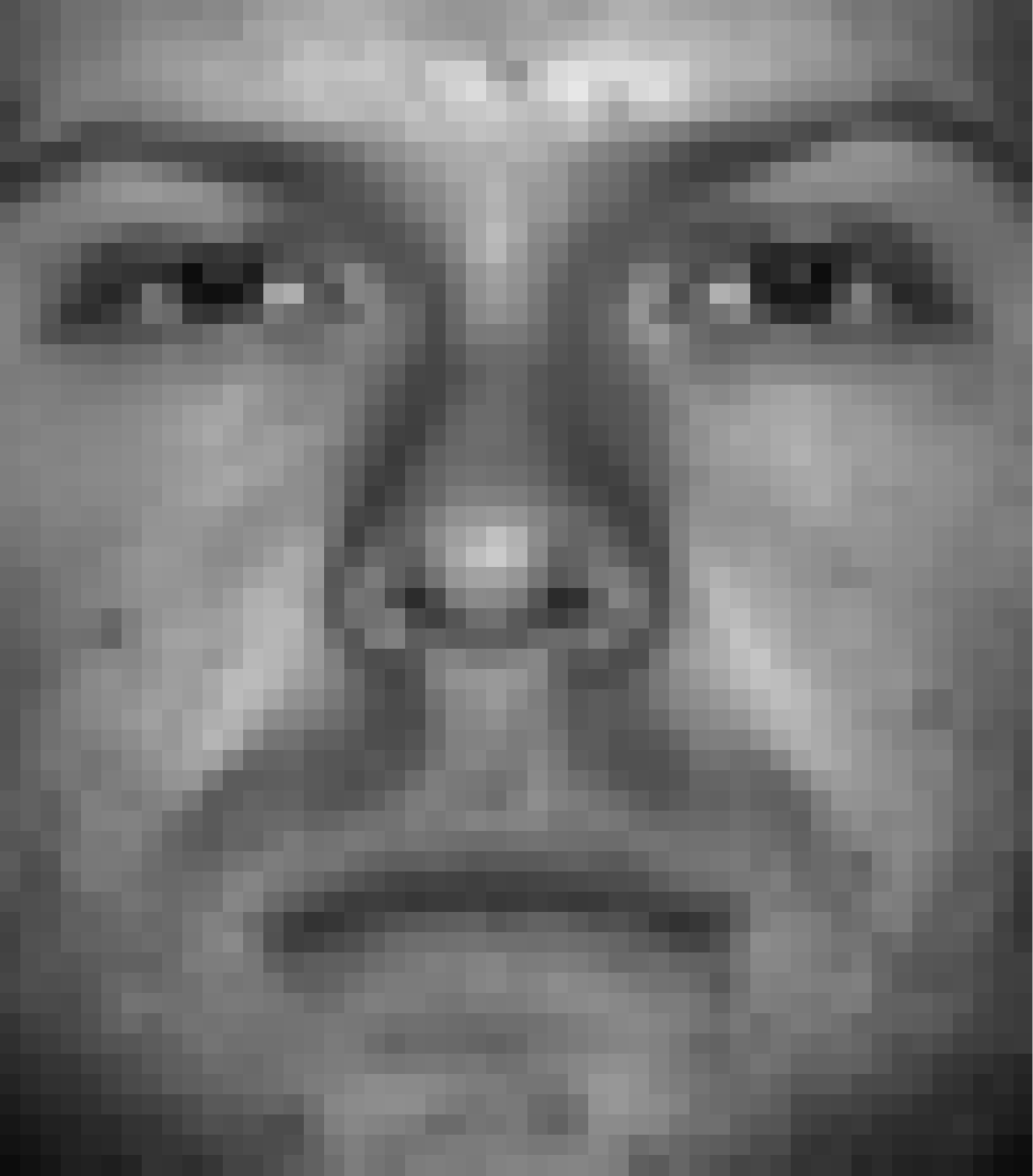}
\includegraphics[width=0.055\textwidth]{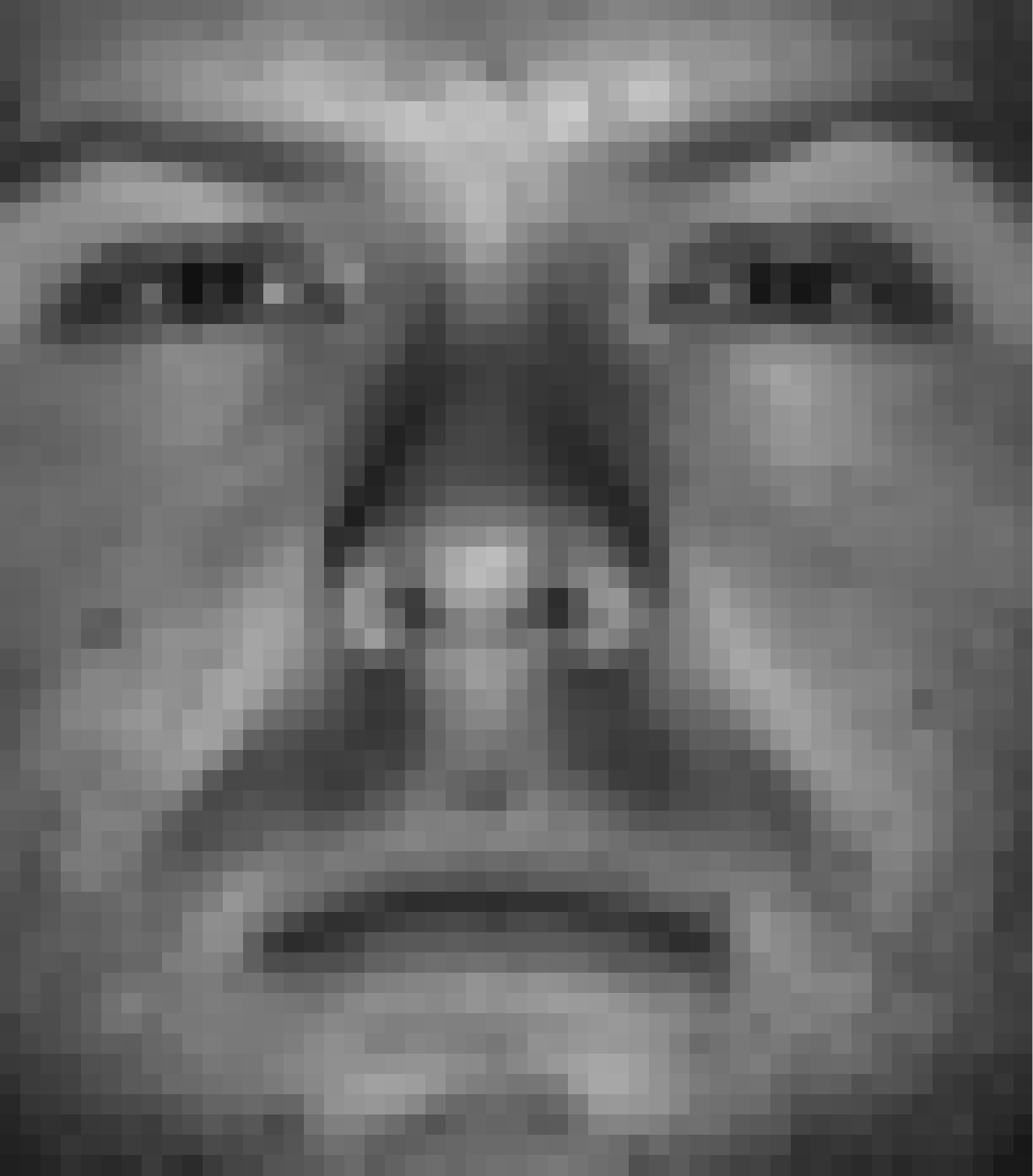}
\includegraphics[width=0.055\textwidth]{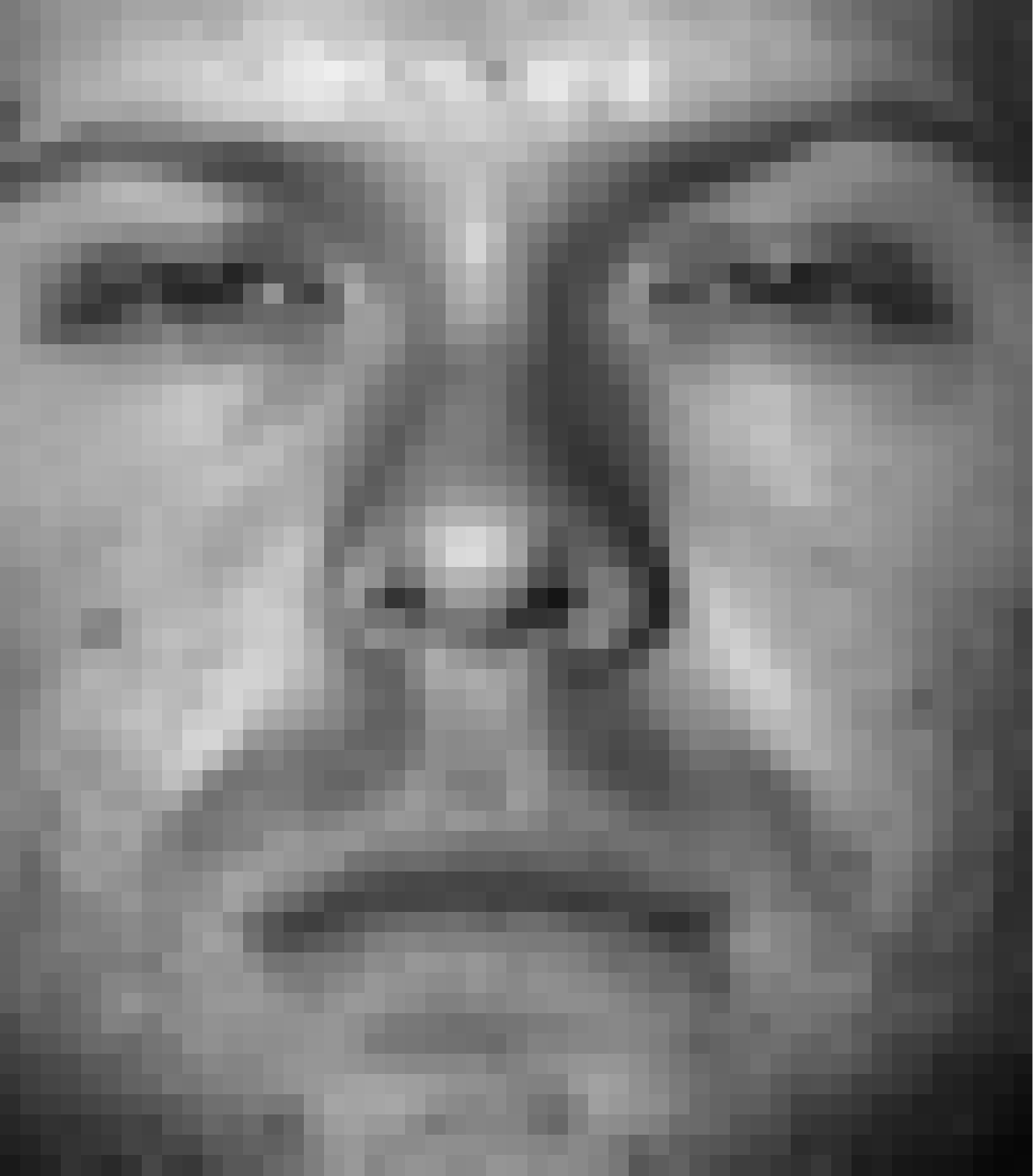}
\includegraphics[width=0.055\textwidth]{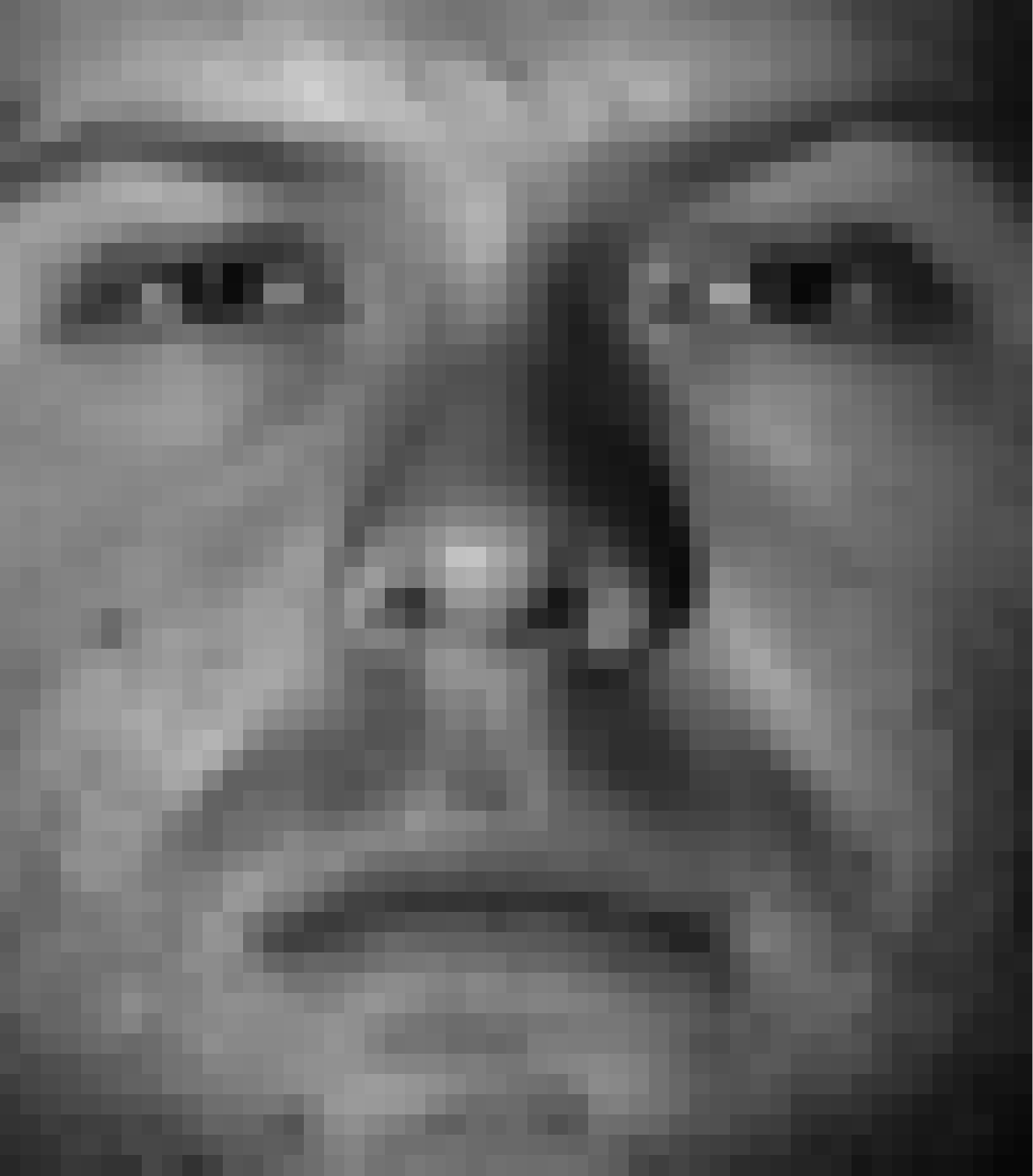}
\includegraphics[width=0.055\textwidth]{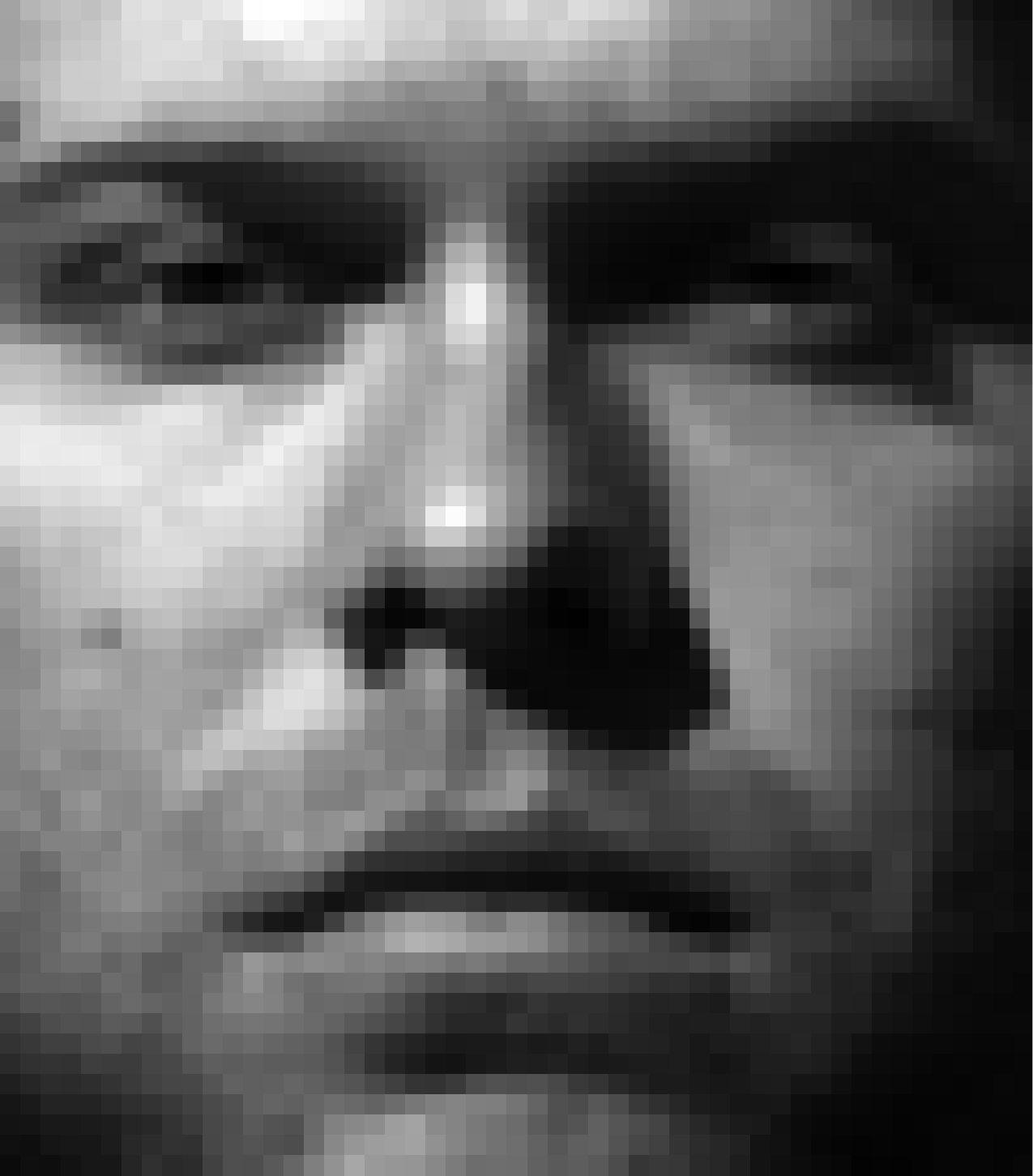}
\includegraphics[width=0.055\textwidth]{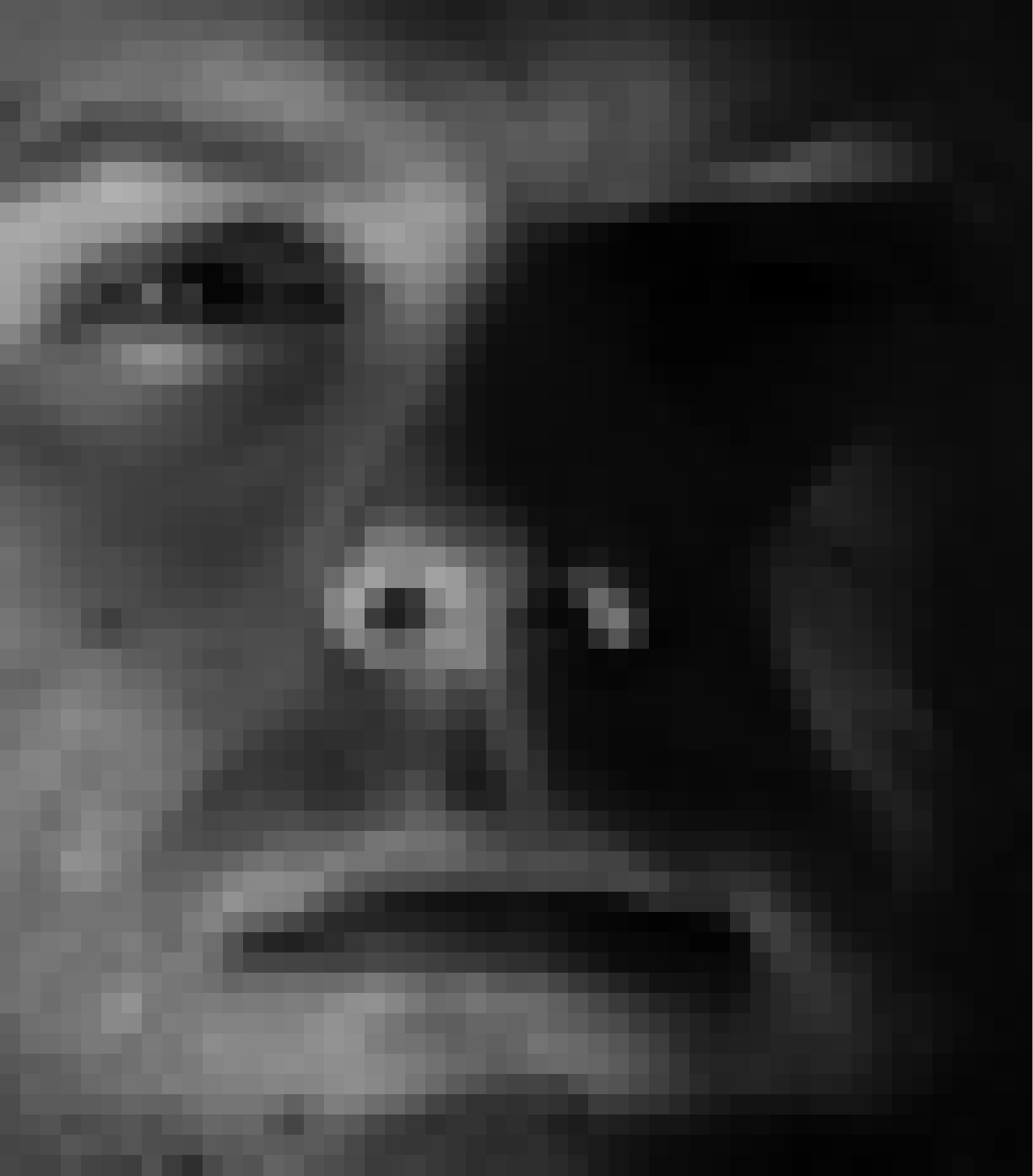}
\includegraphics[width=0.055\textwidth]{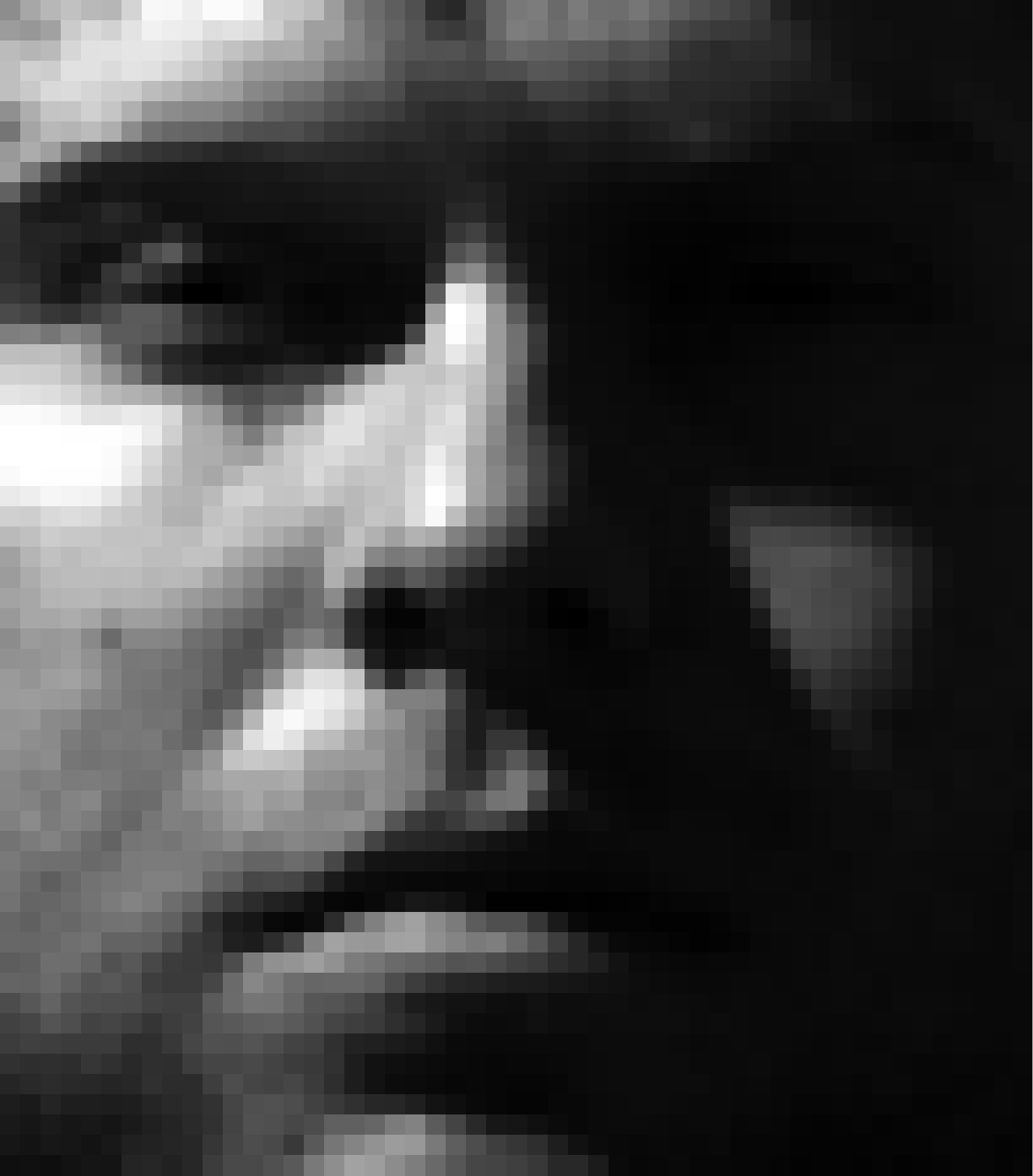}
\includegraphics[width=0.055\textwidth]{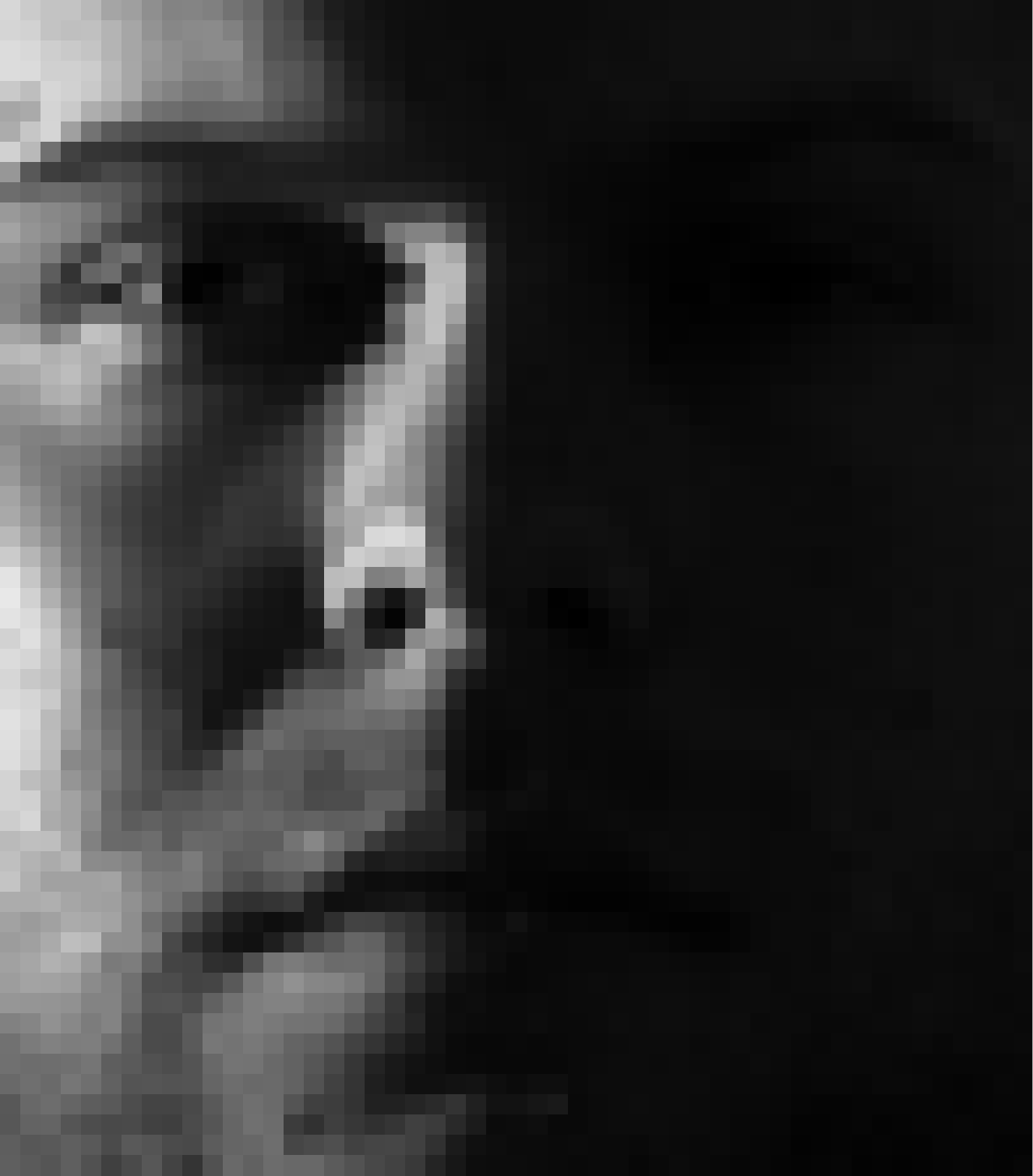}
\includegraphics[width=0.055\textwidth]{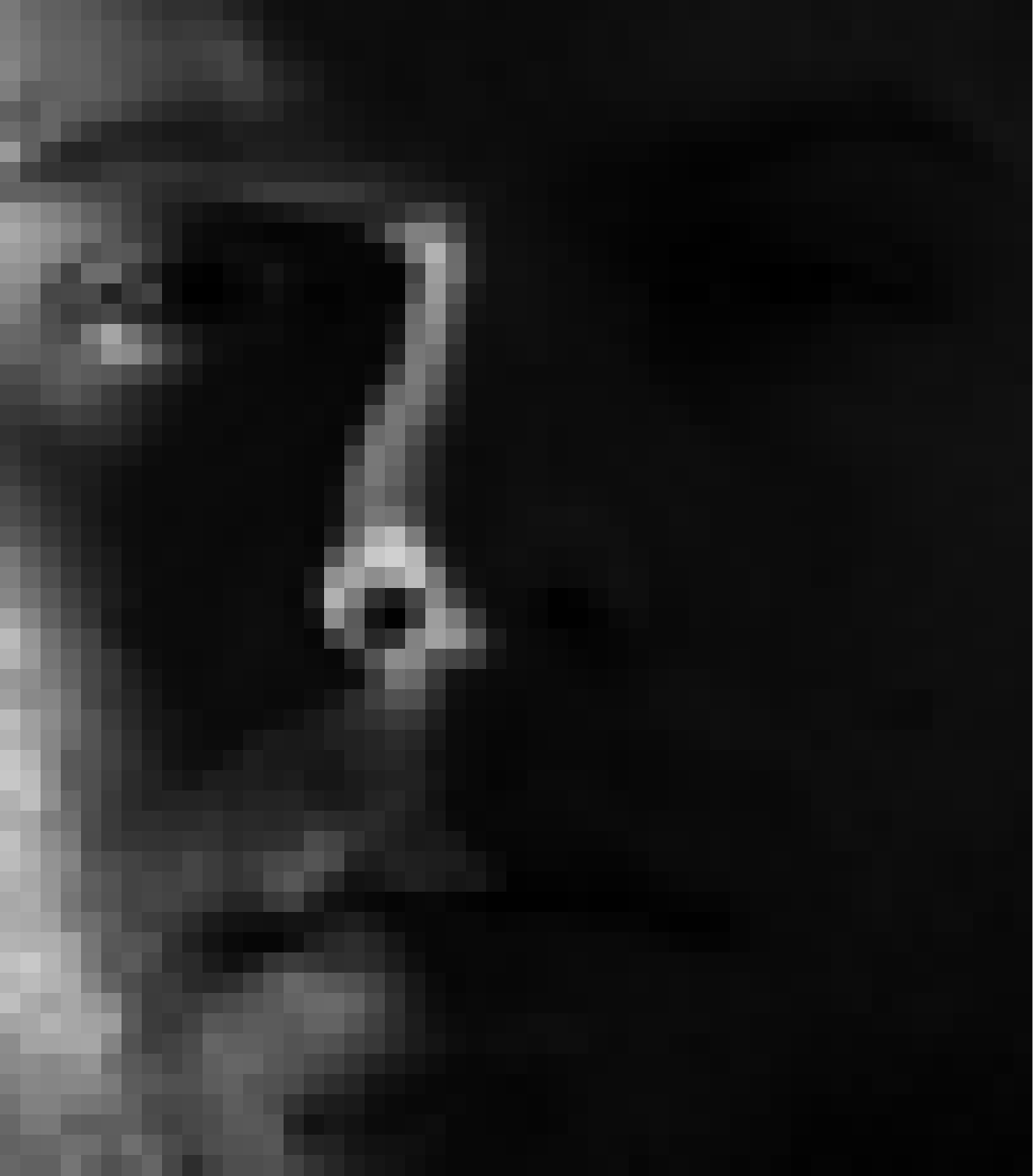}
\includegraphics[width=0.055\textwidth]{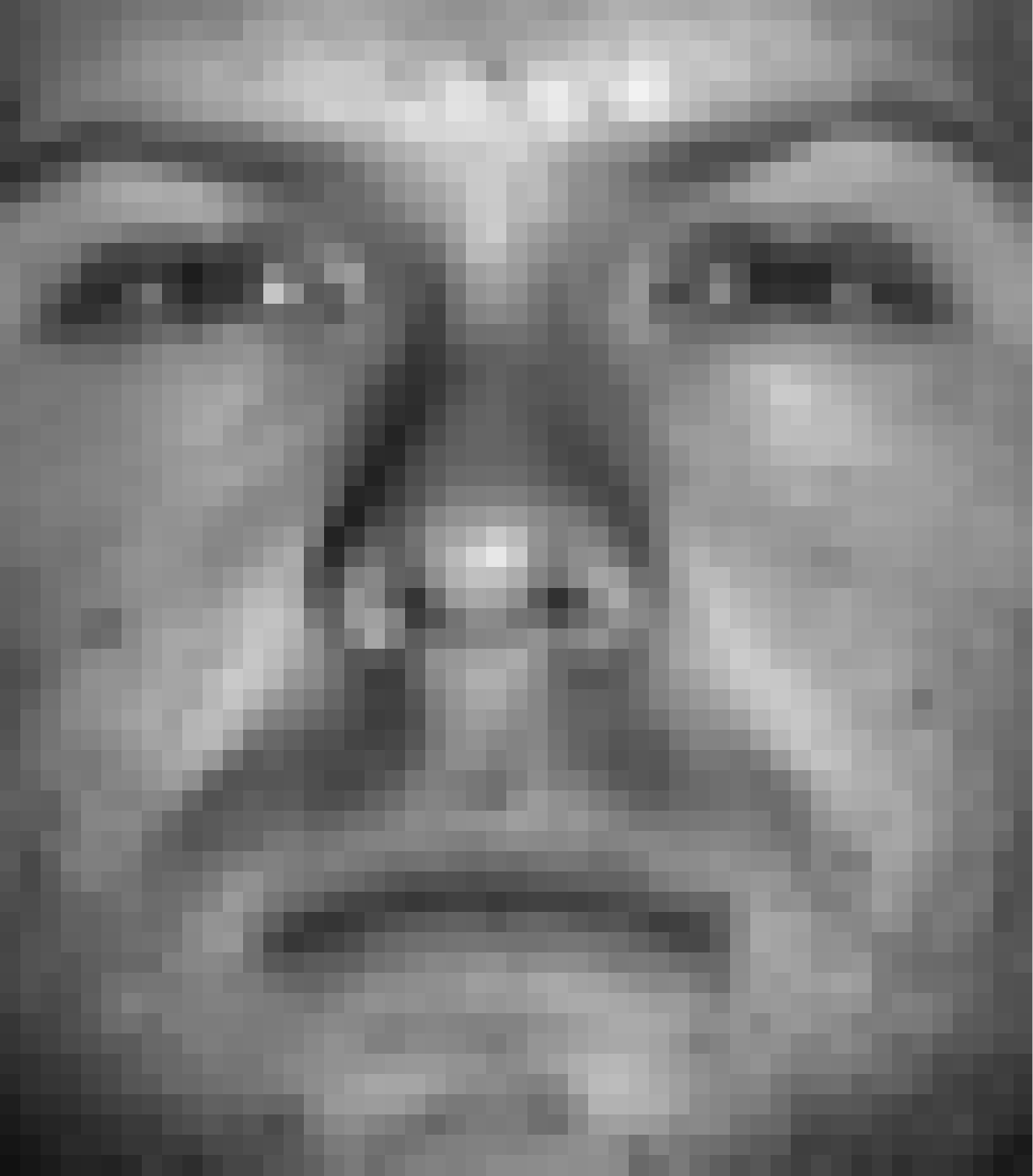}
\includegraphics[width=0.055\textwidth]{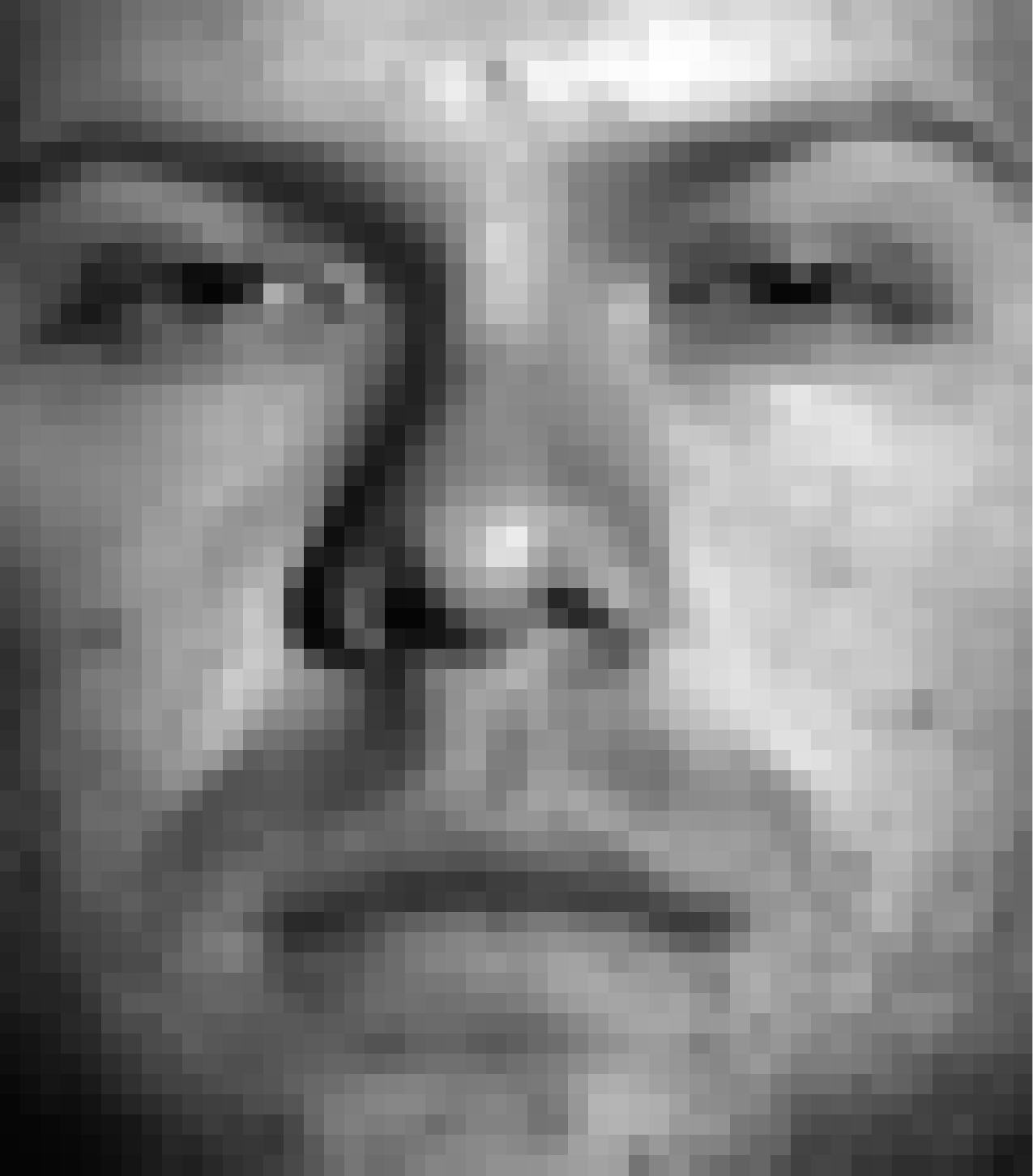}
\includegraphics[width=0.055\textwidth]{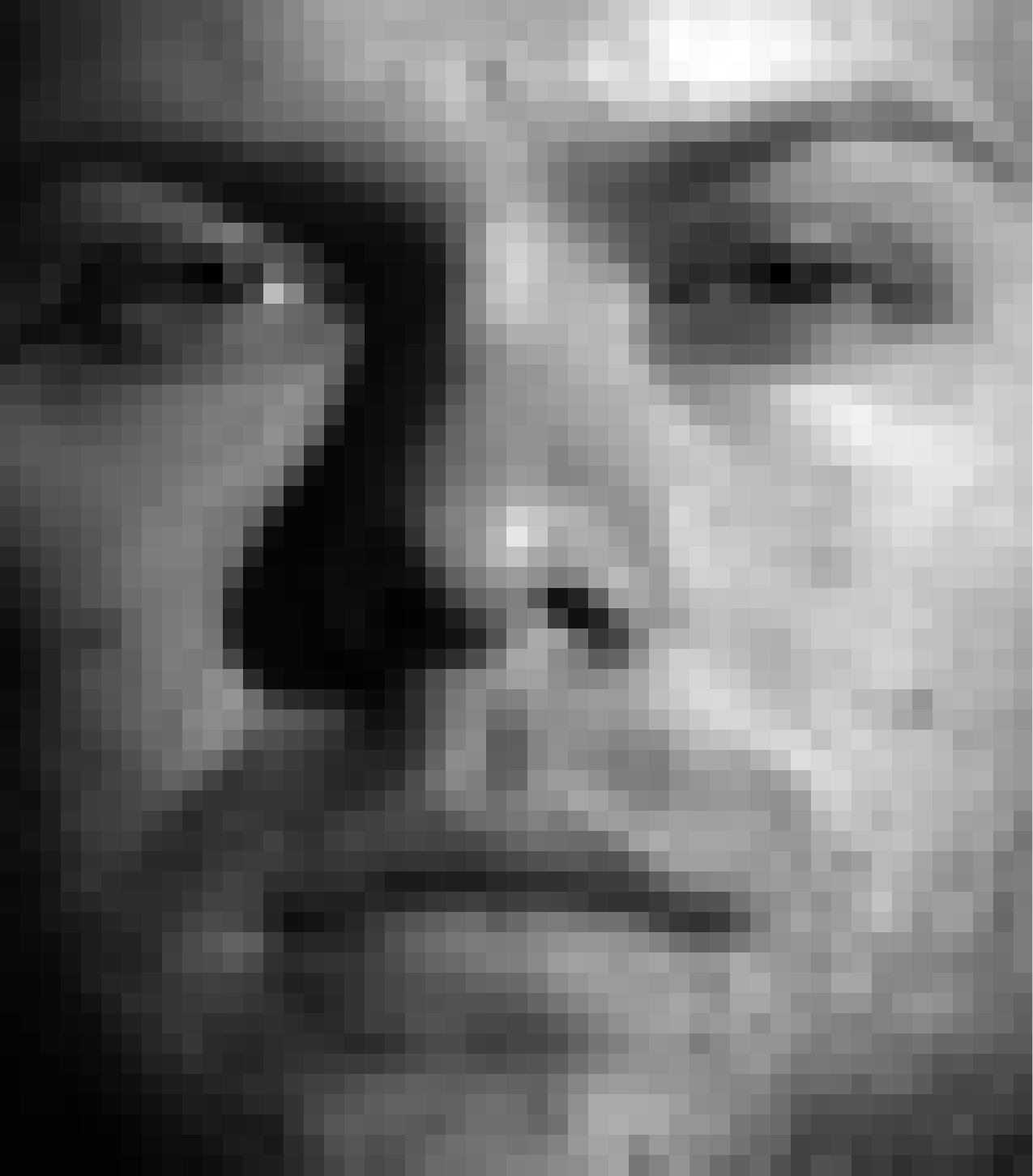}
\includegraphics[width=0.055\textwidth]{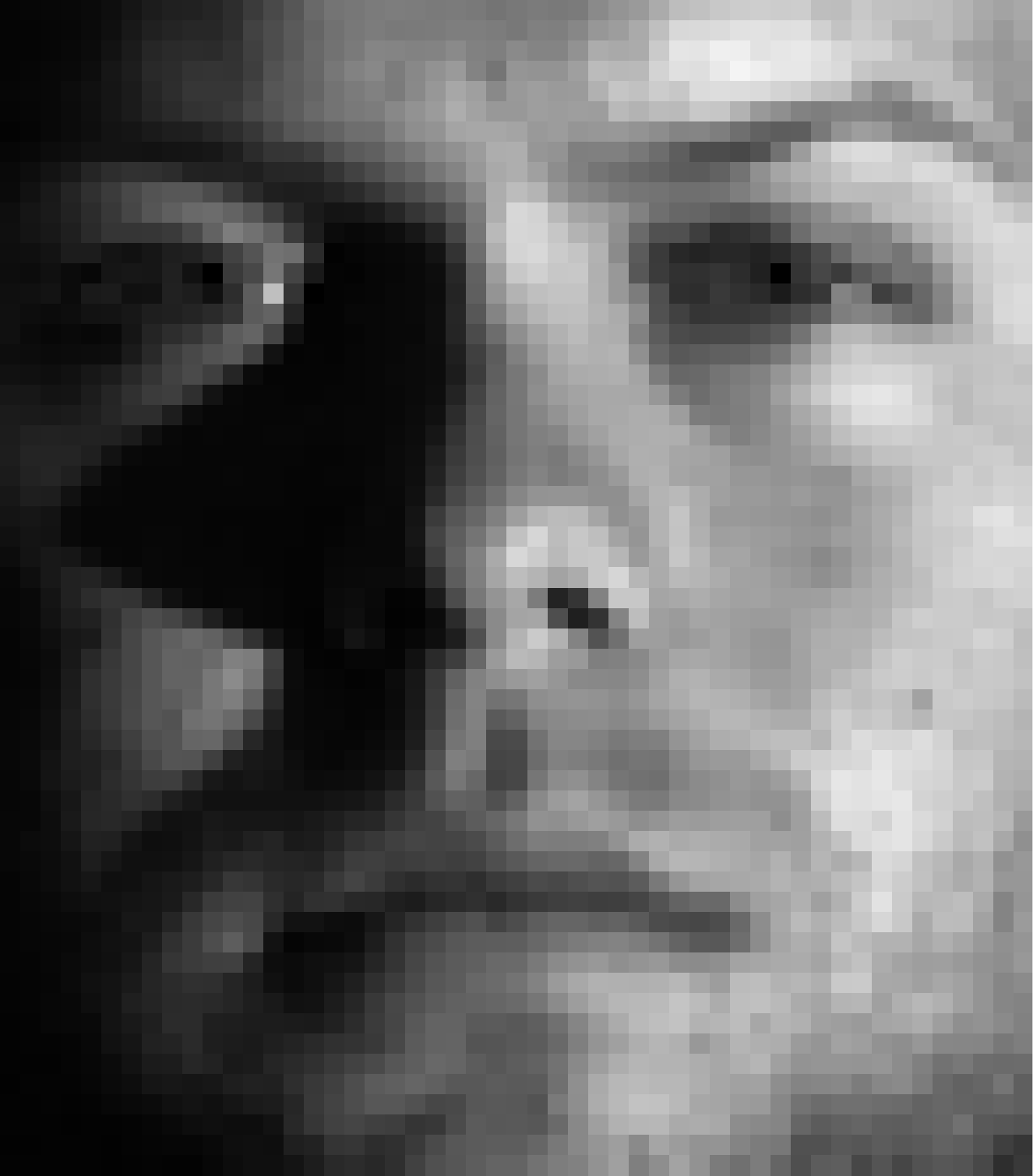}
\includegraphics[width=0.055\textwidth]{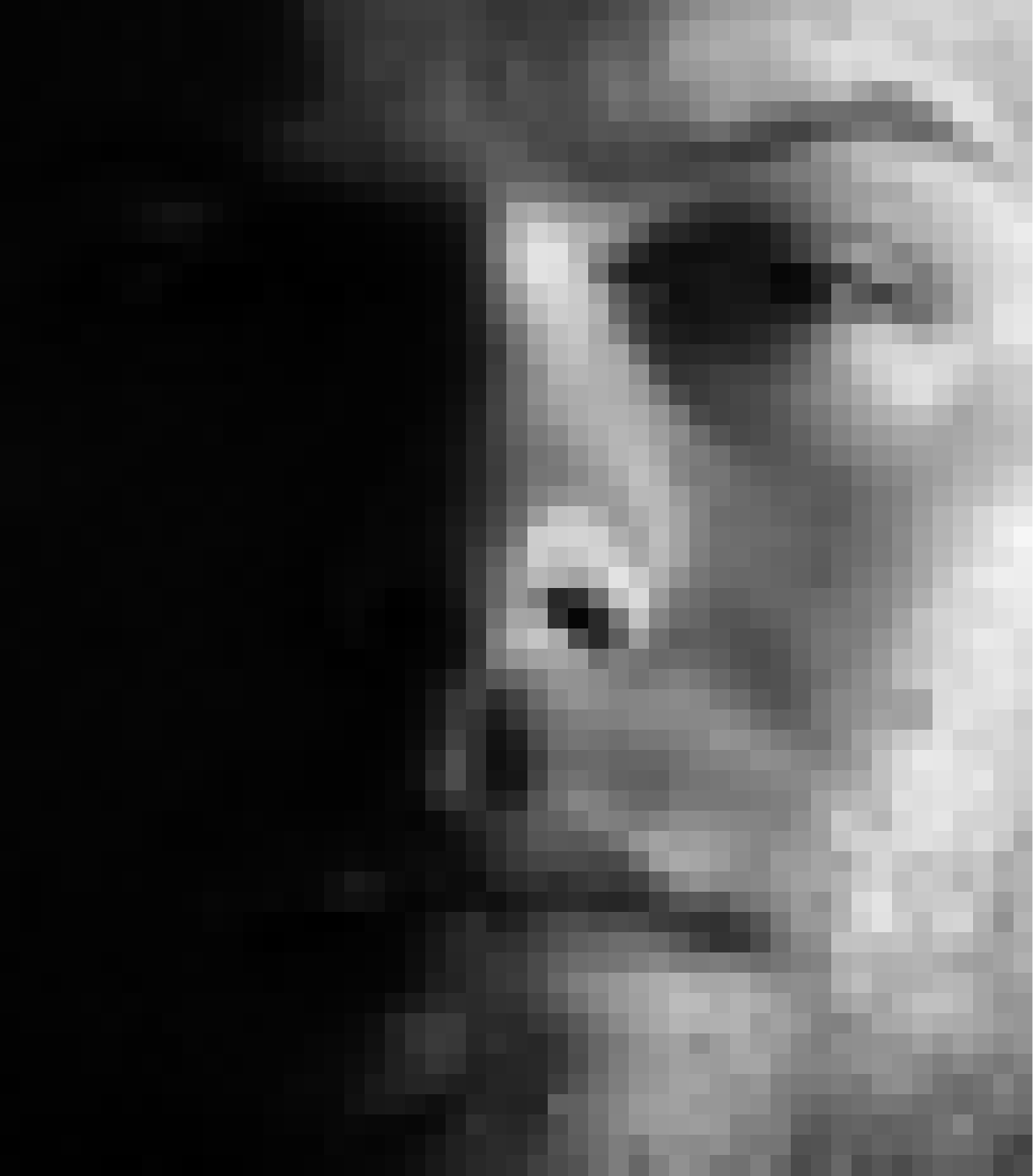}
\includegraphics[width=0.055\textwidth]{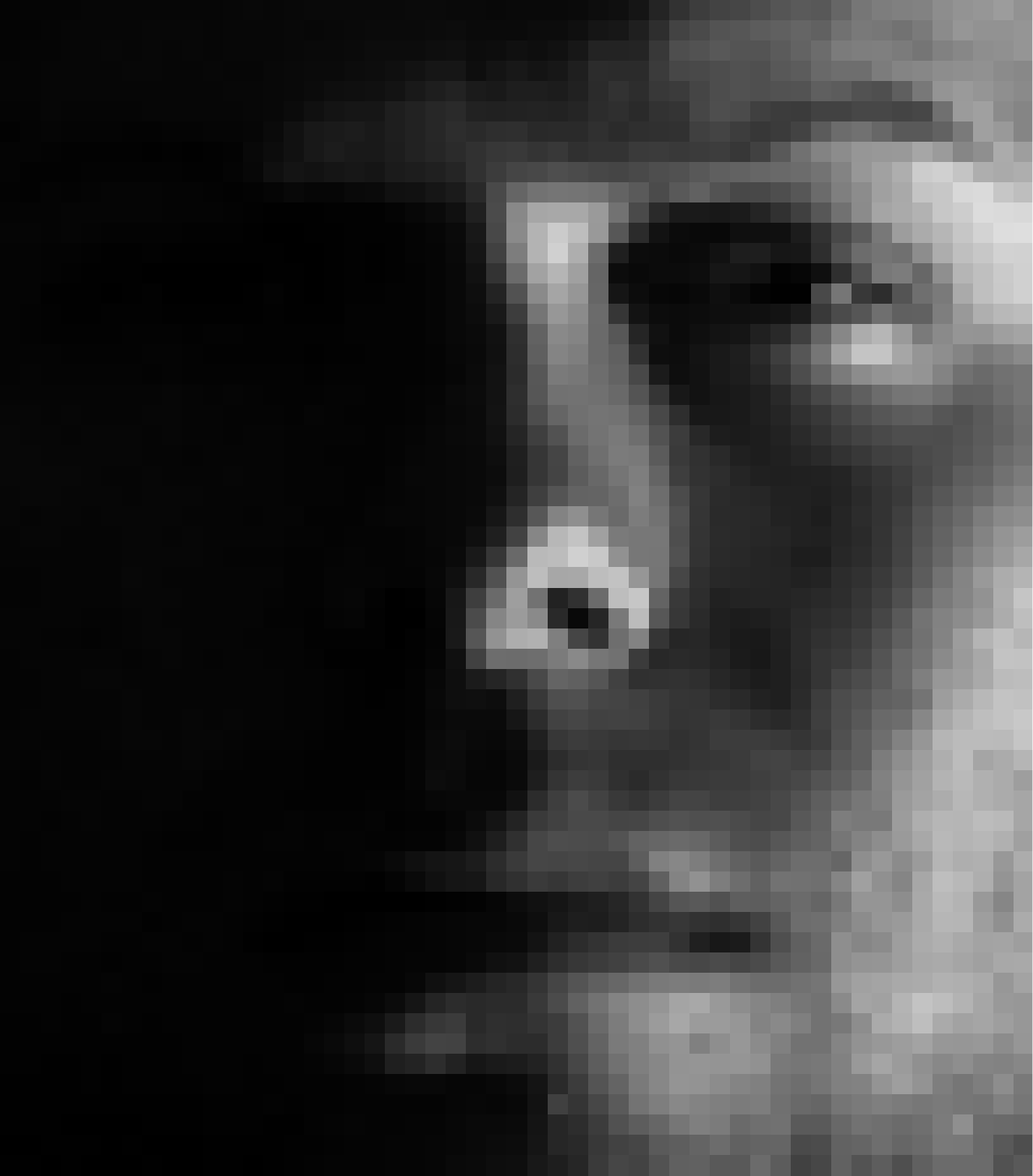}
\includegraphics[width=0.055\textwidth]{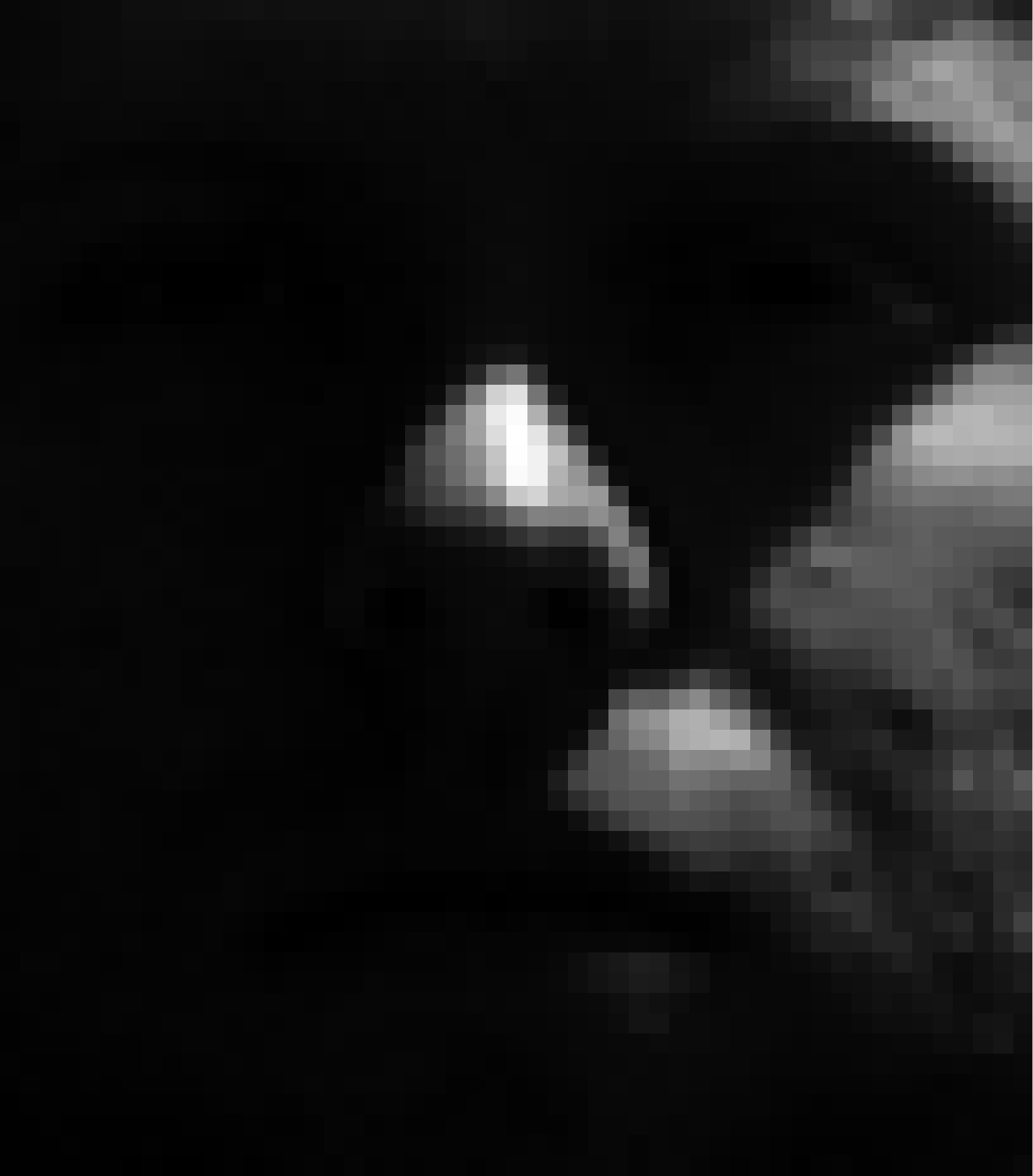}
\end{figure}

We uniformly randomly pick 54 illuminations of images for training and use the remaings for testing. For the training data, we further remove $100(1-\text{SR})$ percent of pixels uniformly at random and then apply the three algorithms to the incomplete training data to get its approximate HOSVD with core size $38\times54\times r$. %To initialize the algorithms, we fill \emph{zero} to missing values and perform HOSVD to the \emph{zero}-imputed tensor. 
Assume the factors of the approximate HOSVD to be $(\hat{\bm{\cC}},\hat{\vA})$. Then $\hat{\vA}_3$ approximately spans the dominant $r$-dimensional subspace of the pixel domain, and $\hat{\vB}_3=\unfold_3(\hat{\bm{\cC}}\times_1\hat{\vA}_1\times\hat{\vA}_2)$ contains the coefficients of images in the training set. For each image $\vx$ in the testing set, we compare the coefficient vector $\hat{\vA}_3^\top\vx$ to each column of $\hat{\vB}_3$ and classify $\vx$ to be the subject corresponding to the closest column. We vary SR among $\{10\%,30\%,50\%,70\%\}$ and $r$ among $\{50,75,100\}$. For each pair of SR and $r$, we repeat the whole process (i.e., randomly choosing training dataset, removing pixels, obtaining approximate HOSVD, and performing classification) 3 times independently and take the average of the classification accuracies (i.e., the ratio of the number of correctly recognized face images over the total testing images) and running time. We compare the tensor face recognition methods to EigenFace, a popular matrix face recognition method, which reshapes each face image into a vector, forms a matrix with each column being a reshaped face image, and then performs eigendecomposition to find a basis for recognition. Table \ref{table:facereg} shows the average results of each algorithm and also accuracies by EigenFace using incomplete training data as a baseline. From the table, we see that all the three algorithms give similar classification accuracies since they solve similar models, and iHOOI performs slightly better than ALSaS and WTucker in most cases. In addition, ALSaS is the fastest while WTucker costs much more time than that by ALSaS and also iHOOI in every case. Note that larger $r$ gives higher accuracies except at $\text{SR}=10\%$ %, larger $r$ gives worse classification. 
because too few observations and large $r$ cause overfitting problem.

\begin{table}\caption{Average classification accuracies (in percentage) and running time (in second) of ALSaS, iHOOI, and WTucker on the extended Yale Database B for face recognition. The numbers in the parenthesis denote the corresponding average running time. The highest accuracies for each pair of SR and $r$ are highlighted in \textbf{bold}.}\label{table:facereg}
\centering
\resizebox{\textwidth}{!}{
\begin{tabular}{|c|cccc|cccc|cccc|}
\hline
& \multicolumn{4}{|c|}{$r=50$} & \multicolumn{4}{|c|}{$r=75$} & \multicolumn{4}{|c|}{$r=100$} \\\hline
SR & ALSaS & iHOOI & WTucker &EigenFace & ALSaS & iHOOI & WTucker &EigenFace & ALSaS & iHOOI & WTucker &EigenFace\\\hline
10\% & 64.12(515) & \textbf{66.31}(1606)  & 63.33(4874) & 4.65 & 62.46(507) & \textbf{65.88}(1506) & 57.28(6366) & 5.35 & 59.47(497) & \textbf{63.51}(1396) & 55.43(5315) &7.46\\ 
 30\% & 73.42(144) & 73.60(197) & \textbf{73.95}(2566) & 11.84 & 78.42(166) & \textbf{78.86}(246) & 78.25(2833) & 17.98 & 80.70(207) & \textbf{80.79}(401) & 80.35(3164) & 21.32\\
 50\% & 71.67(120) & 72.02(164) & \textbf{72.37}(2171) & 21.14 & \textbf{77.98}(136) & \textbf{77.98}(180) & 77.72(2331) & 31.32 & \textbf{81.05}(166) & 80.79(207) & \textbf{81.05}(2457) & 37.98\\
 70\% &76.84(108) & \textbf{76.93}(157) & 75.96(1894) & 57.72 & \textbf{81.49}(131) & 81.40(171) & 81.23(2038) & 65.44 &\textbf{83.60}(155) &83.42(200) & 83.42(2111) & 69.30\\\hline
\end{tabular}}
\end{table}

\subsection{Low-rank tensor completion}
Upon recovering factors $\bm{\cC}$ and $\vA$, one can easily estimate missing entries of the underlying tensor $\bm{\cM}$. When $\bm{\cM}$ has low multilinear rank, it can be exactly reconstructed under certain conditions (see \cite{huang2014provable, yuan2014tensor} for example). In this subsection, we test ALSaS and iHOOI on reconstructing (approximately) low-multilinear-rank tensors and compare them to WTucker and two other state-of-the-art methods: TMac \cite{tmac2015} and geomCG \cite{kressner2013low}. We choose TMac and geomCG for comparison because their codes are publicly available and also they have been shown superior over several other tensor-completion methods including FaLRTC \cite{liu2013tensor} and SquareDeal \cite{mu2013square}.  
%The code of all these methods is publicly available. 
TMac is an alternating least squares method for solving the so-called parallel matrix factorization model:
\begin{equation}\label{eq:tmac}
\begin{array}{l}
\underset{\vX,\vY,\bm{\cZ}}\min \sum_{n=1}^N\alpha_n\|\vX_n\vY_n^\top-\unfold_n(\bm{\cZ})\|_F^2,\\[0.1cm]
 \st \cP_\Omega(\bm{\cZ}) = \cP_\Omega(\bm{\cM}),\, \vX_n\in\RR^{m_n\times r_n}, \vY_n\in\RR^{(\Pi_{i\neq n}m_i)\times r_n}, \forall n,
\end{array}
\end{equation}
and geomCG is a Riemannian conjugate gradient method for
\begin{equation}\label{eq:geom}
\min_{\bm{\cX}}\frac{1}{2}\|\cP_\Omega(\bm{\cX})-\cP_\Omega(\bm{\cM})\|_F^2, \st \rankk\big(\unfold_n(\bm{\cX})\big)=r_n,\, n=1,\ldots,N.
\end{equation}
In \eqref{eq:tmac}, $\sum_n\alpha_n=1$, and $\alpha_n$ acts as a weight on the $n$-th mode fitting and can be adaptively updated. Usually, better low-rankness property leads to better data fitting, and larger $\alpha_n$ is put. Since all tensors tested in this subsection are balanced and have similar low-rankness along each mode, we simply fix $\alpha_n=\frac{1}{N},\,\forall n$. Both \eqref{eq:tmac} and \eqref{eq:geom} require estimation on $r_n$'s, which can be either fixed or adaptively adjusted in a similar way as that for ALSaS and iHOOI. 
First, we compare the recoverability of the five methods on randomly generated low-multilinear-rank tensors. Then, we test their accuracies and efficiency on reconstructing a 3D MRI image. 

\subsubsection*{Phase transition plots}  A phase transition plot uses greyscale colors to depict how likely a certain kind of low-multilinear-rank tensors can be recovered by an algorithm for a range of different ranks and sample ratios. Phase transition plots are important means to compare the performance of different tensor recovery methods.

We use two random datasets in the test. Each tensor in both sets is $50\times50\times50$ and has the form of $\bm{\cM}=\bm{\cC}\times_1\vA_1\times_2\vA_2\times_3\vA_3$. In the first dataset, entries of $\bm{\cC}$ and $\vA_n$'s follow standard normal distribution: $\bm{\cC}$ is generated by MATLAB command \verb|randn(r,r,r)| and $\vA_n$ by \verb|randn(50,r)| for $n=1,2,3$. In the second dataset, entries of $\bm{\cC}$ follow uniform distribution, and $\vA_n$'s have power-law decaying singular values: $\bm{\cC}$ is generated by command \verb|rand(r,r,r)| and $\vA_n$ by \verb|orth(randn(50,r))*diag([1:r].^(-0.5))| for $n=1,2,3$. Usually, the tensors in the second dataset is more difficult to recover than those in the first one. For both datasets, we generate $\Omega$ uniformly at random, and we vary $r$ from 5 to 35 with increment 3 and SR from 10\% to 90\% with increment 5\%. We regard the recovery $\bm{\cM}^{rec}$ to be successful, if
$$\mathrm{relerr}(\bm{\cM}^{rec}):=\frac{\|\bm{\cM}^{rec}-\bm{\cM}\|_F}{\|\bm{\cM}\|_F}\le 10^{-2}.$$

We apply rank-increasing and rank-fixing strategies to ALSaS, iHOOI, TMac and geomCG, and we append ``-inc'' and ``-fix'' respectively after the name of each algorithm to specify which strategy is applied. The rank-increasing strategy initializes $r_n=1$ for all four algorithms and sets $r_n^{\max}=50,\forall n$ for the former three and $r_n^{\max}=r,\forall n$ for geomCG, and the rank-fixing strategy fixes $r_n=r,\, n=1,2,3$. We test WTucker with true ranks (WTucker-true) by fixing $r_n=r,\forall n$ and also rank over estimation (WTucker-over) by fixing $r_n=r+10,\forall n$. All the algorithms are provided with the same random starting points. For each pair of $r$ and SR, we run all the algorithms on 50 independently generated low-multilinear-rank tensors. To save testing time, we simply regard that if an algorithm succeeds 50 times at some pair of $r$ and SR, it will alway succeeds at this $r$ for larger SR's, and if it fails 50 times at some pair of $r$ and SR, it will never succeed at this SR for larger $r$'s. Figure \ref{fig:phase-3G} depicts the phase transition plot of each method on the Gaussian randomly generated tensors and Figure \ref{fig:phase-3plaw} on random tensors with power-law decaying singular values. From the figures, we see that ALSaS and iHOOI performs comparably well to TMac with both rank-increasing and rank-fixing strategies and also to geomCG with rank-increasing strategy. In addition, even without knowing the true ranks, ALSaS and iHOOI by adaptively increasing rank estimates can perform as well as those by assuming true ranks. Both ALSaS and iHOOI successfully recover more low-multilinear-rank tensors than WTucker, and geomCG with rank-fixing strategy.

\begin{figure}\caption{Phase transition plots of different methods on \textbf{3-way random tensors whose factors have Gaussian random entries}. ALSaS and iHOOI are the proposed methods; TMac \cite{tmac2015} solves \eqref{eq:tmac}; geomCG \cite{kressner2013low} solves \eqref{eq:geom}; WTucker \cite{filipovic2013tucker} solves a model similar to \eqref{eq:main1} without orthogonality constraint on $\vA_n$'s.}\label{fig:phase-3G}
\centering
{\footnotesize
\begin{tabular}{ccccc}
ALSaS-inc & iHOOI-inc & TMac-inc & geomCG-inc & WTucker-over\\
\includegraphics[width=0.17\textwidth]{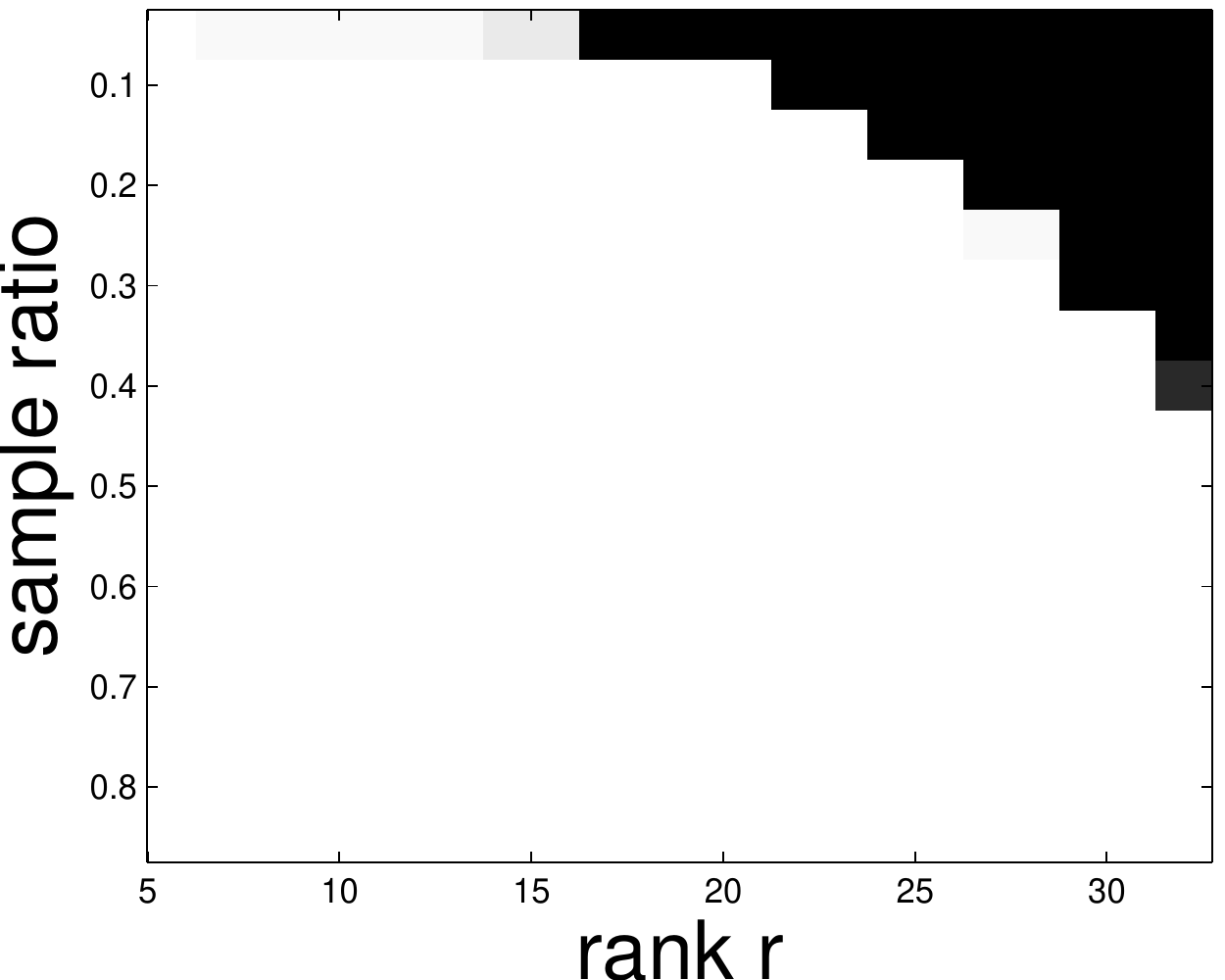} &
\includegraphics[width=0.17\textwidth]{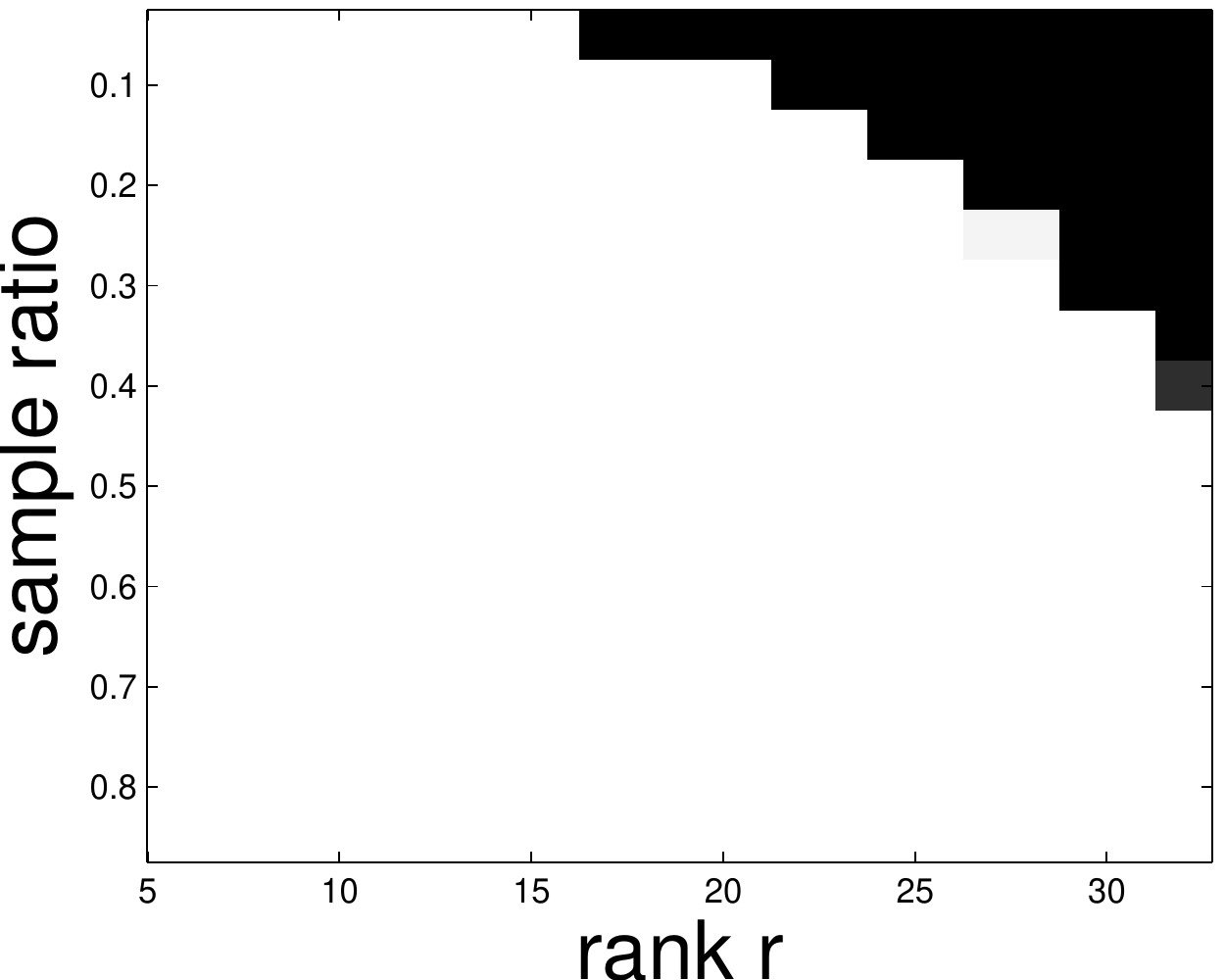} &
\includegraphics[width=0.17\textwidth]{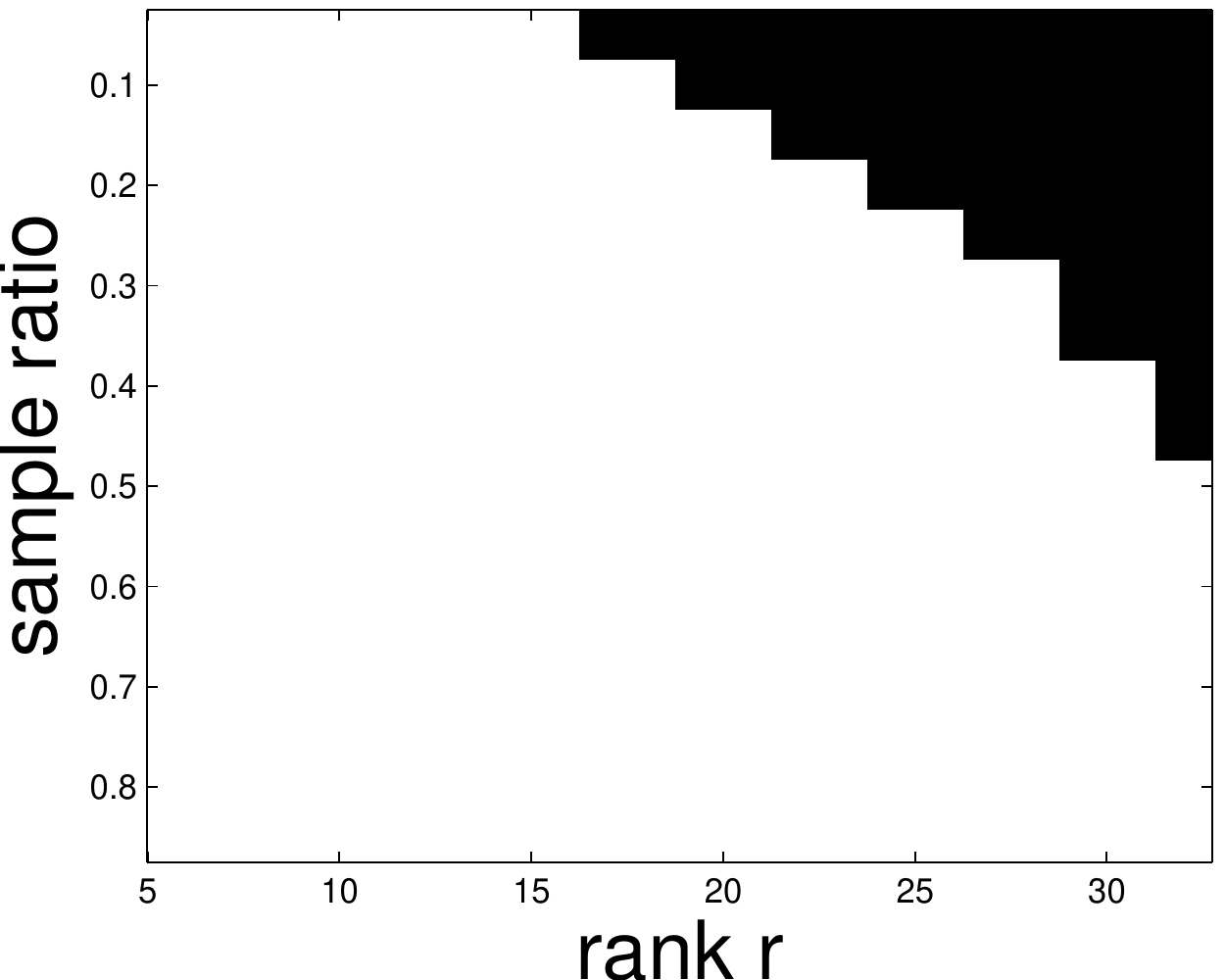} &
\includegraphics[width=0.17\textwidth]{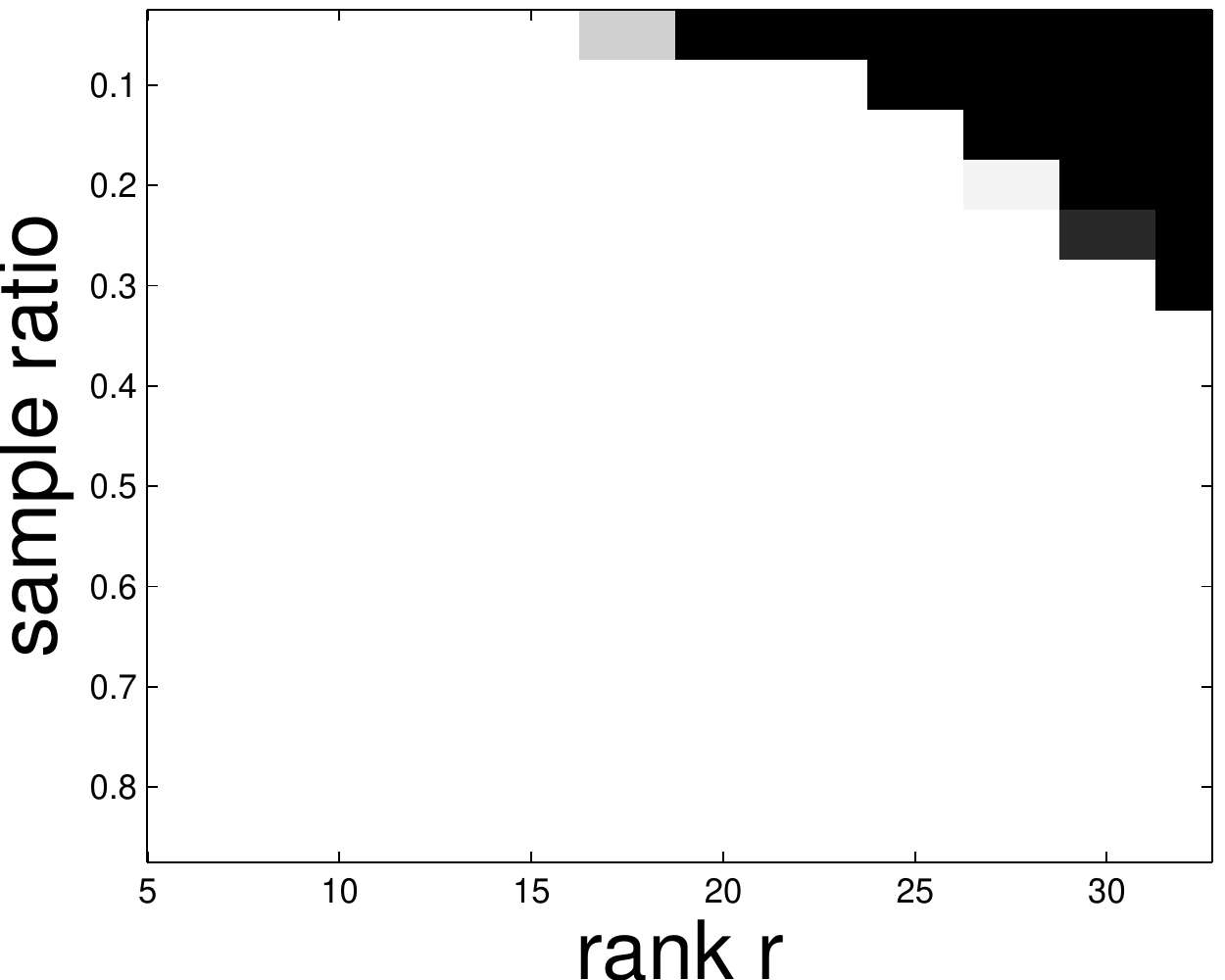} &
\includegraphics[width=0.17\textwidth]{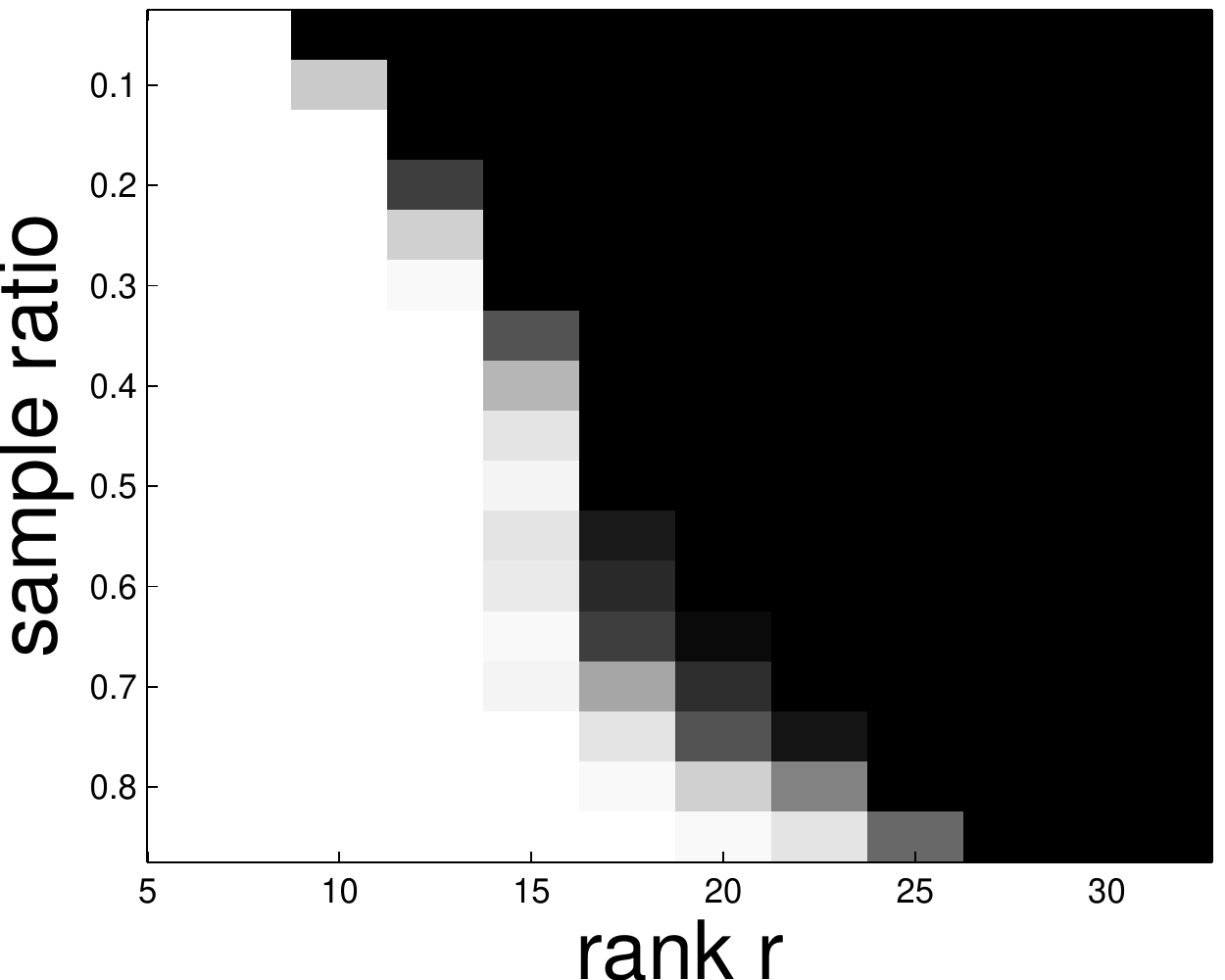}\\
ALSaS-fix & iHOOI-fix & TMac-fix & geomCG-fix & WTucker-true \\
\includegraphics[width=0.17\textwidth]{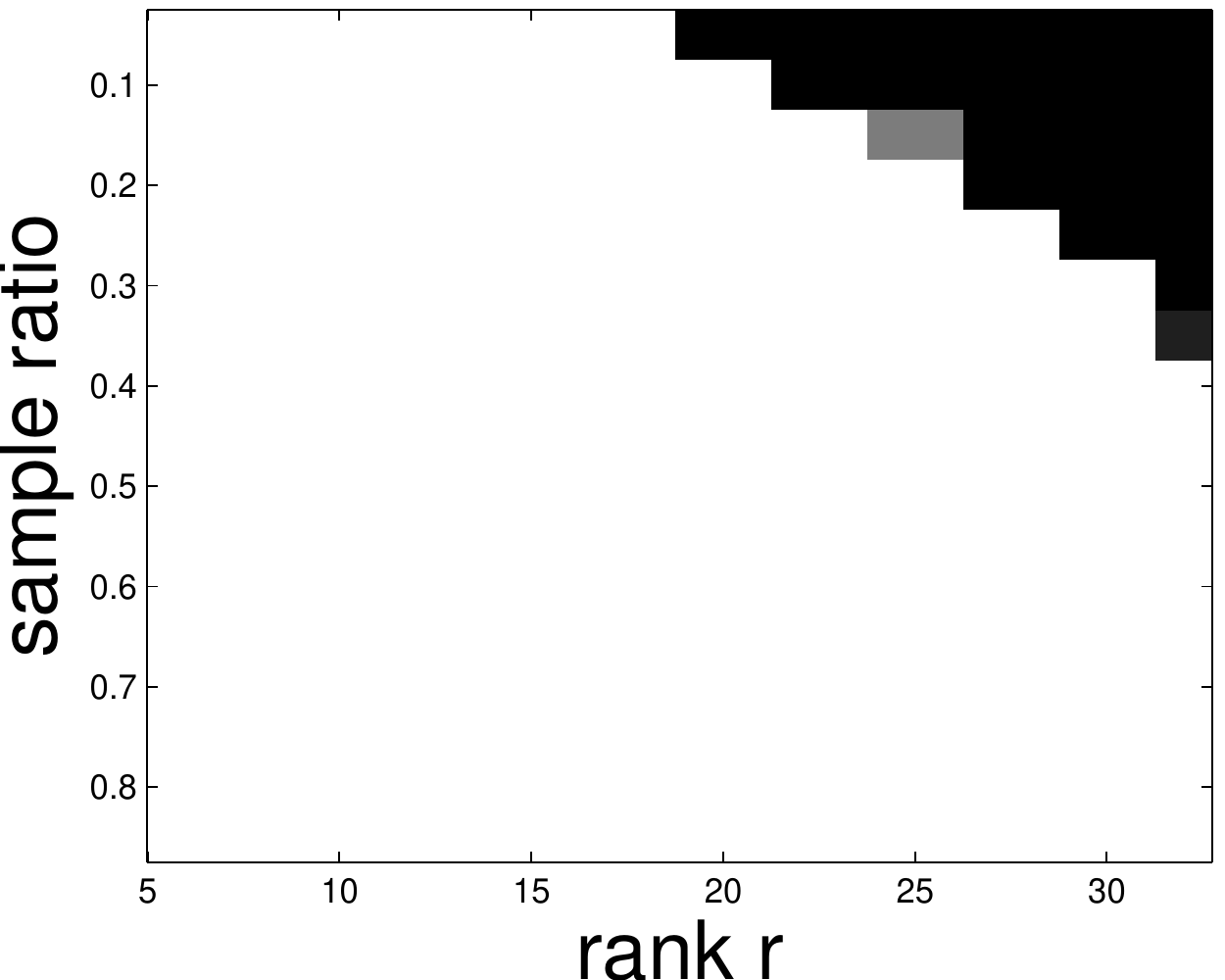} &
\includegraphics[width=0.17\textwidth]{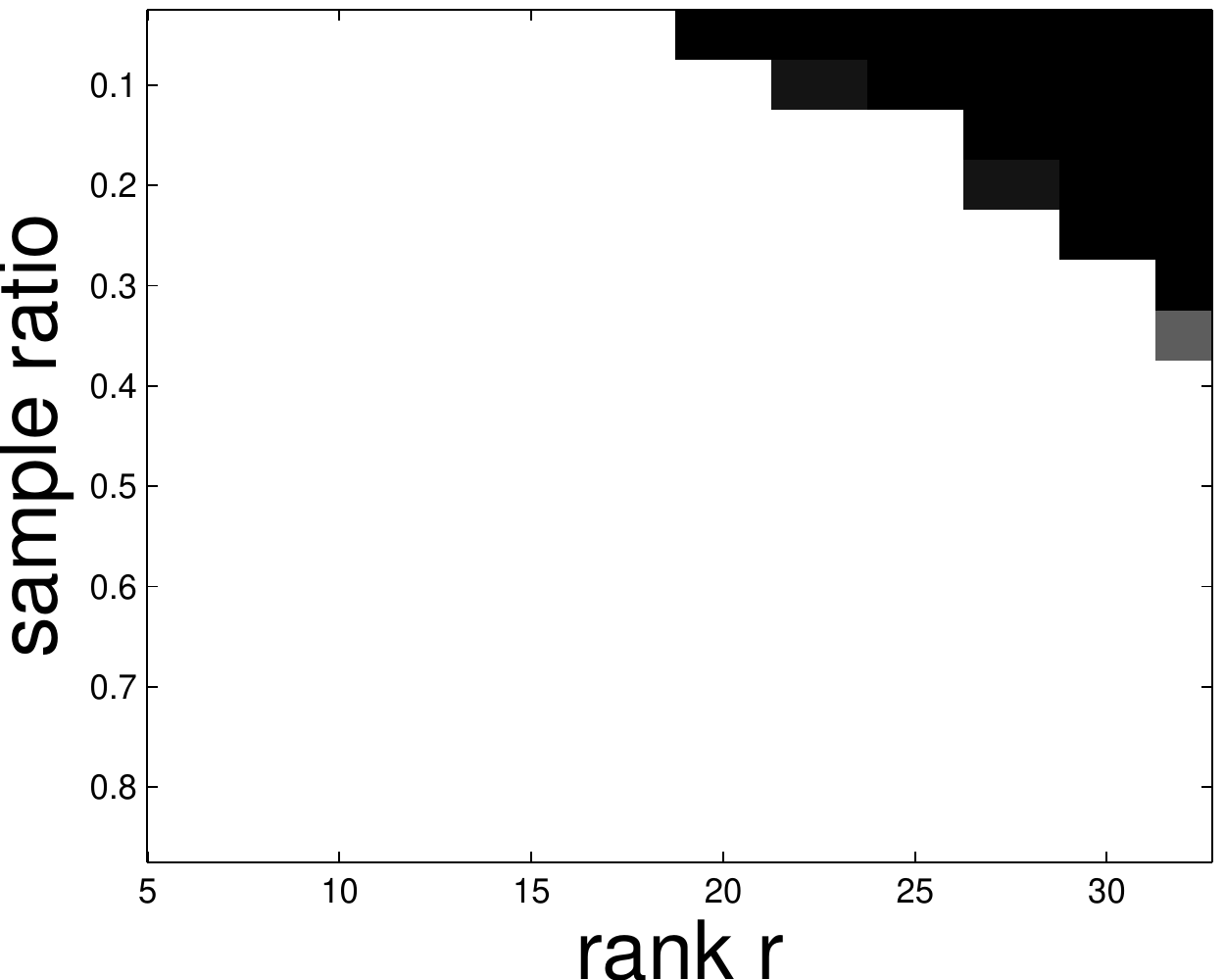} &
\includegraphics[width=0.17\textwidth]{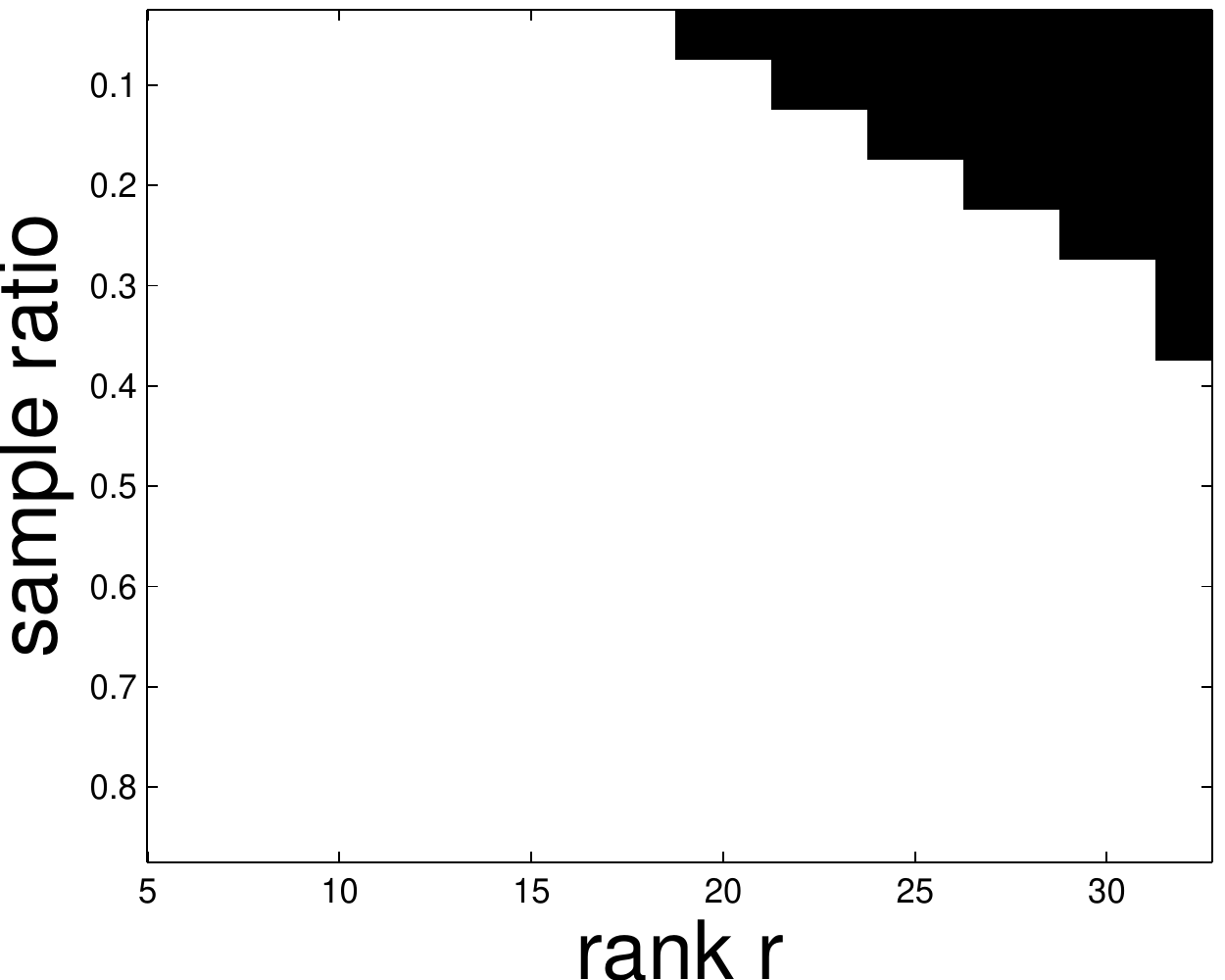} &
\includegraphics[width=0.17\textwidth]{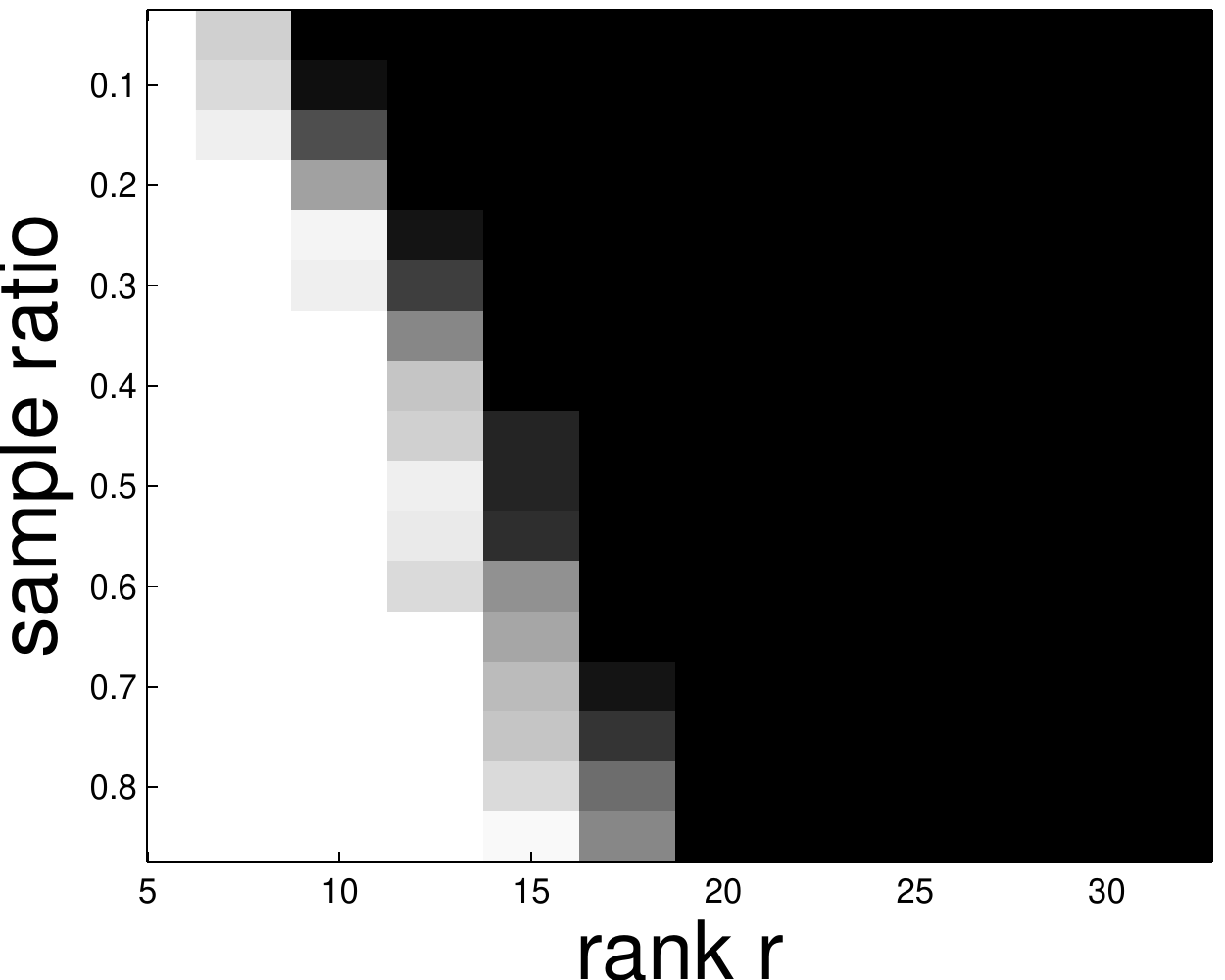} &
\includegraphics[width=0.17\textwidth]{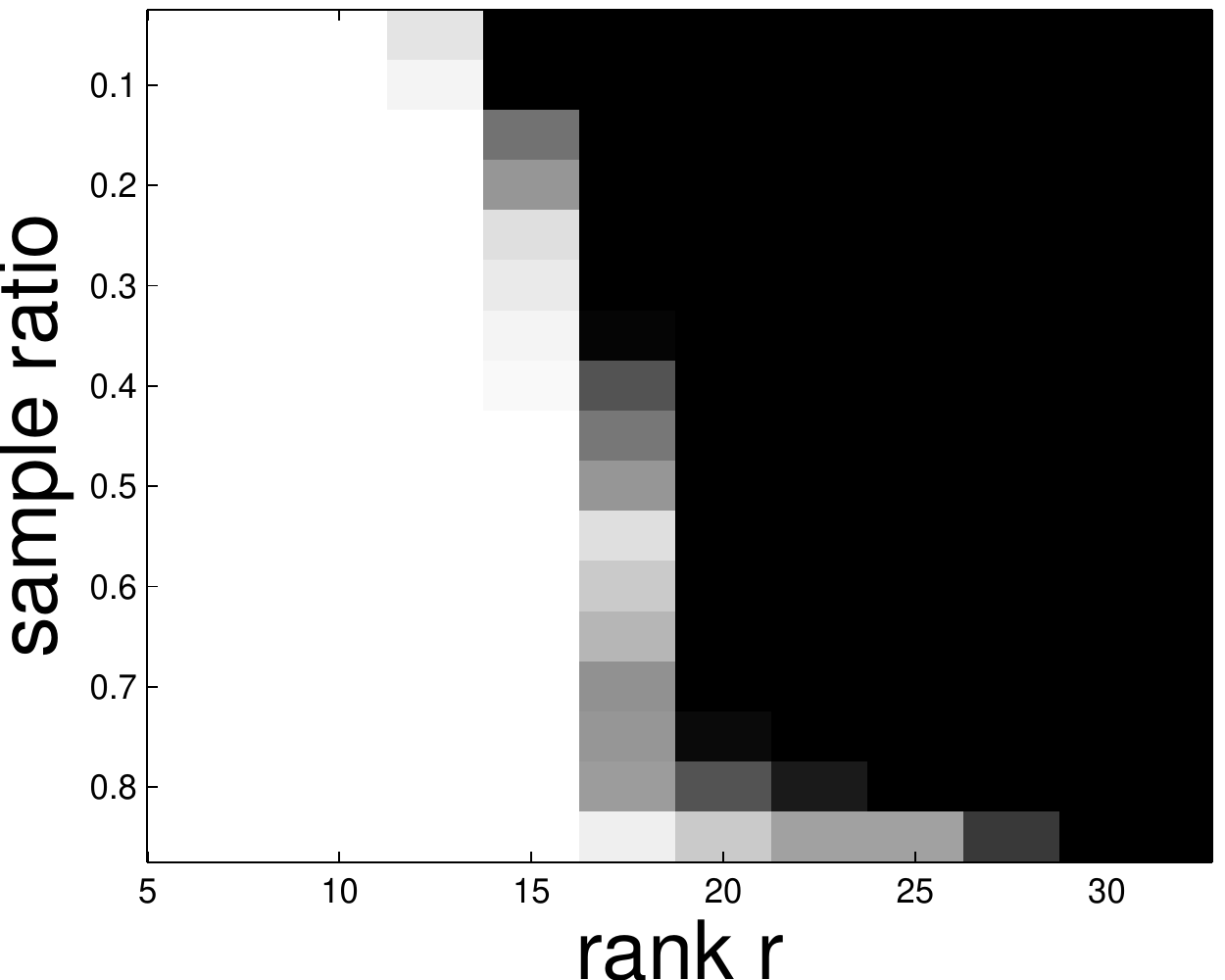}
\end{tabular}}
\end{figure}

\begin{figure}\caption{Phase transition plots of different methods on \textbf{3-way random tensors whose factors have power-law decaying singular values}. ALSaS and iHOOI are the proposed methods; TMac \cite{tmac2015} solves \eqref{eq:tmac}; geomCG \cite{kressner2013low} solves \eqref{eq:geom}; WTucker \cite{filipovic2013tucker} solves a model similar to \eqref{eq:main1} without orthogonality constraint on $\vA_n$'s.}\label{fig:phase-3plaw}
\centering
{\footnotesize
\begin{tabular}{ccccc}
ALSaS-inc & iHOOI-inc & TMac-inc & geomCG-inc & WTucker-over\\
\includegraphics[width=0.17\textwidth]{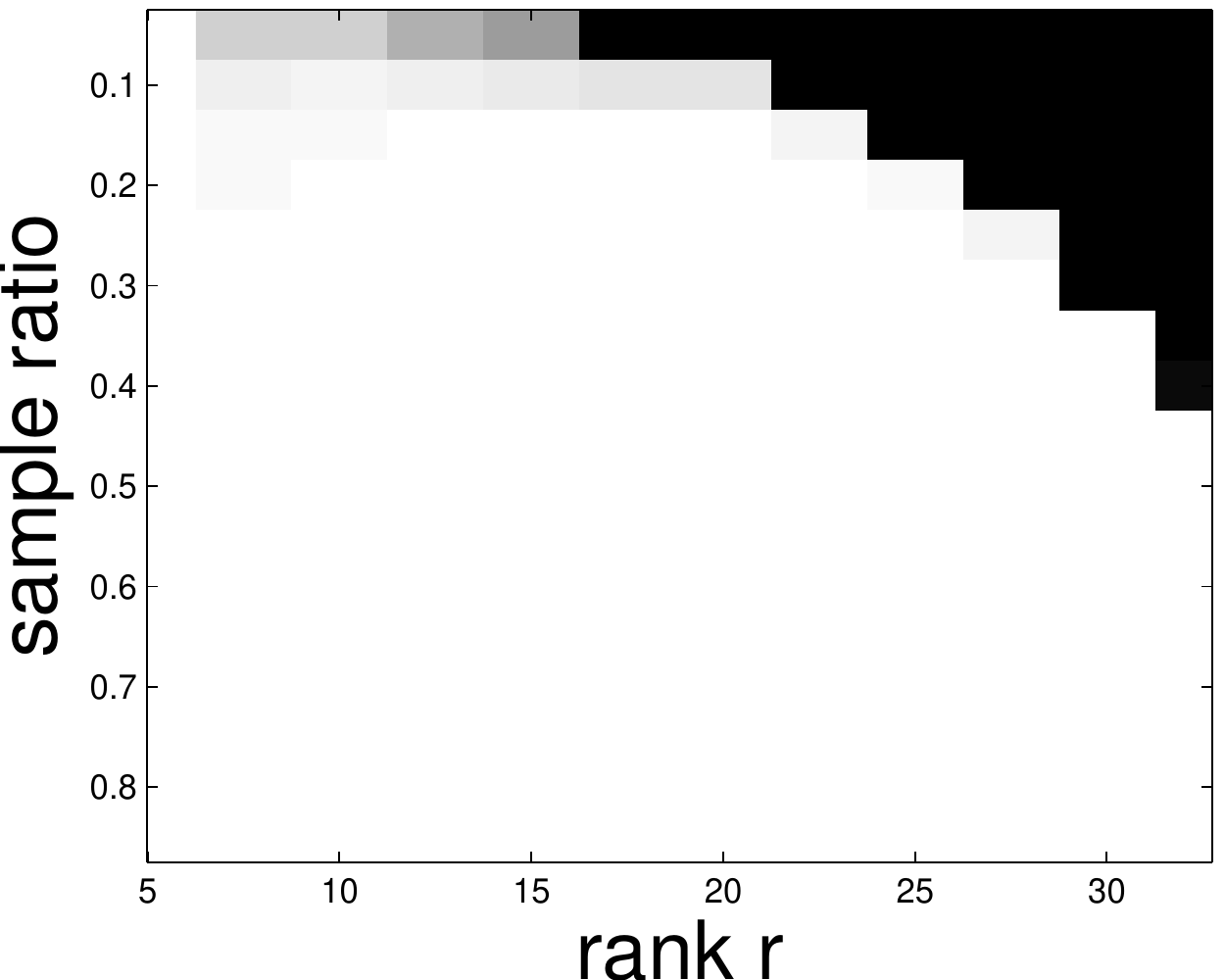} &
\includegraphics[width=0.17\textwidth]{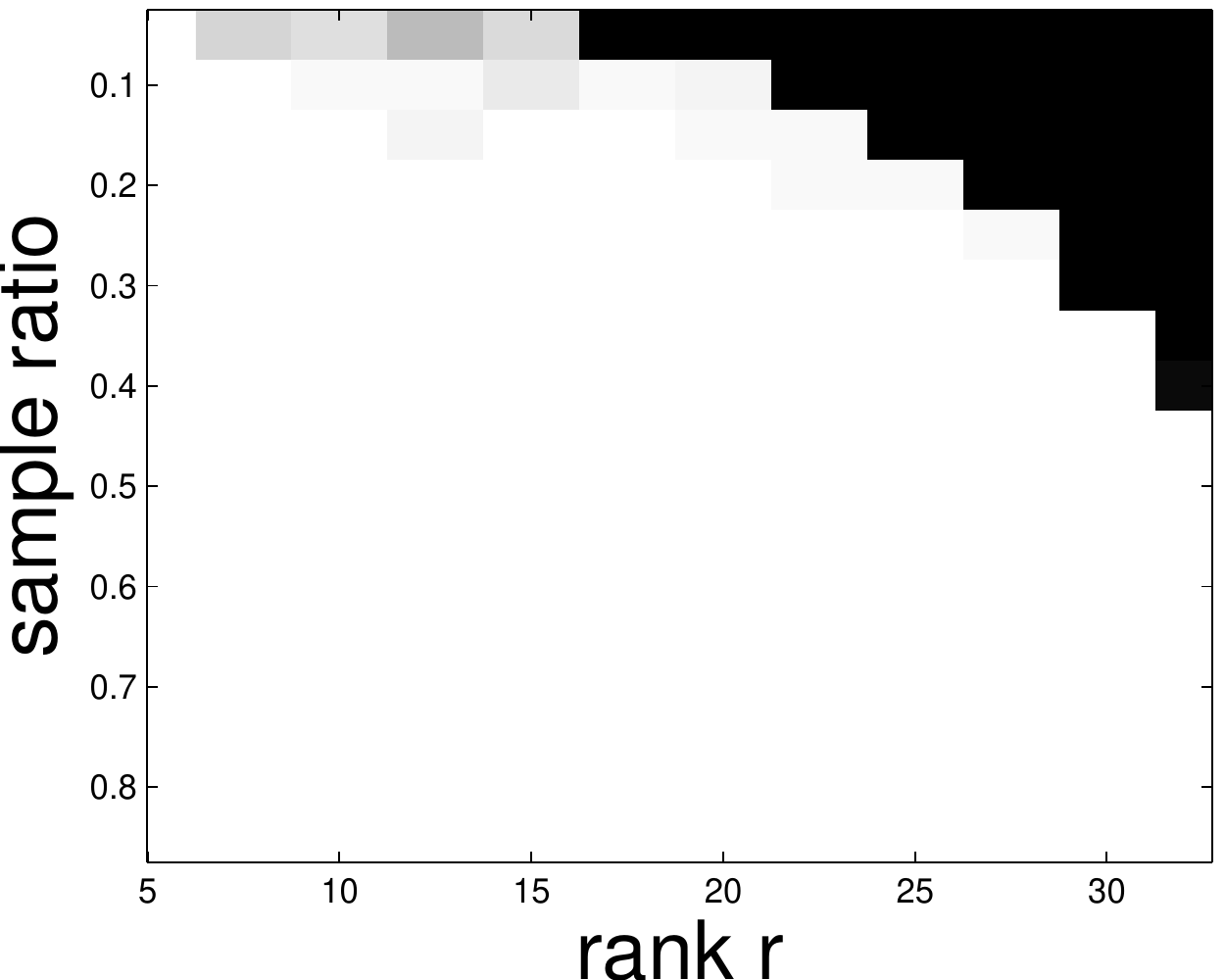} &
\includegraphics[width=0.17\textwidth]{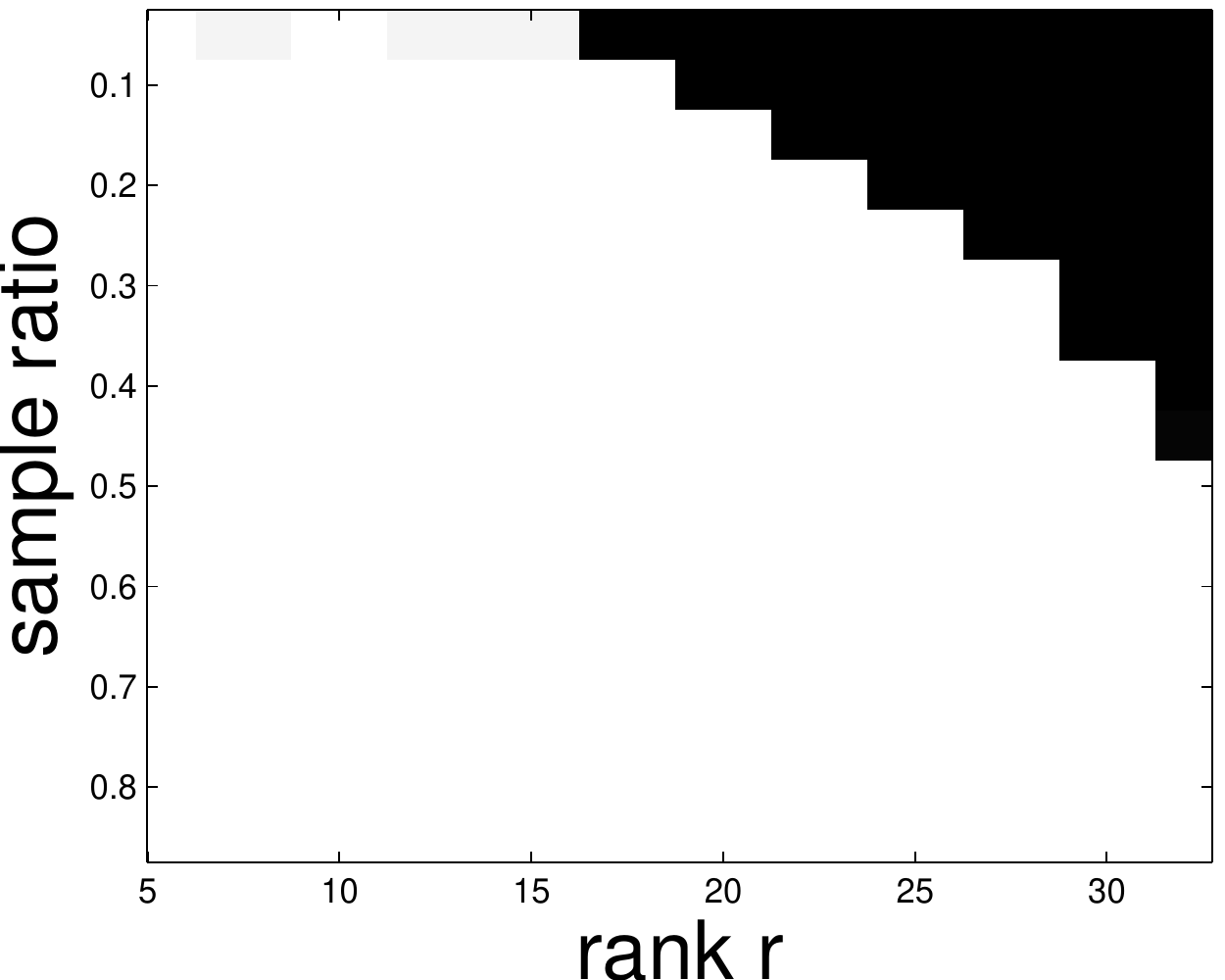} &
\includegraphics[width=0.17\textwidth]{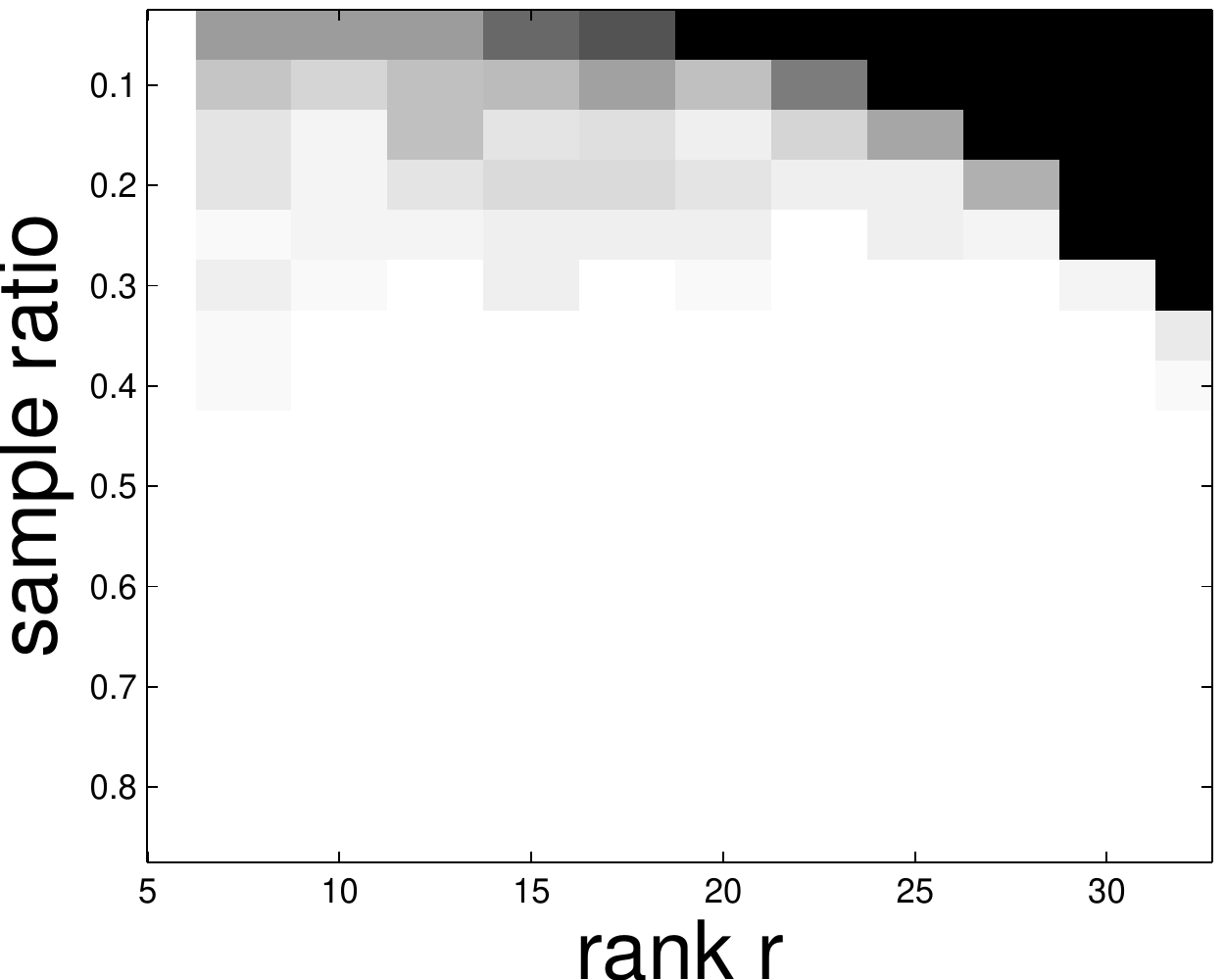} &
\includegraphics[width=0.17\textwidth]{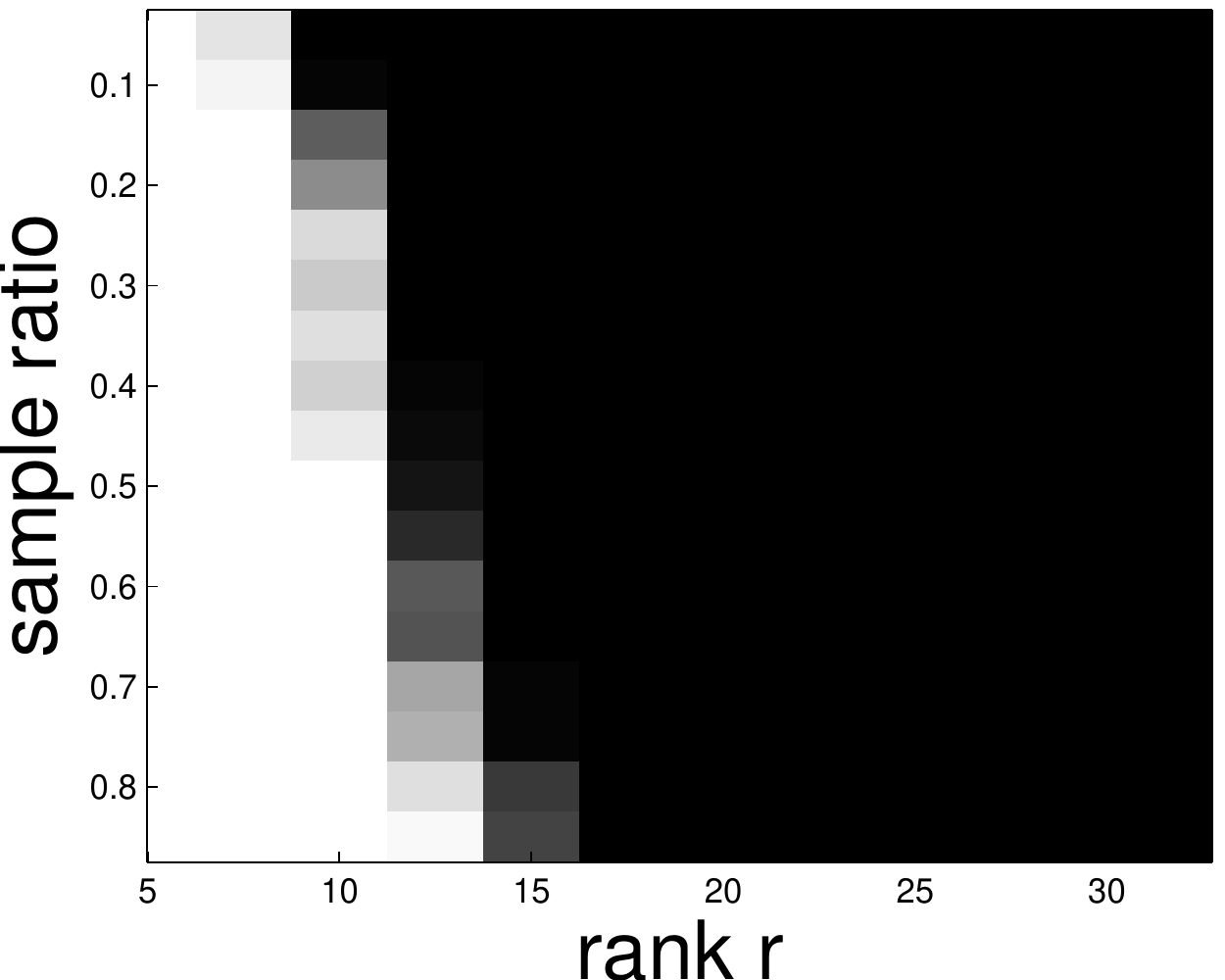}\\
ALSaS-fix & iHOOI-fix & TMac-fix & geomCG-fix & WTucker-true \\
\includegraphics[width=0.17\textwidth]{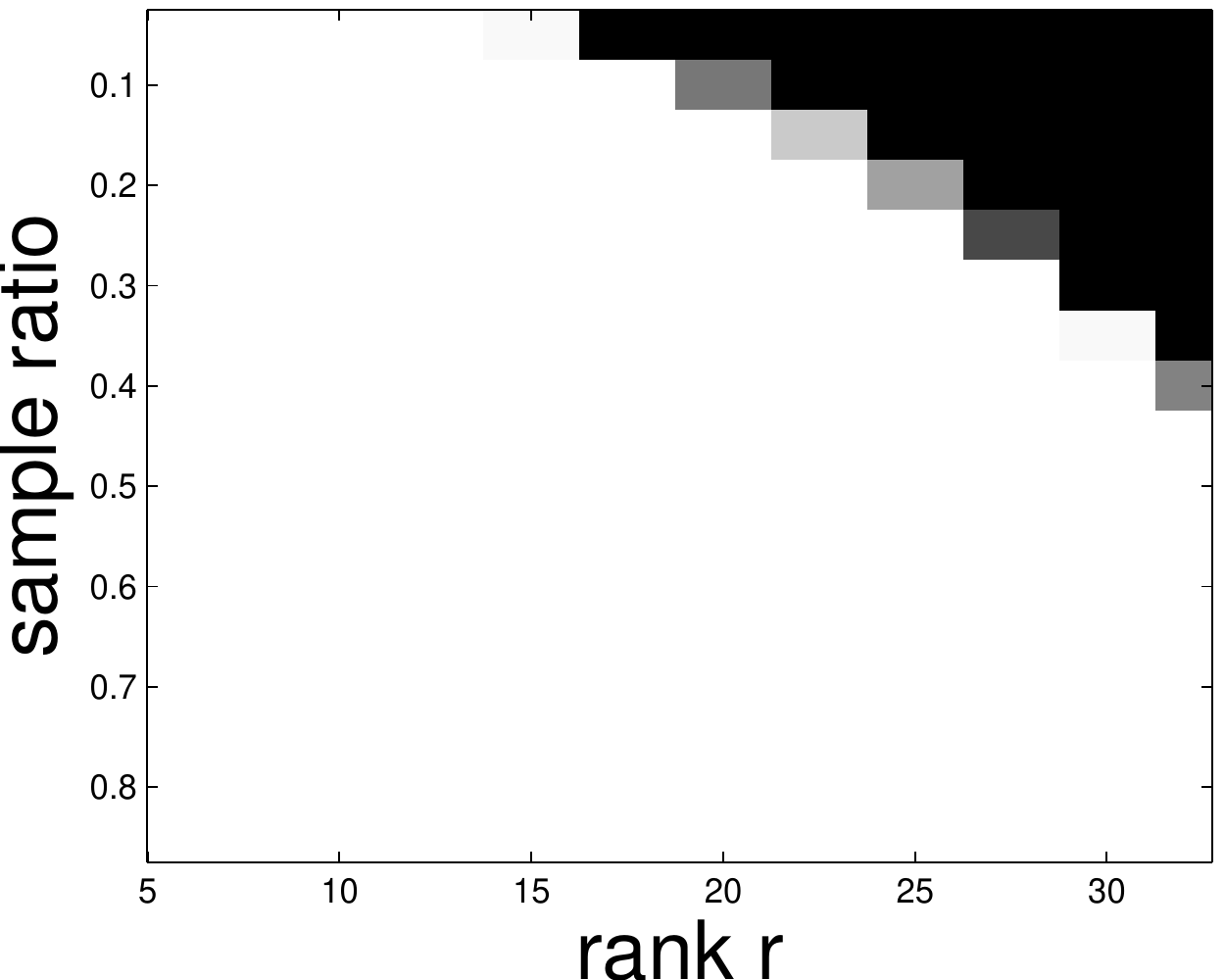} &
\includegraphics[width=0.17\textwidth]{pics/Als_fix_3G.pdf} &
\includegraphics[width=0.17\textwidth]{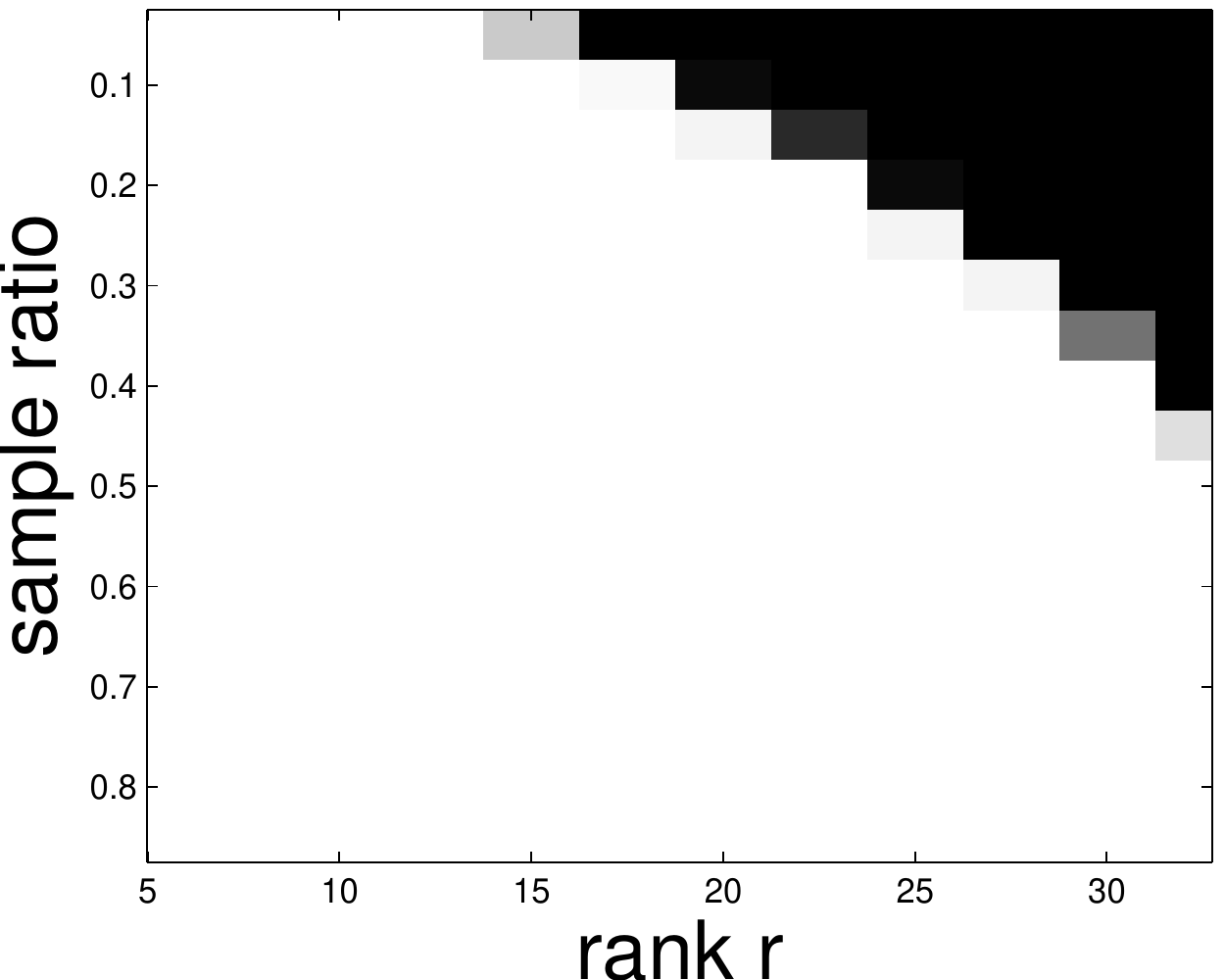} &
\includegraphics[width=0.17\textwidth]{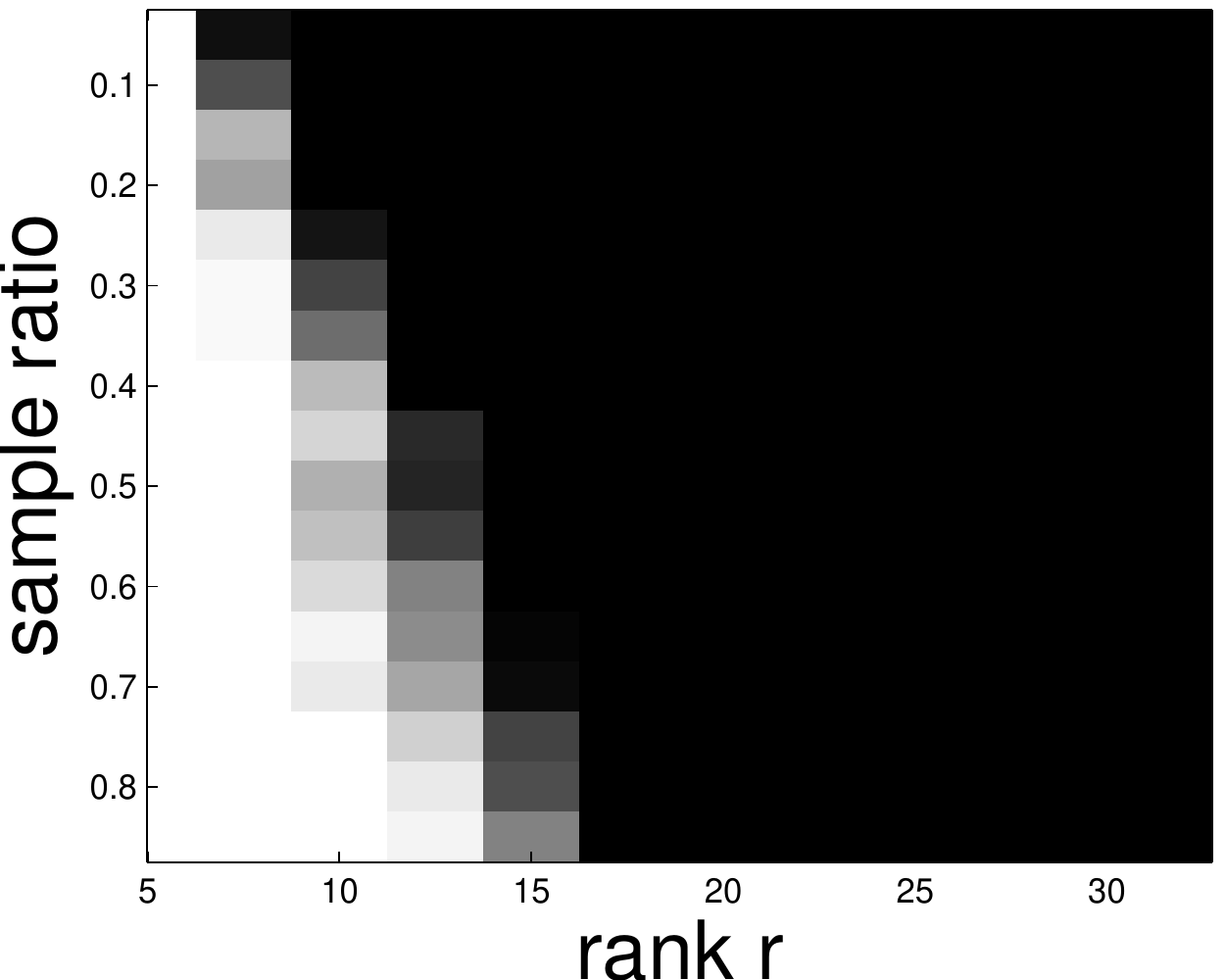} &
\includegraphics[width=0.17\textwidth]{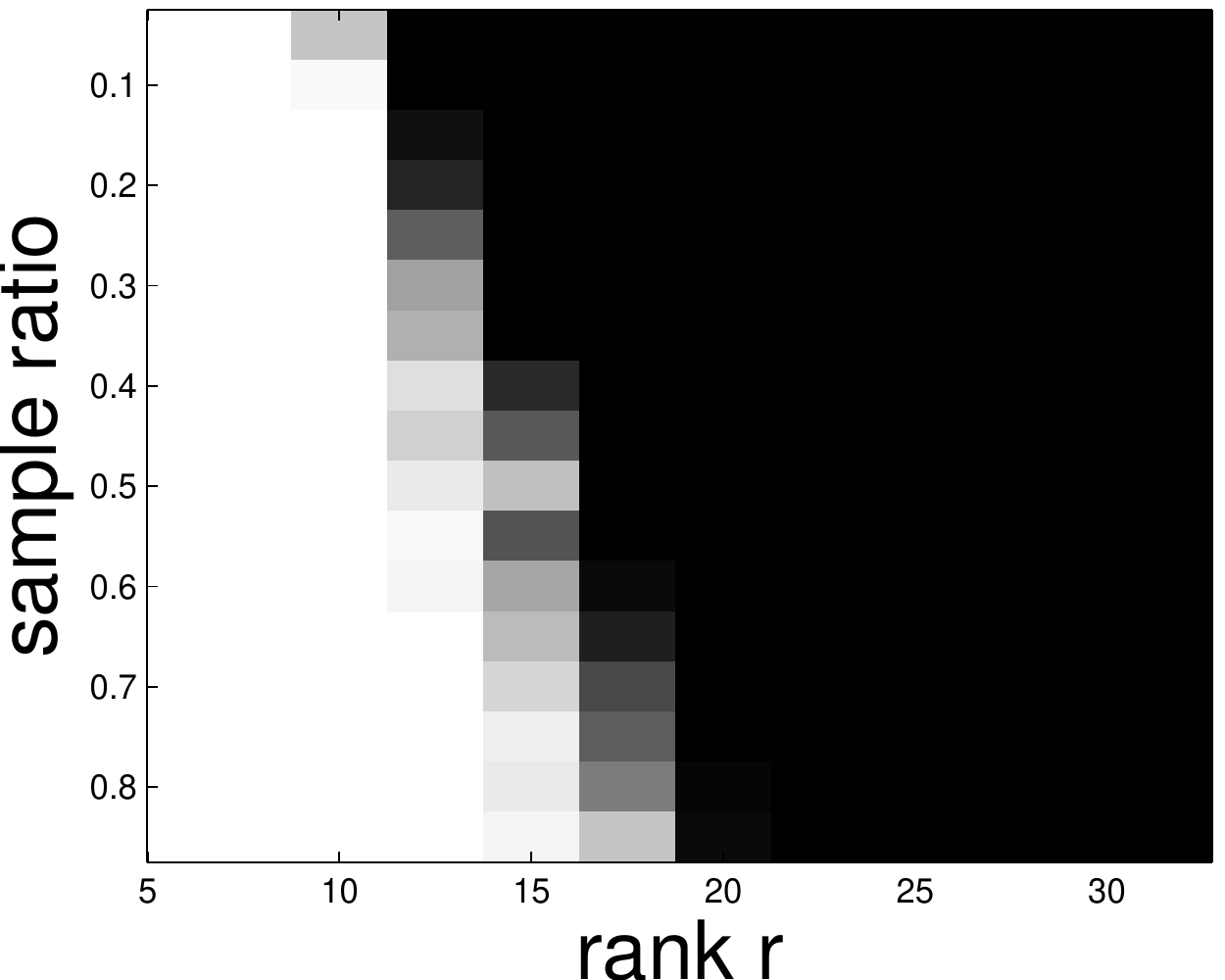}
\end{tabular}}
\end{figure}

\subsubsection*{Application to MRI image reconstruction} We compare the performance of the above five algorithms on reconstructing a 3D brain MRI image, which has been used in \cite{liu2013tensor, tmac2015} for low-rank tensor completion test. %and face images in the Weizmann database\footnote{\url{http://www.wisdom.weizmann.ac.il/~vision/FaceBase/}}.
The MRI dataset has 181 images of resolution $217\times181$. %The Weizmann face image dataset consists of 28 subjects, and each of them are taken photos from 5 different viewpoints under 3 different illuminations and with 3 different face expressions. %; see Figure \ref{fig:weizmann}. 
%Each face image originally has pixels of $352\times 512$ and is downsampled to $71\times103$ and then vectorized in our test. 
We form it into a $181\times217\times181$ tensor, %and $28\times5\times3\times3\times7313$ tensors respectively. 
and Figure \ref{fig:dist-svd} plots its scaled singular values of each mode matricization. From the figure, we see that the dataset has very good multilinear low-rankness property, and it can be well approximated by a rank-$(50,50,50)$ tensor. Hence, we set $r_n^{\max}=50$ and initialize $r_n=1,n=1,2,3$ for ALSaS, iHOOI, TMac, and geomCG and fix $r_n=50,n=1,2,3$ for WTucker. Figure \ref{fig:mri_rec} depicts three slices of the original and 95\% masked data  and corresponding reconstructed ones by ALSaS and iHOOI, and Table \ref{table:mri_rec} gives the average relative reconstruction errors and running time of 3 independent trials by the five algorithms from 5\% and 10\% entries sampled uniformly at random. From the table, we see that ALSaS, iHOOI, TMac, and geomCG can all give highly accurate reconstructions while WTucker achieves a relatively lower accuracy. In addition, geomCG and WTucker takes much more time than the other three and iHOOI is the fastest one among the compared methods.

%\begin{figure}\caption{Selected images of the first subject corresponding to three different viewpoints}\label{fig:weizmann}
%\vspace{0.1cm}
%\centering
%\includegraphics[width=0.15\textwidth]{pics/Weizmann_pp1_vv1_ill1}
%\includegraphics[width=0.15\textwidth]{pics/Weizmann_pp1_vv1_ill2}
%\includegraphics[width=0.15\textwidth]{pics/Weizmann_pp1_vv1_ill3}\\[0.1cm]
%\includegraphics[width=0.15\textwidth]{pics/Weizmann_pp1_vv3_ill1}
%\includegraphics[width=0.15\textwidth]{pics/Weizmann_pp1_vv3_ill2}
%\includegraphics[width=0.15\textwidth]{pics/Weizmann_pp1_vv3_ill3}\\[0.1cm]
%\includegraphics[width=0.15\textwidth]{pics/Weizmann_pp1_vv5_ill1}
%\includegraphics[width=0.15\textwidth]{pics/Weizmann_pp1_vv5_ill2}
%\includegraphics[width=0.15\textwidth]{pics/Weizmann_pp1_vv5_ill3}
%\end{figure}

\begin{figure}\caption{Scaled singular values of each mode matricization of a $181\times217\times181$ brain MRI image.}\label{fig:dist-svd}
\centering
\includegraphics[width=0.3\textwidth]{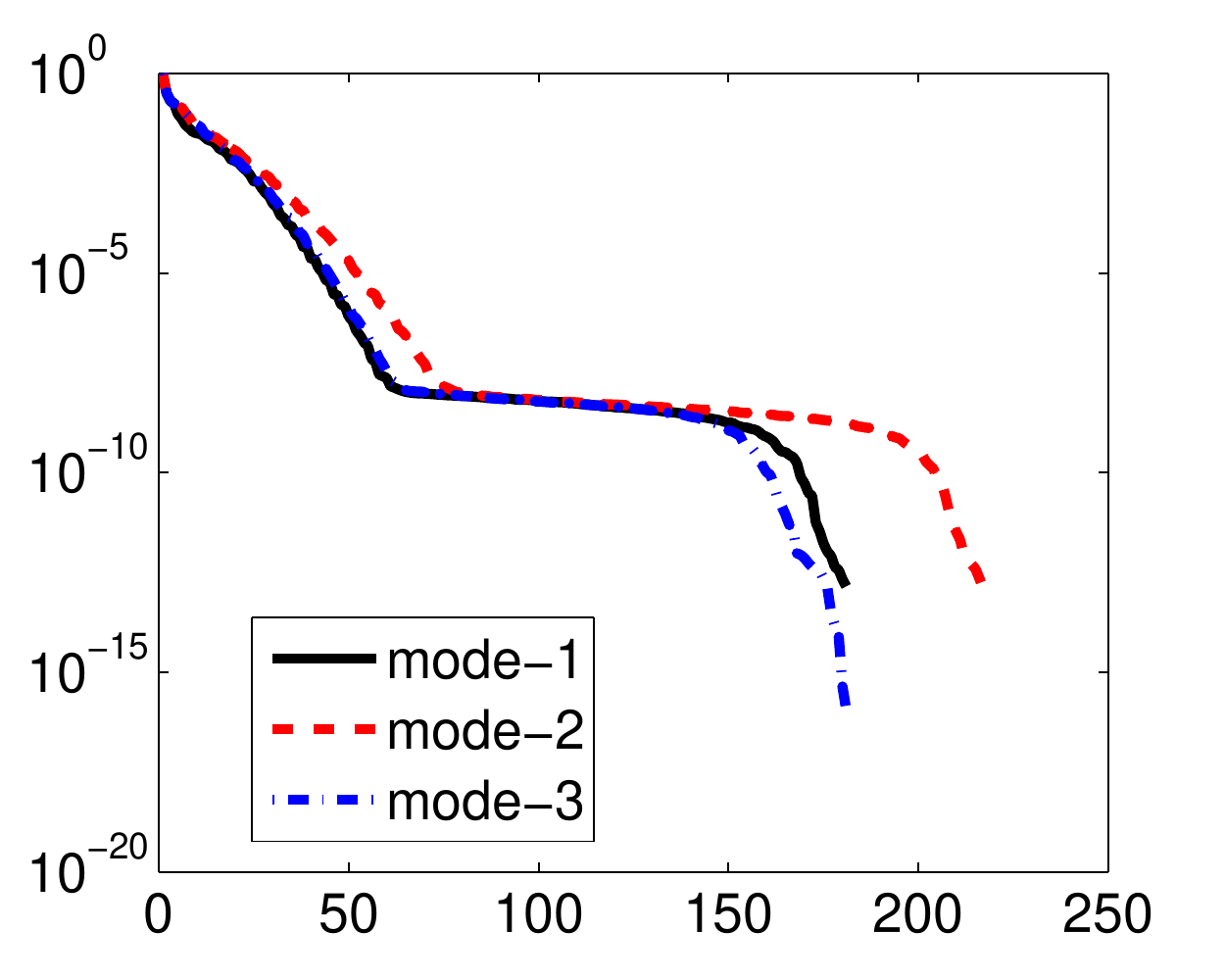}
\end{figure}

\begin{figure}\caption{Brain MRI images: three original slices, the corresponding slices with 95\% pixels missing, and the reconstructed slices by ALSaS and iHOOI.}\label{fig:mri_rec}
\centering
\vspace{0.1cm}
{\small
\begin{tabular}{cccc}
Original & 95\% masked & ALSaS & iHOOI\\
\includegraphics[width=0.20\textwidth]{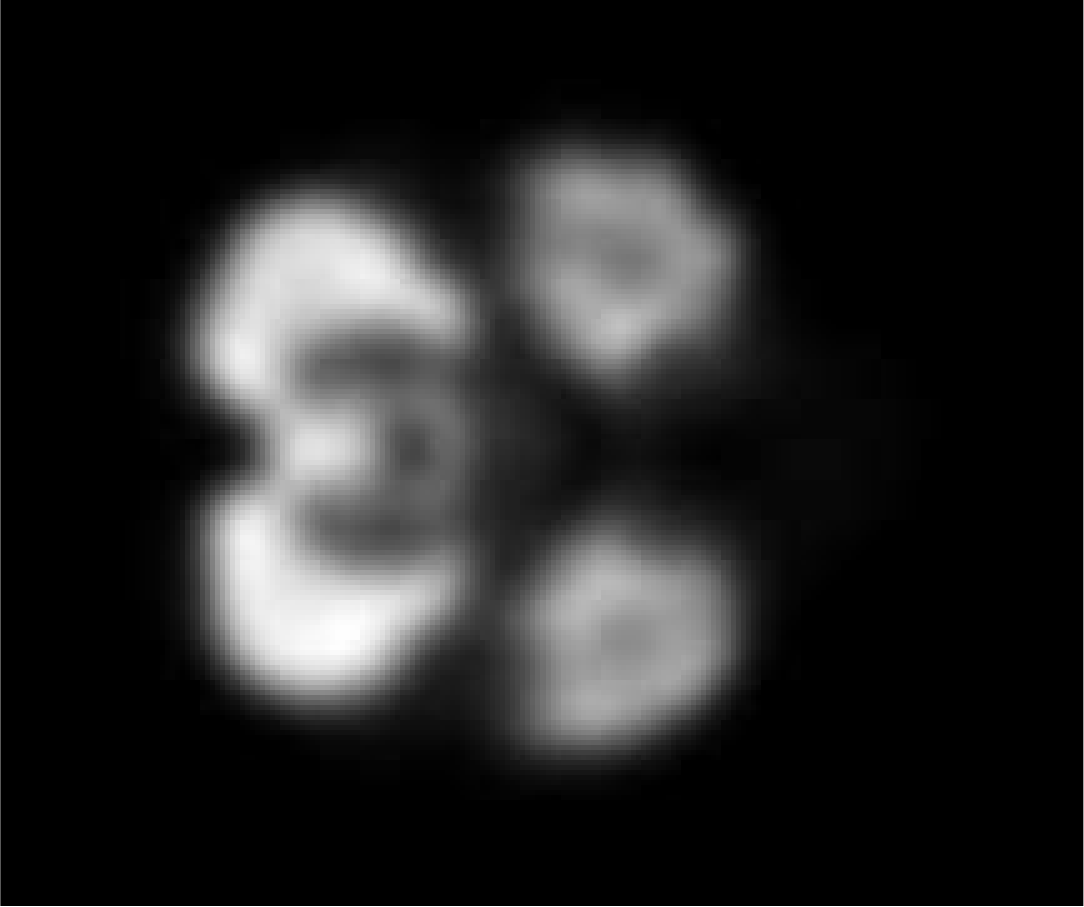} &
\includegraphics[width=0.20\textwidth]{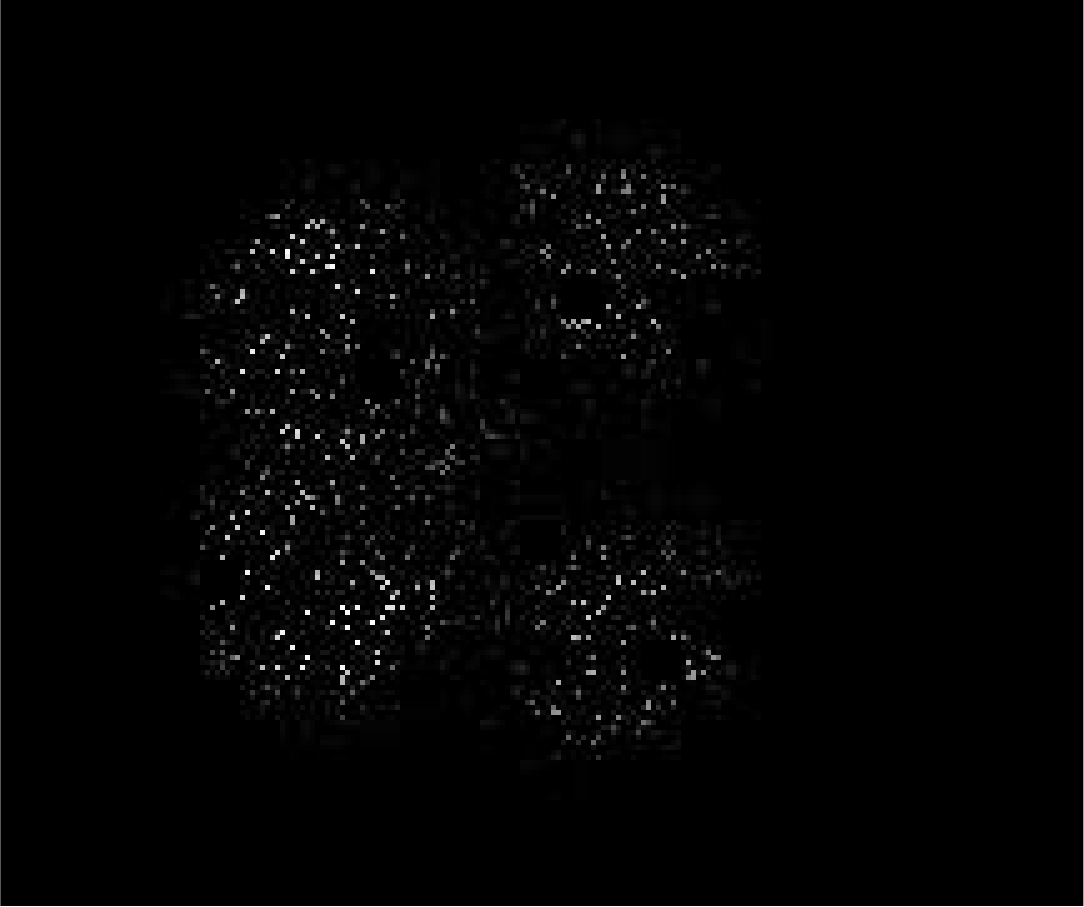} &
\includegraphics[width=0.20\textwidth]{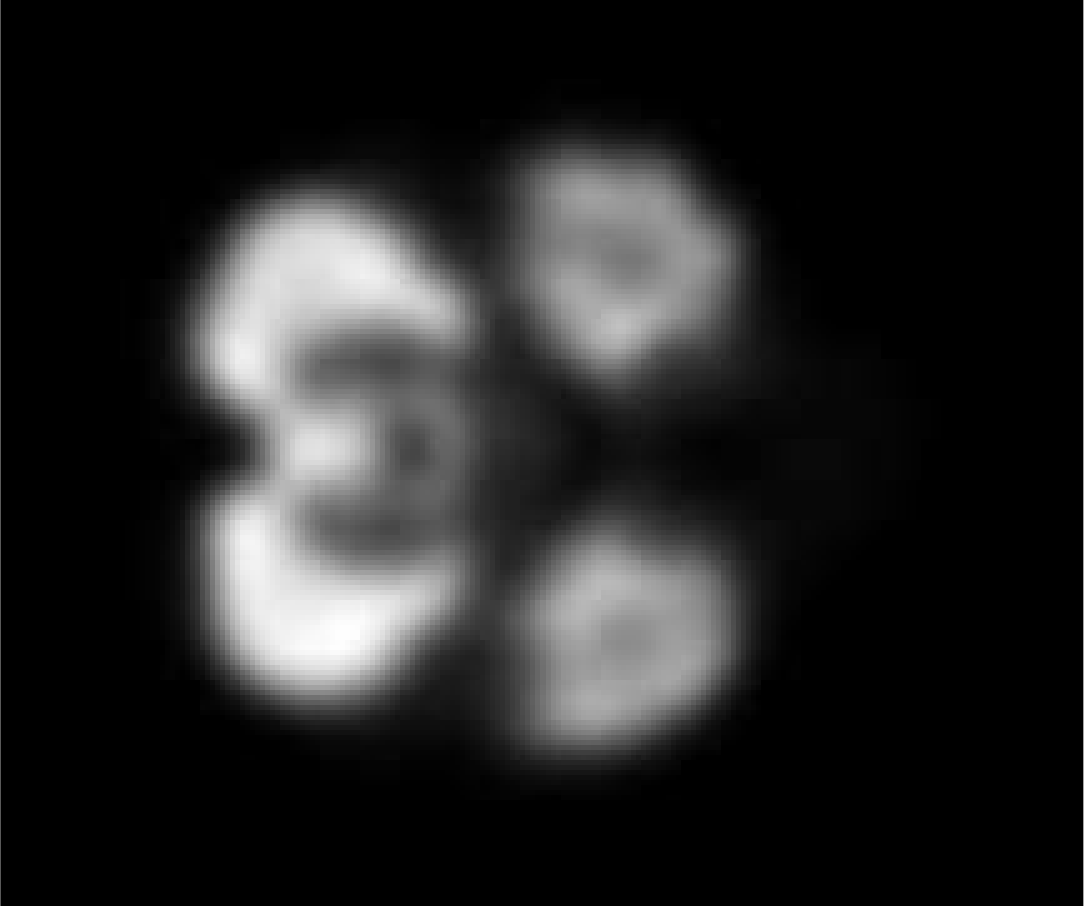} &
\includegraphics[width=0.20\textwidth]{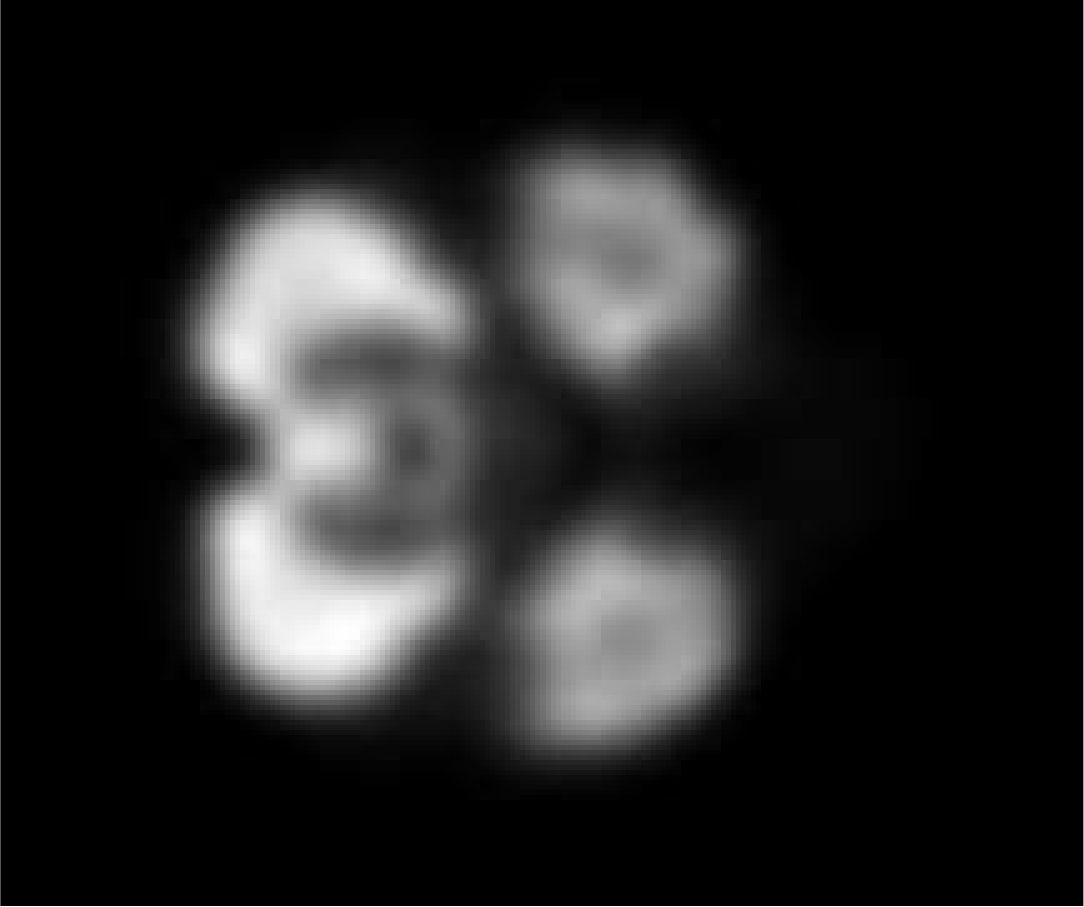} \\
\includegraphics[width=0.20\textwidth]{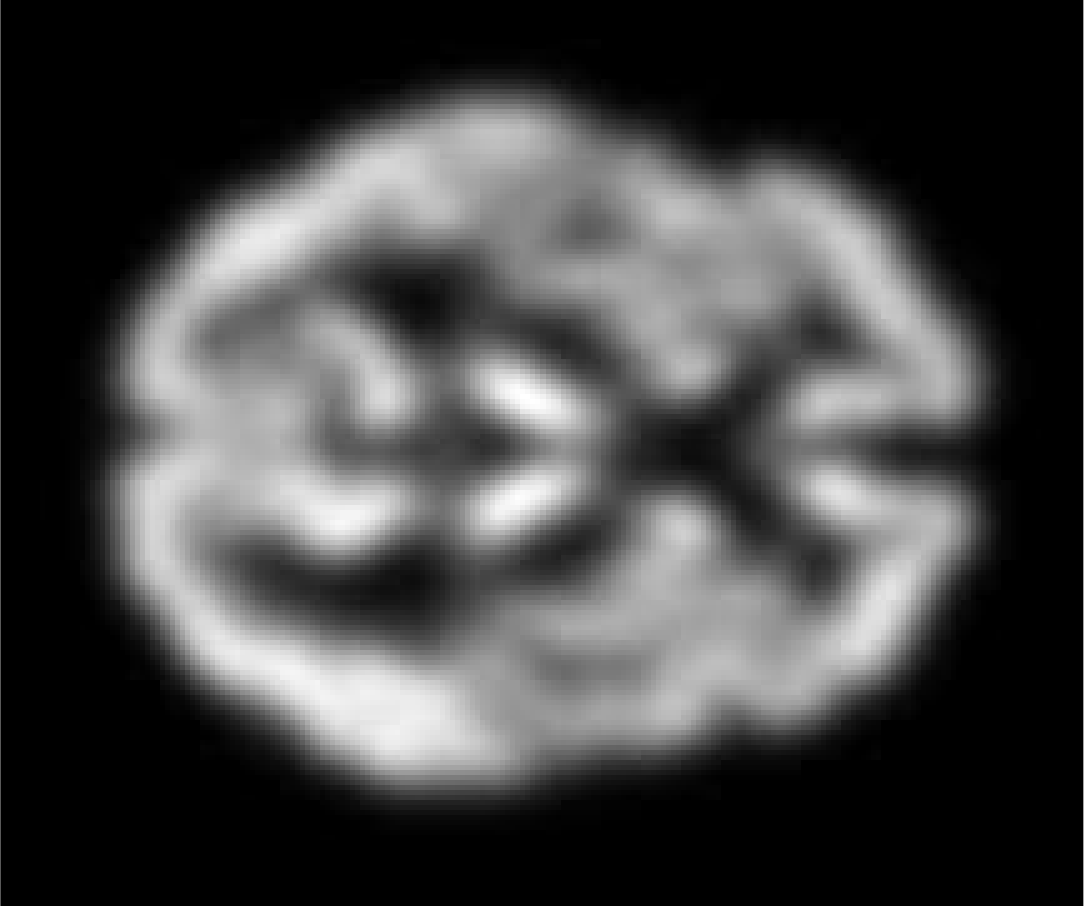} &
\includegraphics[width=0.20\textwidth]{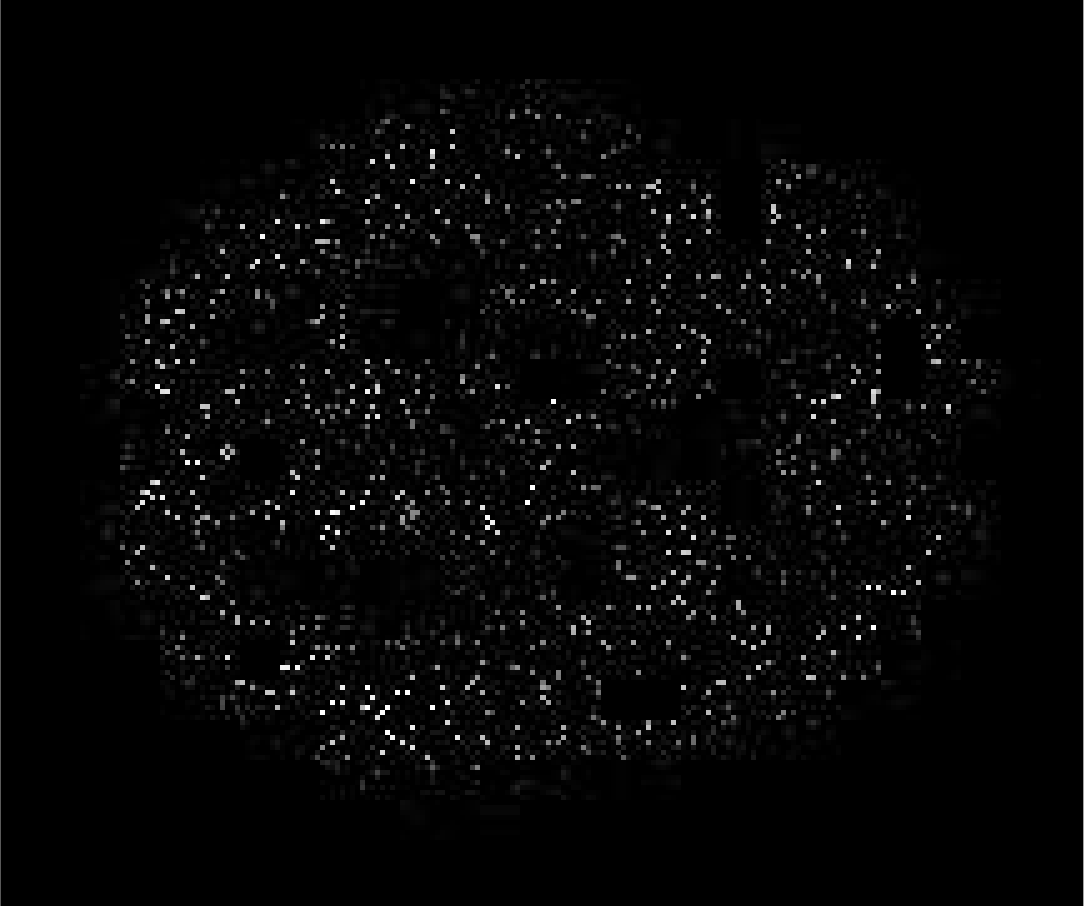} &
\includegraphics[width=0.20\textwidth]{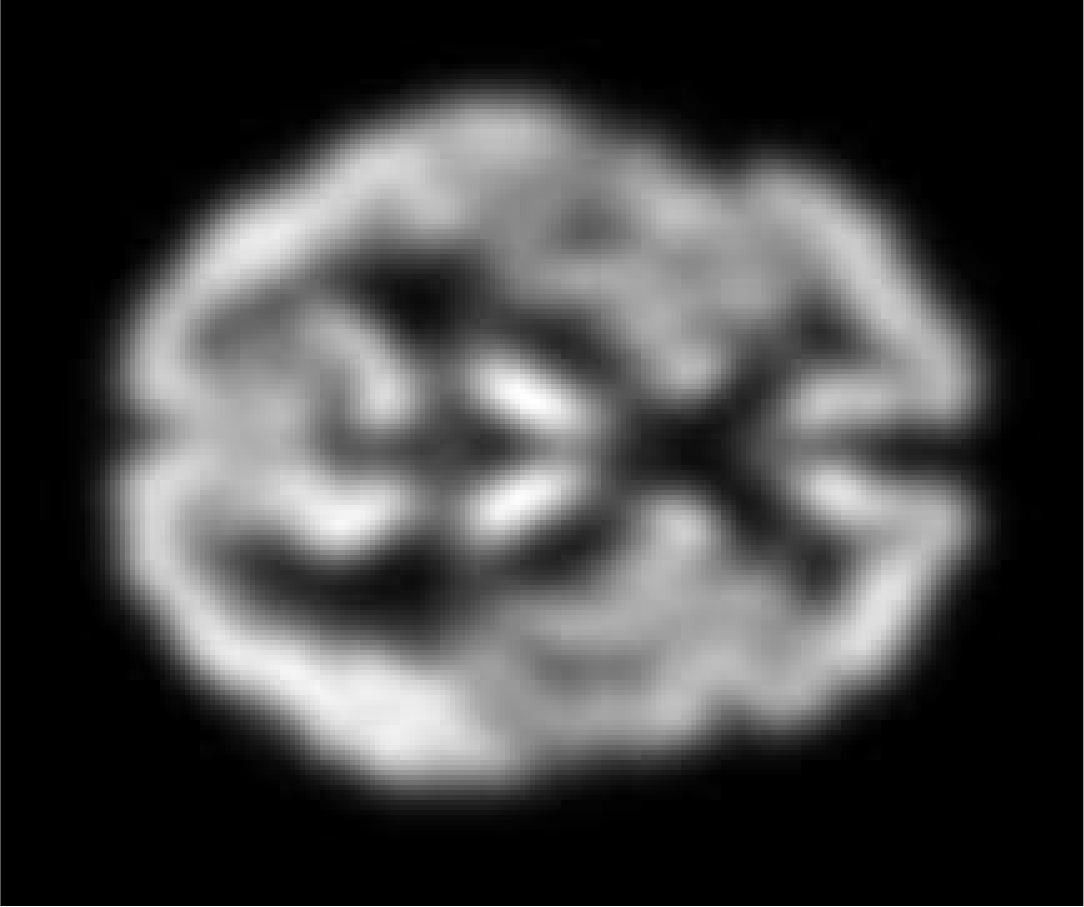} &
\includegraphics[width=0.20\textwidth]{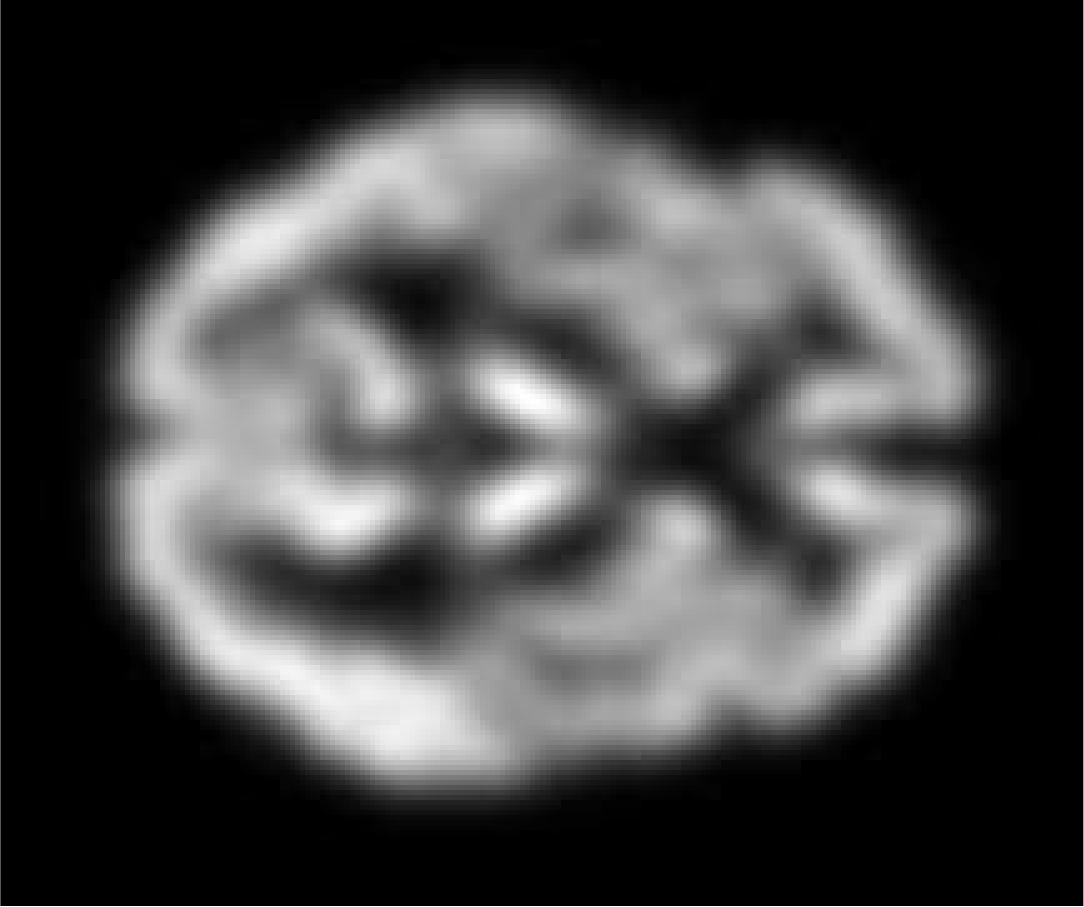} \\
\includegraphics[width=0.20\textwidth]{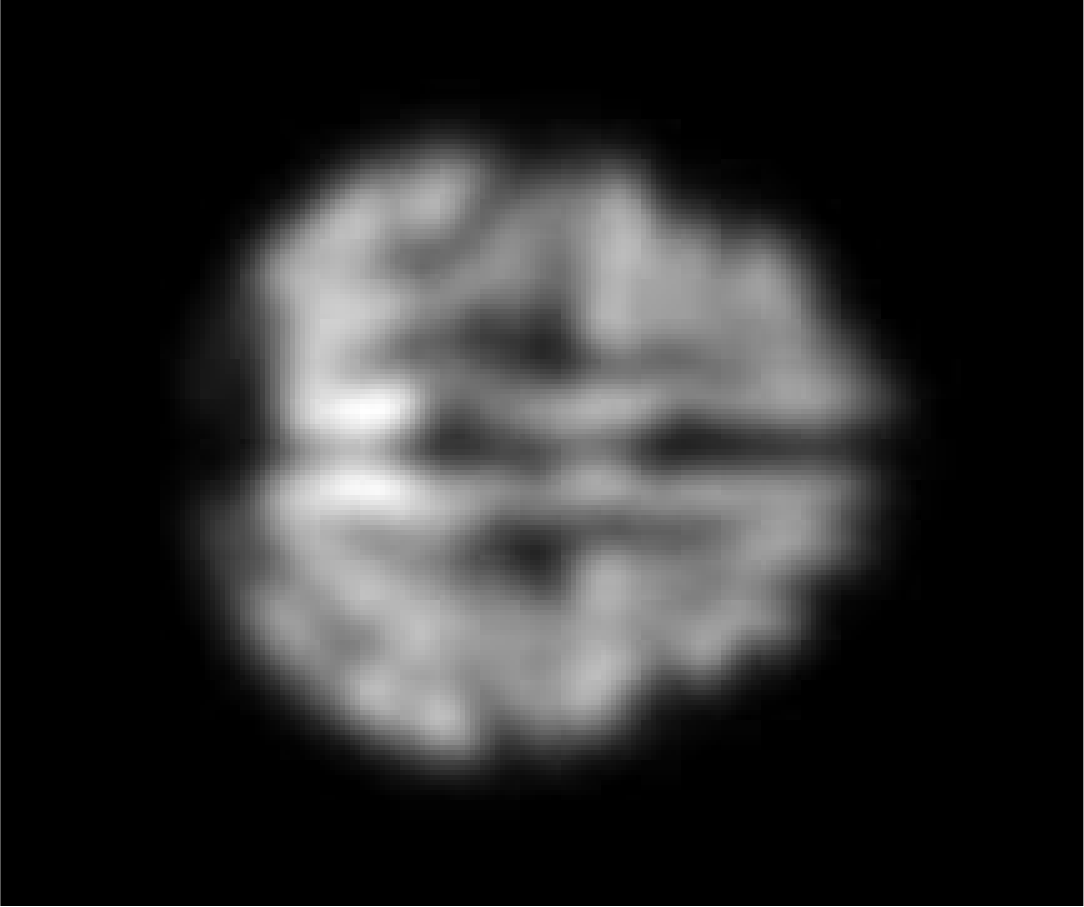} &
\includegraphics[width=0.20\textwidth]{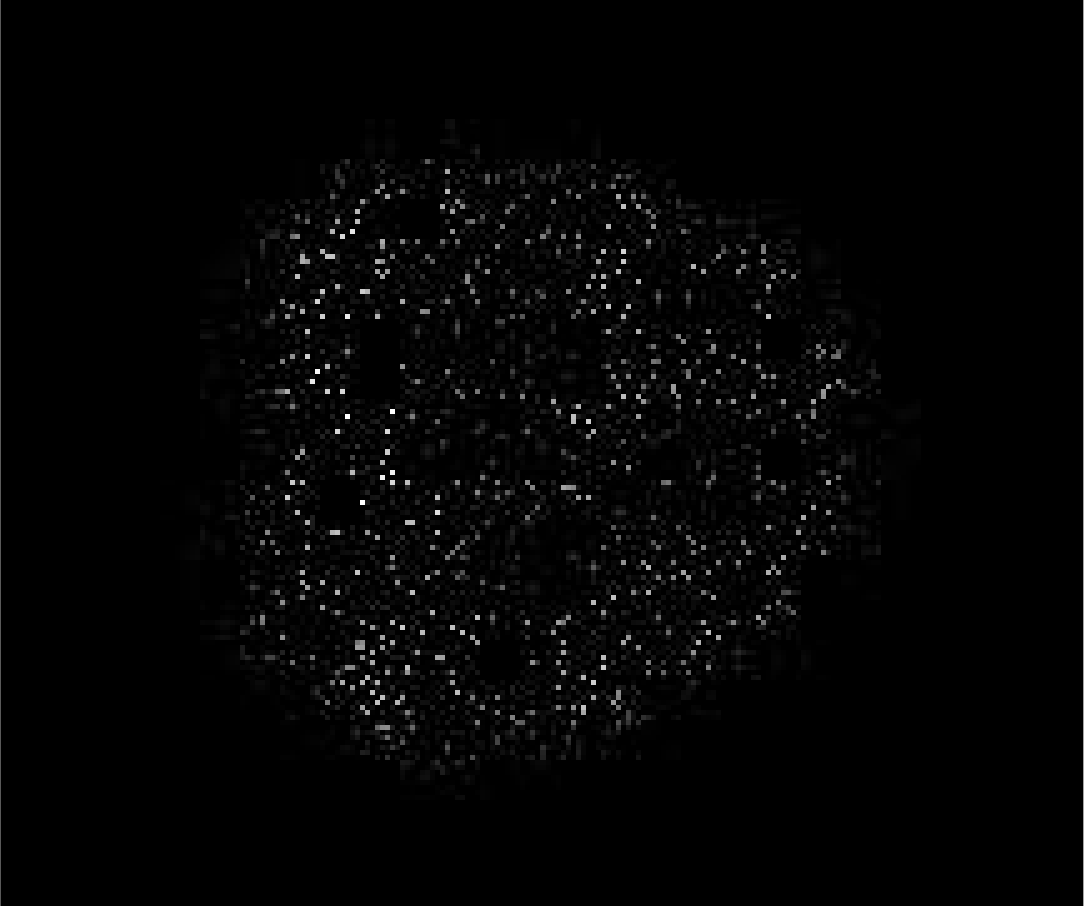} &
\includegraphics[width=0.20\textwidth]{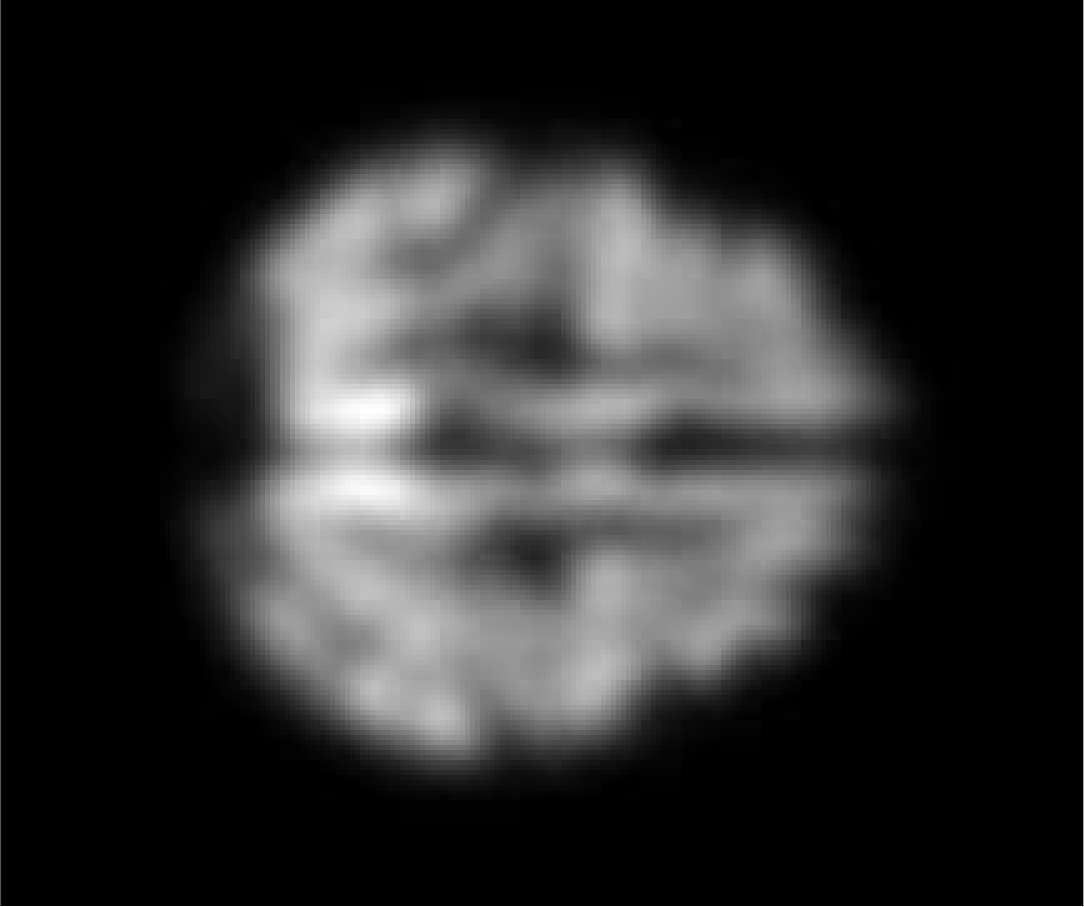} &
\includegraphics[width=0.20\textwidth]{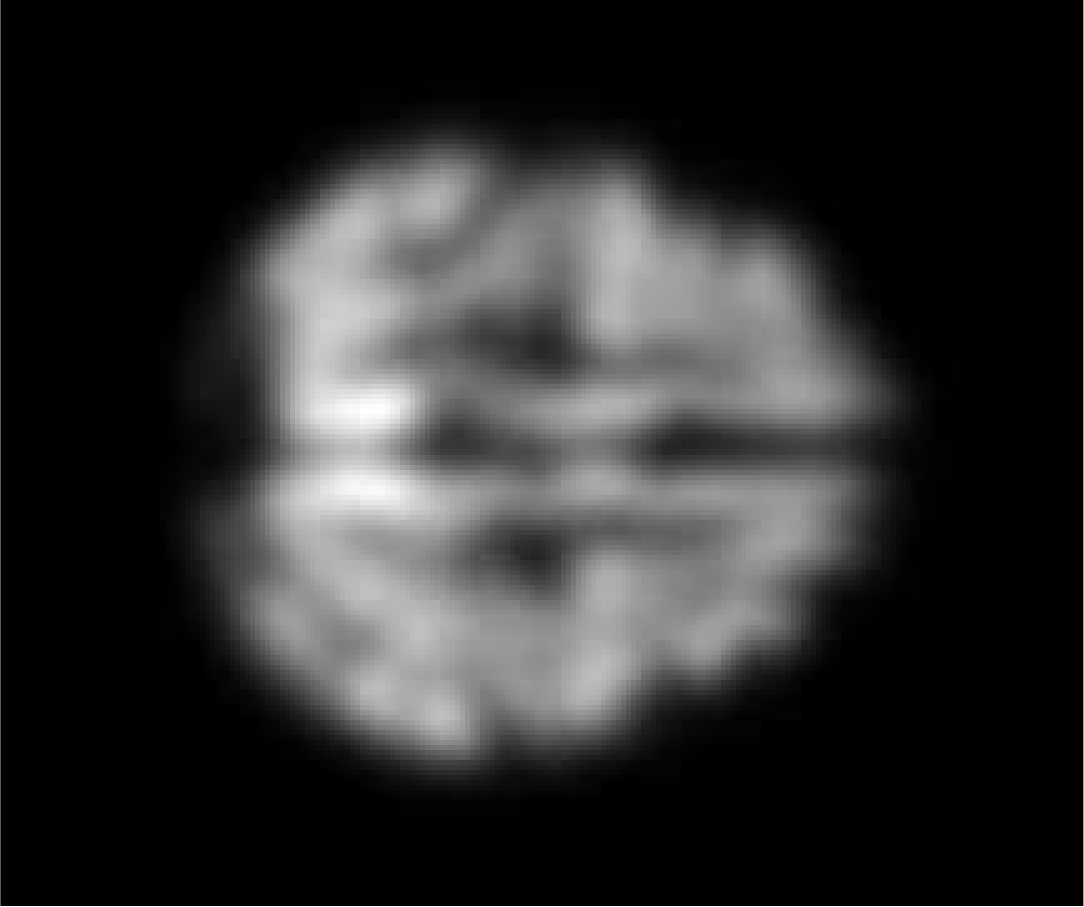} 
\end{tabular}}
\end{figure}

\begin{table}\caption{Average relative reconstruction errors (abbreviated as relerr) and running time (in second) of ALSaS, iHOOI, TMac, geomCG, and WTucker from 5\% and 10\% uniformly randomly sampled entries of a $181\times217\times181$ MRI dataset. The most accurate solutions are highlighted in \textbf{bold}.}\label{table:mri_rec}
\centering
{\small\begin{tabular}{|c|cc|cc|cc|cc|cc|}\hline
 & \multicolumn{2}{|c|}{ALSaS} & \multicolumn{2}{|c|}{iHOOI} & \multicolumn{2}{|c|}{TMac} & \multicolumn{2}{|c|}{geomCG} & \multicolumn{2}{|c|}{WTucker}\\\hline
 SR & relerr & time & relerr & time & relerr & time & relerr & time & relerr & time\\\hline
 5\% & 2.04e-4 & 580 & 1.97e-4 & 377 & 2.94e-4 & 711 & \textbf{1.80e-4} & 5055 & 1.45e-2 & 4580\\
 10\% & 2.85e-5 & 357 & \textbf{2.77e-5} & 310 & 3.36e-5 & 319 & 4.36e-5 & 12243 & 9.78e-3 & 5652\\\hline
\end{tabular}}
\end{table}

\section{Conclusions}\label{sec:conclusion}
We have formulated an incomplete higher-order singular value decomposition (HOSVD) problem and also presented one algorithm, called iHOOI, for solving the problem based on block coordinate update. Under boundedness assumption on the iterates, we have shown the global convergence of iHOOI in terms of a first-order optimality condition if there is a positive gap between the $r_n$-th and $(r_n+1)$-th largest singular values of intermediate points. %, we show that the second algorithm also has a global convergence result. 
Hence, for the first time, we have given global convergence of the popular higher-order orthogonality iteration (HOOI) method by regarding it as a special case of iHOOI. In addition, we have tested the efficiency and reliability of the proposed method on obtaining dominant factors of underlying tensors and also reconstructing low-multilinear-rank tensors, on both of which we have demonstrated that it can outperform state-of-the-art methods.

\appendix %We present the alternating least squares method for solving \eqref{eq:main2}, and also we give the proof of \eqref{eq:grad-partx}.
\section{Alternating least squares}\label{app:als}
We present another algorithm for finding an approximate higher-order singular value decomposition of an given tensor $\bm{\cM}$ from its partial entries. This algorithm is based on the alternating least squares method for solving \eqref{eq:main2}. 

Note that in \eqref{eq:main2}, $f$ is convex with respect to each block variable among $\bm{\cC}, \vA_1,\ldots,\vA_N$ and $\bm{\cX}$ while keeping the others fixed. In addition, we have $\bm{\cC}\times_{i=1}^N(\vA_i\vR_i)=(\bm{\cC}\times_{i=1}^N\vR_i)\times_{i=1}^N\vA_i$ from \eqref{eq:ttm}. Hence, we propose to solve \eqref{eq:main2} by first cyclically updating $\bm{\cC},\vA_1,\ldots,\vA_N$ and $\bm{\cX}$ without orthogonality constraint and then normalizing $\vA_n$'s at the end of each cycle. Specifically, let $(\hat{\bm{\cC}},\hat{\vA},\hat{\bm{\cX}})$ be the current value of the variables. We renew $({\bm{\cC}},{\vA})$ to $(\tilde{\bm{\cC}},\tilde{\vA})$ by first performing the updates
\begin{subequations}\label{eq:update}
\begin{align}
\bm{\mcc}^{mid}&=\argmin_{\bm{\mcc}}f(\bm{\mcc},\hat{\mbfa},\hat{\bm{\mcx}}),\label{eq:update-c}\\
\mbfa_n^{mid}&=\argmin_{\mbfa_n}f(\bm{\mcc}^{mid},\mbfa_{<n}^{mid},\mbfa_n,\hat{\mbfa}_{>n},\hat{\bm{\mcx}}),\ n = 1,\ldots,N,\label{eq:update-a}
\end{align}
\end{subequations}
Assuming the economy QR decomposition of $\vA_n^{mid}$ to be $\vA_n^{mid}=\vQ_n\vR_n,\,n=1,\ldots,N$, we then let 
$$\tilde{\bm{\cC}}=\bm{\cC}^{mid}\times_{i=1}^N\vR_i,\quad \tilde{\vA}_n=\vQ_n,\,n=1,\ldots,N.$$
We renew ${\bm{\cX}}$ to $\tilde{\bm{\cX}}$ by
\begin{equation}\label{eq:update-x}
\tilde{\bm{\mcx}}=\argmin_{\mcp_\Omega(\bm{\mcx})=\mcp_\Omega(\bm{\mcm})}f(\tilde{\bm{\mcc}},\tilde{\mbfa},\bm{\mcx}),
\end{equation}
which can be explicitly written as 
\begin{equation}\label{eq:sol-x}
\tilde{\bm{\mcx}}=\mcp_{\Omega^c}(\tilde{\bm{\mcc}}\times_1\tilde{\mbfa}_1\ldots\times_N\tilde{\mbfa}_N)+\mcp_\Omega(\bm{\mcm}).
\end{equation}

\subsubsection*{$\bm{\cC}$- and $\vA$-subproblems} According to \eqref{eq:vec}, the problem in \eqref{eq:update-c} can be written as
\begin{equation}\label{eq:prob-c}
\bm{\mcc}^{mid}=\argmin_{\bm{\mcc}}\frac{1}{2}\big\|\big(\otimes_{n=N}^1\hat{\mbfa}_n\big)\vvec(\bm{\mcc})-\vvec(\hat{\bm{\mcx}})\big\|_2^2,
\end{equation}
the solution of which can be written as
$$\vvec(\bm{\mcc}^{mid})=\left(\big(\otimes_{n=N}^1\hat{\mbfa}_n\big)^\top
\big(\otimes_{n=N}^1\hat{\mbfa}_n\big)\right)^\dagger
\left(\big(\otimes_{n=N}^1\hat{\mbfa}_n\big)^\top\vvec(\hat{\bm{\mcx}})\right).$$
Using \eqref{eq:vec} and \eqref{eq:kronpro} and noticing $\hat{\vA}_n^\top\hat{\vA}_n=\vI,\,\forall n$, we have
\begin{equation}\label{eq:sol-c}
\bm{\mcc}^{mid}=\hat{\bm{\mcx}}\times_1\hat{\mbfa}_1^\top\ldots
\times_N\hat{\mbfa}_N^\top.
\end{equation}
Let
$
\hat{\mbfb}_n%=&\,\mbfc^{mid}_{(n)}\big(\otimes_{i=N}^{n+1}\hat{\mbfa}_i\big)^\top
%\big(\otimes_{i=n-1}^{1}\mbfa_i^{mid}\big)^\top\\
=\unfold_n(\mbfc^{mid}\times_{i=1}^{n-1}\mbfa_i^{mid}
\times_{i=n+1}^N\hat{\mbfa}_i).
$
Then according to \eqref{eq:mat}, each problem in \eqref{eq:update-a} can be written as
\begin{equation}\label{eq:prob-a}
\mbfa_n^{mid}=\argmin_{\mbfa_n}\frac{1}{2}\big\|\mbfa_n\hat{\mbfb}_n-\hat{\mbfx}_{(n)}\big\|_F^2,
\end{equation}
whose solution can be explicitly written as
\begin{equation}\label{eq:sol-a}\mbfa_n^{mid}=\hat{\mbfx}_{(n)}(\hat{\mbfb}_n)^\top\big(\hat{\mbfb}_n\hat{\mbfb}_n^\top\big)^\dagger.
\end{equation}

Summarizing the above discussion, we have the pseudocode of the proposed method in Algorithm \ref{alg:lrtfit}.

\begin{algorithm}\caption{\textbf{A}lternating \textbf{L}east \textbf{S}qu\textbf{a}res for HO\textbf{S}VD with missing values (ALSaS)}\label{alg:lrtfit}
\begin{algorithmic}[1]
{\small
%\DontPrintSemicolon
\STATE \textbf{Input:} index set $\Omega$, observed entries $\mcp_\Omega(\bm{\mcm})$, and initial point $(\bm{\mcc}^0,\mbfa^0,\bm{\mcx}^0)$ with $(\vA_i^0)^\top \vA_i^0= \vI,\forall i$ and $\mcp_\Omega(\bm{\mcx}^0)=\mcp_\Omega(\bm{\mcm})$.
\FOR{$k=0,1,\ldots$}
\STATE Update $\bm{\cC}$ by
\begin{equation}\label{alg:update-c}
\bm{\cC}^{k+\frac{1}{2}}=\bm{\mcx}^k\times_{i=1}^N(\mbfa_i^k)^\top.
\end{equation}
\FOR{$n = 1,\ldots,N$}
\STATE Update $\vA_n$ by 
$$\vA_n^{k+\frac{1}{2}}=\bm{\cX}_{(n)}^k(\vB_n^k)^\top(\vB_n^k(\vB_n^k)^\top)^\dagger,$$
where
\begin{equation}\label{eq:bnk}
\vB_n^k=\unfold_n\big(\bm{\cC}^{k+\frac{1}{2}}\times_{i=1}^{n-1}\vA_i^{k+\frac{1}{2}}\times_{i=n+1}^N\vA_i^k\big).
\end{equation}
\ENDFOR
\STATE Let the economy QR decomposition of $\vA_n^{k+\frac{1}{2}}$ be $\vQ_n^k\vR_n^k$ and do the normalization
\begin{equation}\label{eq:normal}\bm{\cC}^{k+1}=\bm{\cC}^{k+\frac{1}{2}}\times_{i=1}^N\vR_i^k,\quad \vA_n^{k+1}=\vQ_n^k,\, n=1,\ldots,N.
\end{equation}
\STATE Update $\bm{\cX}$ by 
\begin{equation}\label{alg:update-x}
\bm{\mcx}^{k+1}=\mcp_{\Omega^c}(\bm{\mcc}^{k+1}\times_{i=1}^N\mbfa_i^{k+1})+\mcp_\Omega(\bm{\mcm}).
\end{equation}
\IF{stopping criterion is satisfied}
\STATE Return $(\bm{\mcc}^{k+1},\mbfa^{k+1},\bm{\mcx}^{k+1})$.
\ENDIF
\ENDFOR
}
\end{algorithmic}
\end{algorithm}

Following the convergence analysis of \cite{tmac2015}, it is not difficult to show the global convergence of Algorithm \ref{alg:lrtfit}, and we summarize the result as follows with no proof.
\begin{theorem}[Global convergence of Algorithm \ref{alg:lrtfit}]\label{thm:convg1}
Let $\{(\bm{\cC}^k,\vA^k,\bm{\cX}^k)\}_{k=1}^\infty$ be the sequence generated from Algorithm \ref{alg:lrtfit}. Also, let $\bm{\cY}^k=\cP_\Omega(\bm{\cC}^k\times_{i=1}^N\vA_i^k)-\cP_\Omega(\bm{\cM})$ and $\bm{\Lambda}^k=\vzero$ for all $k$. Then any finite limit point of $\{(\bm{\cC}^k,\vA^k,\bm{\cX}^k,\bm{\Lambda}^k,\bm{\cY}^k)\}_{k=1}^\infty$ satisfies the first-order optimality conditions of \eqref{eq:main2}. Furthermore, if $\{\cP_{\Omega^c}(\bm{\cX}^k)\}_{k=1}^\infty$ is bounded, then 
\begin{equation}\label{eq:glbcvg}
\lim_{k\to\infty}\nabla\cL_f(\bm{\cC}^k,\vA^k,\bm{\cX}^k,\bm{\Lambda}^k,\bm{\cY}^k)=\vzero.
\end{equation}
\end{theorem}

\section{Proof of \eqref{eq:grad-partx}}\label{app:pf-partx}
Let $\vD_1, \vD_2$ be diagonal matrices such that $\mathrm{vec}(\cP_{\Omega}(\bm{\cX}))=\vD_1\mathrm{vec}(\bm{\cX})$ and $\mathrm{vec}(\cP_{\Omega^c}(\bm{\cX}))=\vD_2\mathrm{vec}(\bm{\cX})$ respectively. Then from \eqref{eq:vec}, it follows that
$$h(\bm{\cX};\tilde{\vA})=\frac{1}{2}\big\|\vD_2\mathrm{vec}(\bm{\cX})\big\|^2-\frac{1}{2}\big\|(\otimes_{i=N}^1\tilde{\vA}_i^\top)(\vD_2\mathrm{vec}(\bm{\cX})+\vD_1\mathrm{vec}(\bm{\cM}))\big\|^2.$$
Hence,
\begin{align*}\vvec\big(\nabla_{\bm{\cX}}h(\hat{\bm{\cX}};\tilde{\vA})\big)=&\,\vD_2^\top\vD_2\vvec(\hat{\bm{\cX}})-\vD_2^\top
(\otimes_{i=N}^1\tilde{\vA}_i^\top)^\top(\otimes_{i=N}^1\tilde{\vA}_i^\top)(\vD_2\mathrm{vec}(\hat{\bm{\cX}})+\vD_1\mathrm{vec}(\bm{\cM}))\\
=&\,\vD_2\vvec(\hat{\bm{\cX}})-\vD_2(\otimes_{i=N}^1\tilde{\vA}_i\tilde{\vA}_i^\top)(\vD_2\mathrm{vec}(\hat{\bm{\cX}})+\vD_1\mathrm{vec}(\bm{\cM}))\\
=&\,\vD_2\vvec(\hat{\bm{\cX}})-\vD_2(\otimes_{i=N}^1\tilde{\vA}_i\tilde{\vA}_i^\top)\mathrm{vec}(\hat{\bm{\cX}})\qquad(\text{because }\cP_\Omega(\hat{\bm{\cX}})=\cP_\Omega(\bm{\cM})).
\end{align*}
Using \eqref{eq:vec}, we have
$$\nabla_{\bm{\cX}}h(\hat{\bm{\cX}};\tilde{\vA})=\cP_{\Omega^c}(\hat{\bm{\cX}})-\cP_{\Omega^c}(\hat{\bm{\cX}}\times_{i=1}^N\tilde{\vA}_i\tilde{\vA}_i^\top),$$
which completes the proof.

\end{document}

%% file: macros.tex
\usepackage[normalem]{ulem} % to use \sout

\newcommand{\bm}[1]{\boldsymbol{#1}}

\newcommand{\mbfa}{\mathbf{A}}
\newcommand{\mbfb}{\mathbf{B}}
\newcommand{\mbfc}{\mathbf{C}}
\newcommand{\mbfd}{\mathbf{D}}

\newcommand{\mbfx}{\mathbf{X}}

\newcommand{\mbr}{\mathbb{R}}

\newcommand{\mcc}{\mathcal{C}}

\newcommand{\mcm}{\mathcal{M}}

\newcommand{\mcp}{\mathcal{P}}

\newcommand{\mcx}{\mathcal{X}}
\newcommand{\mcy}{\mathcal{Y}}

\newcommand{\vx}{{\mathbf{x}}}

\newcommand{\vA}{{\mathbf{A}}}
\newcommand{\vB}{{\mathbf{B}}}
\newcommand{\vC}{{\mathbf{C}}}
\newcommand{\vD}{{\mathbf{D}}}

\newcommand{\vG}{{\mathbf{G}}}
\newcommand{\vH}{{\mathbf{H}}}
\newcommand{\vI}{{\mathbf{I}}}

\newcommand{\vM}{{\mathbf{M}}}

\newcommand{\vQ}{{\mathbf{Q}}}
\newcommand{\vR}{{\mathbf{R}}}

\newcommand{\vU}{{\mathbf{U}}}
\newcommand{\vV}{{\mathbf{V}}}

\newcommand{\vX}{{\mathbf{X}}}
\newcommand{\vY}{{\mathbf{Y}}}
\newcommand{\vZ}{{\mathbf{Z}}}

\newcommand{\cC}{{\mathcal{C}}}

\newcommand{\cG}{{\mathcal{G}}}

\newcommand{\cK}{{\mathcal{K}}}
\newcommand{\cL}{{\mathcal{L}}}
\newcommand{\cM}{{\mathcal{M}}}

\newcommand{\cO}{{\mathcal{O}}}
\newcommand{\cP}{{\mathcal{P}}}

\newcommand{\cX}{{\mathcal{X}}}
\newcommand{\cY}{{\mathcal{Y}}}
\newcommand{\cZ}{{\mathcal{Z}}}

\newcommand{\RR}{\mathbb{R}} % real
\newcommand{\vzero}{\mathbf{0}} % 0 vector
\newcommand{\diag}{{\mathrm{diag}}} % matrix diagonal -> vector
\newcommand{\vvec}{{\mathrm{vec}}}
\newcommand{\fold}{{\mathbf{fold}}} % fold into a tensor
\newcommand{\unfold}{{\mathbf{unfold}}} % unfold a tensor
\newcommand{\fit}{{\mathrm{fit}}} % data fitting
\newcommand{\obj}{{\mathrm{obj}}} % data fitting
\newcommand{\round}{{\mathrm{round}}}
\newcommand{\err}{{\mathrm{err}}}

\newcommand{\st}{\mbox{ s.t. }}

\newcommand{\rankk}{{\hbox{rank}}}

\DeclareMathOperator*{\argmin}{arg\,min} % argmin
\DeclareMathOperator*{\argmax}{arg\,max} % argmax

\newcommand{\bc}{\begin{center}}
\newcommand{\ec}{\end{center}}

\newcommand{\bdm}{\begin{displaymath}}
\newcommand{\edm}{\end{displaymath}}

\newcommand{\beq}{\begin{equation}}
\newcommand{\eeq}{\end{equation}}

\newcommand{\bfl}{\begin{flushleft}}
\newcommand{\efl}{\end{flushleft}}

\newcommand{\bt}{\begin{tabbing}}
\newcommand{\et}{\end{tabbing}}

\newcommand{\beqn}{\begin{eqnarray}}
\newcommand{\eeqn}{\end{eqnarray}}

\newcommand{\beqs}{\begin{align*}} % no equation numbers
\newcommand{\eeqs}{\end{align*}}  % no equation numbers

%% file: iHOSVD_JCM_final.bbl
\begin{thebibliography}{10}

\bibitem{acar2010scalable}
{\sc E.~Acar, D.~M. Dunlavy, T.~G. Kolda, and M.~M{\o}rup}, {\em Scalable
  tensor factorizations with missing data.}, in Proceedings of the SIAM International Conference on Data Mining, (2010), pp.~701--712.

\bibitem{BeckTeboulle2009}
{\sc A.~Beck and M.~Teboulle}, {\em A fast iterative shrinkage-thresholding
  algorithm for linear inverse problems}, SIAM Journal on Imaging Sciences, 2
  (2009), pp.~183--202.

\bibitem{Bertsekas-NLP}
{\sc D.~P. Bertsekas}, {\em Nonlinear Programming}, {Athena Scientific},
  September 1999.

\bibitem{buchanan2005damped}
{\sc A.~M. Buchanan and A.~W. Fitzgibbon}, {\em Damped newton algorithms for
  matrix factorization with missing data}, in IEEE Computer Society Conference on Computer Vision and Pattern
  Recognition, 2 (2005), pp.~316--322.

\bibitem{chen2013simultaneous}
{\sc Y.~Chen, C.~Hsu, and H.~Liao}, {\em Simultaneous tensor decomposition and
  completion using factor priors},  (2013).

\bibitem{de2000multilinear}
{\sc L.~De~Lathauwer, B.~De~Moor, and J.~Vandewalle}, {\em A multilinear
  singular value decomposition}, SIAM Journal on Matrix Analysis and
  Applications, 21 (2000), pp.~1253--1278.

\bibitem{de2000best}
\leavevmode\vrule height 2pt depth -1.6pt width 23pt, {\em On the best rank-1
  and rank-$(r_1, r_2,\ldots, r_n)$ approximation of higher-order tensors},
  SIAM Journal on Matrix Analysis and Applications, 21 (2000), pp.~1324--1342.

\bibitem{filipovic2013tucker}
{\sc M.~Filipovi{\'c} and A.~Juki{\'c}}, {\em Tucker factorization with missing
  data with application to low-n-rank tensor completion}, Multidimensional
  Systems and Signal Processing,  (2013), pp.~1--16.

\bibitem{gabriel1979lower}
{\sc K.~R. Gabriel and S.~Zamir}, {\em Lower rank approximation of matrices by
  least squares with any choice of weights}, Technometrics, 21 (1979),
  pp.~489--498.

\bibitem{gandy2011tensor}
{\sc S.~Gandy, B.~Recht, and I.~Yamada}, {\em Tensor completion and
  low-$n$-rank tensor recovery via convex optimization}, Inverse Problems, 27
  (2011), pp.~1--19.

\bibitem{geng2009facial}
{\sc X.~Geng and K.~Smith-Miles}, {\em Facial age estimation by multilinear
  subspace analysis}, in IEEE International Conference on Acoustics, Speech and Signal Processing (ICASSP), (2009), pp.~865--868.

\bibitem{geng2011face}
{\sc X.~Geng, K.~Smith-Miles, Z.-H. Zhou, and L.~Wang}, {\em Face image
  modeling by multilinear subspace analysis with missing values}, IEEE Transactions on Systems, Man,
  and Cybernetics, Part B: Cybernetics, 41 (2011),
  pp.~881--892.

\bibitem{georghiades2001few}
{\sc A.~S. Georghiades, P.~N. Belhumeur, and D.~Kriegman}, {\em From few to
  many: Illumination cone models for face recognition under variable lighting
  and pose}, IEEE Transactions on Pattern Analysis and Machine Intelligence, 23 (2001), pp.~643--660.

\bibitem{GolubVanLoan1996}
{\sc G.~H. Golub and C.~F. Van~Loan}, {\em Matrix computations}, Johns Hopkins
  Studies in the Mathematical Sciences, Johns Hopkins University Press,
  Baltimore, MD, third~ed., 1996.

\bibitem{horn1991topics}
{\sc R.~Horn and C.~Johnson}, {\em Topics in matrix analysis}, Cambridge Univ.
  Press Cambridge etc, 1991.

\bibitem{huang2014provable}
{\sc B.~Huang, C.~Mu, D.~Goldfarb, and J.~Wright}, {\em Provable low-rank
  tensor recovery}, Pacific Journal of Optimization, 11 (2015), pp.~339--364.

\bibitem{kolda2009tensor}
{\sc T.~Kolda and B.~Bader}, {\em Tensor decompositions and applications}, SIAM
  review, 51 (2009), pp.~455--500.

\bibitem{kreimer2012tensor}
{\sc N.~Kreimer and M.~D. Sacchi}, {\em A tensor higher-order singular value
  decomposition for prestack seismic data noise reduction and interpolation},
  Geophysics, 77 (2012), pp.~113--122.

\bibitem{kressner2013low}
{\sc D.~Kressner, M.~Steinlechner, and B.~Vandereycken}, {\em Low-rank tensor
  completion by riemannian optimization}, BIT Numerical Mathematics,  (2013),
  pp.~1--22.

\bibitem{kurucz2007methods}
{\sc M.~Kurucz, A.~A. Bencz{\'u}r, and K.~Csalog{\'a}ny}, {\em Methods for
  large scale svd with missing values}, in Proceedings of KDD Cup and Workshop,
  vol.~12, Citeseer, 2007, pp.~31--38.

\bibitem{lee2005acquiring}
{\sc K.-C. Lee, J.~Ho, and D.~Kriegman}, {\em Acquiring linear subspaces for
  face recognition under variable lighting}, IEEE Transactions on Pattern Analysis and Machine
  Intelligence, 27 (2005), pp.~684--698.

\bibitem{ling2012decentralized}
{\sc Q.~Ling, Y.~Xu, W.~Yin, and Z.~Wen}, {\em Decentralized low-rank matrix
  completion}, in IEEE
  International Conference on Acoustics, Speech and Signal Processing (ICASSP), (2012), pp.~2925--2928.

\bibitem{liu2013tensor}
{\sc J.~Liu, P.~Musialski, P.~Wonka, and J.~Ye}, {\em Tensor completion for
  estimating missing values in visual data}, IEEE Transactions on Pattern
  Analysis and Machine Intelligence,  (2013), pp.~208--220.

\bibitem{liufactor}
{\sc Y.~Liu, F.~Shang, H.~Cheng, J.~Cheng, and H.~Tong}, {\em Factor matrix
  trace norm minimization for low-rank tensor completion}, Proceedings of the
  2014 SIAM International Conference on Data Mining,  (2014).

\bibitem{mirsky1975trace}
{\sc L.~Mirsky}, {\em A trace inequality of John von Neumann}, Monatshefte
  f{\"u}r mathematik, 79 (1975), pp.~303--306.

\bibitem{morup2008algorithms}
{\sc M.~M{\o}rup, L.~Hansen, and S.~Arnfred}, {\em Algorithms for sparse
  nonnegative {Tucker} decompositions}, Neural computation, 20 (2008),
  pp.~2112--2131.

\bibitem{mu2013square}
{\sc C.~Mu, B.~Huang, J.~Wright, and D.~Goldfarb}, {\em Square deal: Lower
  bounds and improved relaxations for tensor recovery}, arXiv preprint
  arXiv:1307.5870,  (2013).

\bibitem{omberg2007tensor}
{\sc L.~Omberg, G.~H. Golub, and O.~Alter}, {\em A tensor higher-order singular
  value decomposition for integrative analysis of dna microarray data from
  different studies}, Proceedings of the National Academy of Sciences, 104
  (2007), pp.~18371--18376.

\bibitem{paatero1997weighted}
{\sc P.~Paatero}, {\em A weighted non-negative least squares algorithm for
  three-way ‘parafac’factor analysis}, Chemometrics and Intelligent
  Laboratory Systems, 38 (1997), pp.~223--242.

\bibitem{CF2016}
{\sc Z.~Peng, T.~Wu, Y.~Xu, M.~Yan,  and W.~Yin}, {\em Coordinate Friendly Structures, Algorithms and Applications}, Annals of Mathematical Sciences and Applications, 1 (2016), pp.~57--119.

\bibitem{phan2011damped}
{\sc A.~Phan, P.~Tichavsky, and A.~Cichocki}, {\em Damped gauss-newton
  algorithm for nonnegative tucker decomposition}, in IEEE Statistical Signal
  Processing Workshop (SSP), (2011), pp.~665--668.

\bibitem{romera2013new}
{\sc B.~Romera-Paredes and M.~Pontil}, {\em A new convex relaxation for tensor
  completion}, arXiv preprint arXiv:1307.4653,  (2013).

\bibitem{ruhe1974numerical}
{\sc A.~Ruhe}, {\em Numerical computation of principal components when several
  observations are missing}, University of Umea, Institute of Mathematics and
  Statistics Report,  (1974).

\bibitem{savas2007handwritten}
{\sc B.~Savas and L.~Eld{\'e}n}, {\em Handwritten digit classification using
  higher order singular value decomposition}, Pattern recognition, 40 (2007),
  pp.~993--1003.

\bibitem{sorber2013structured}
{\sc L.~Sorber, M.~Van~Barel, and L.~De~Lathauwer}, {\em Structured data
  fusion}, ESAT-STADIUS, KU Leuven, Belgium, Tech. Rep,  (2013), pp.~13--177.

\bibitem{srebro2003weighted}
{\sc N.~Srebro and T.~Jaakkola}, {\em {Weighted low-rank approximations}}, in Machine Learning International Workshop, 20 (2003),
  pp.~720--727.

\bibitem{sun2005cubesvd}
{\sc J.-T. Sun, H.-J. Zeng, H.~Liu, Y.~Lu, and Z.~Chen}, {\em Cubesvd: a novel
  approach to personalized web search}, in Proceedings of the 14th
  international conference on World Wide Web, ACM, 2005, pp.~382--390.

\bibitem{tomasi2005parafac}
{\sc G.~Tomasi and R.~Bro}, {\em Parafac and missing values}, Chemometrics and
  Intelligent Laboratory Systems, 75 (2005), pp.~163--180.

\bibitem{uschmajew2014ALS}
{\sc A.~Uschmajew}, {\em A new convergence proof for the high-order power method
  and generalizations}, Arxiv preprint arXiv:1407.4586,  (2014).

\bibitem{vasilescu2002human}
{\sc M.~A.~O. Vasilescu}, {\em Human motion signatures: Analysis, synthesis,
  recognition}, in Proceedings of IEEE 16th International
  Conference on Pattern Recognition, 3 (2002), pp.~456--460.

\bibitem{vasilescu2002multilinear}
{\sc M.~A.~O. Vasilescu and D.~Terzopoulos}, {\em Multilinear image analysis
  for facial recognition}, in IEEE Computer Society International Conference on Pattern Recognition, 2 (2002), pp.~20511--20511.

\bibitem{walczak2001dealing}
{\sc B.~Walczak and D.~Massart}, {\em Dealing with missing data: Part i},
  Chemometrics and Intelligent Laboratory Systems, 58 (2001), pp.~15--27.

\bibitem{wen2012lmafit}
{\sc Z.~Wen, W.~Yin, and Y.~Zhang}, {\em Solving a low-rank factorization model
  for matrix completion by a nonlinear successive over-relaxation algorithm},
  Mathematical Programming Computation,  (2012), pp.~1--29.

\bibitem{xu2014NTD}
{\sc Y.~Xu}, {\em Alternating proximal gradient method for sparse nonnegative
  tucker decomposition}, Mathematical Programming Computation, 7 (2015),
  pp.~39--70.

\bibitem{tmac2015}
{\sc Y.~Xu, R.~Hao, W.~Yin, and Z.~Su}, {\em Parallel matrix factorization for
  low-rank tensor completion}, Inverse Problems and Imaging, 9 (2015), pp.~601--624.

\bibitem{xu2013block}
{\sc Y.~Xu and W.~Yin}, {\em A block coordinate descent method for regularized
  multiconvex optimization with applications to nonnegative tensor
  factorization and completion}, SIAM Journal on Imaging Sciences, 6 (2013),
  pp.~1758--1789.

\bibitem{admmNMF2012}
{\sc Y.~Xu, W.~Yin, Z.~Wen, and Y.~Zhang}, {\em An alternating direction algorithm for matrix completion with nonnegative factors}, Frontiers of Mathematics in China, 7 (2012), pp.~365--384.

\bibitem{yuan2014tensor}
{\sc M.~Yuan and C.-H. Zhang}, {\em On tensor completion via nuclear norm
  minimization}, Foundations of Computational Mathematics, (2015), pp.~1--38.

\bibitem{zhang2001rank}
{\sc T.~Zhang and G.~H. Golub}, {\em Rank-one approximation to high order
  tensors}, SIAM Journal on Matrix Analysis and Applications, 23 (2001),
  pp.~534--550.

\bibitem{zhao2014bayesian}
{\sc Q.~Zhao, L.~Zhang, and A.~Cichocki}, {\em Bayesian {CP} factorization of
  incomplete tensors with automatic rank determination}, arXiv preprint
  arXiv:1401.6497,  (2014).

\end{thebibliography}
